\documentclass[11pt,reqno]{amsart}

\usepackage{amssymb, amsmath}

\newtheorem{theorem}{Theorem}[subsection]

\newtheorem{cj}[theorem]{Conjecture}
\theoremstyle{definition}

\theoremstyle{remark}

\numberwithin{equation}{section}

\newcommand{\eqn}[1]{(\ref{#1})}

\newcommand{\om}{\omega}

\newcommand{\ZZ}{{\mathbb Z}}
\newcommand{\RR}{{\mathbb R}}

\newcommand{\HH}{{\mathbb H}}
\newcommand{\LL}{{\mathbb L}}

\newcommand{\PP}{{\mathbb P}}
\newcommand{\QQ}{{\mathbb Q}}

\newcommand{\CC}{{\mathbb C}}

\newcommand{\be}{{\bf e}}

\newcommand{\bG}{{\bf G}}

\newcommand{\bv}{{\bf v}}

\newcommand{\sC}{{\mathcal C}}

\newcommand{\sE}{{\mathcal E}}

\newcommand{\sG}{{\mathcal G}}

\newcommand{\EG}{{\delta}}

\newcommand{\sP}{{\mathcal P}}

\newcommand{\sEP}{{\mathcal E}{\mathcal P}}
\newcommand{\sGP}{{\mathcal G}{\mathcal P}}
\newcommand{\sHP}{{\mathcal H}{\mathcal P}}

\newcommand{\blambda}{{\bf \lambda}}
\newcommand{\bE}{{\bf E}}
\newcommand{\Ein}{{\rm Ein}}
\newcommand{\Eone}{{\rm E}_1}
\newcommand{\Li}{\,{\rm Li}}

\newcommand{\meas}{{ \rm meas}}

\renewcommand{\theequation}{\arabic{section}.\arabic{equation}}

\begin{document}

\title{Euler's constant: Euler's work and modern developments}

\author{Jeffrey C. Lagarias}
\address{Dept. of Mathematics,
University of Michigan,
Ann Arbor, MI 48109-1043, USA}
\curraddr{}
\email{lagarias@umich.edu}

\thanks{The research  of the author was supported by NSF Grants DMS-0801029 and DMS-1101373.}


\date{October 10, 2013}

\dedicatory{}

\begin{abstract}
This paper has two parts. The first part  surveys  Euler's work on the
constant $\gamma=0.57721...$ bearing his name, 
together with some of his related work on the gamma function, values of the zeta function,
and divergent series.  
The second part describes various  mathematical developments involving Euler's constant,
as well as another constant,
 the Euler-Gompertz constant. These developments
 include  connections with arithmetic functions
 and the Riemann hypothesis,  and with
 sieve methods,  random permutations,
  and  random matrix products.
It also  includes recent results  on Diophantine
approximation and transcendence related to Euler's constant. 
\end{abstract}

\maketitle
\tableofcontents



%
%
%
%
\setlength{\baselineskip}{1.0\baselineskip}

\section{Introduction}

Euler discovered the constant $\gamma$, defined by
\begin{equation*}
\gamma := \lim_{n \to \infty}  \left( \sum_{j=1}^n \frac{1}{j} \,  - \log n\right) .
\end{equation*}
This is a fundamental constant, now called {\em Euler's constant.}  We give 
it to $50$ decimal places  as
\begin{equation*}
\gamma = 0. 57721~56649~01532 ~86060~ 65120~90082~40243~10421~59335~ 93992...
\end{equation*}
In this paper we shall describe  Euler's  work on it, and its 
connection with values of the gamma function and Riemann zeta function.
We consider as well its close 
relatives $e^{\gamma} =1. 7810724... $
and  $e^{-\gamma}=0.561459...$.
We also inquire as to  its possible meaning, by describing many situations
in mathematics where it occurs.

The constant $\gamma$ is often known as the {\em Euler-Mascheroni constant},
after later work of Mascheroni discussed in Section \ref{sec26}. As explained there,
based on their relative contributions  it seems appropriate to name it after  Euler.

There are many famous unsolved problems about the nature of this
constant. The most well known one is: 


\begin{cj} \label{conj1}
Euler's constant is irrational.
\end{cj}

This is a long-standing open problem. A recent and stronger
version of  it is the following. 


\bigskip
\begin{cj}~\label{conj2}
Euler's constant is not a Kontsevich-Zagier period. In particular,
Euler's constant is transcendental.
\end{cj}

A {\em period} is defined by Kontsevich and Zagier \cite{KZ01}
to be a complex constant whose real and
imaginary parts separately are given as a finite sum of absolutely convergent integrals of
rational functions in any number of variables with rational coefficients, integrated
over a domain cut  out by a finite set of polynomial equalities and inequalities
with rational coefficients, see Section \ref{sec310}.  Many constants are known to be periods, 
 in particular all zeta values $\{\zeta(n): n \ge 2\}.$
The set of all periods forms a ring ${\mathcal P}$ which
includes the field $\overline{\QQ}$
of all algebraic numbers. It follows that if Euler's constant is not a period, then it must
be   a transcendenta numberl.   Conjecture \ref{conj2}  also implies  that $\gamma$ would be
$\overline{\QQ}$-linearly independent of all odd zeta values
$\zeta(3), \zeta(5), \zeta(7), ...$ and of $\pi$.

%
This paper  also presents  results on  another constant 
defined by Euler,  the {\em Euler-Gompertz constant}
\begin{equation}
\label{104aa}
\EG := \int_{0}^1 \frac{dv}{1-\log v} = \int_{0}^{\infty} \frac{e^{-t}}{1+t} dt  = 0.59634 \,73623\, 23194 \dots
\end{equation}
Here we adopt the notation of Aptekarev \cite{Apt09} for this constant (in which $\EG$
follows $\gamma$) and refer to it by the name given  by
Finch \cite[Sec. 6.2]{Fin03} (who denotes it $C_2$). Some results about $\delta$  intertwine  directly with
Euler's constant, see  Sections \ref{sec34},  \ref{sec35},  \ref{sec311} and  \ref{sec312}.

Euler's name  is also associated  with other  constants such as  $e=2.71828...$ given by
\begin{equation*}
e := \sum_{n=0}^{\infty}\,\frac{1}{n!}.
\end{equation*}
He did not discover this constant, but  did standardize  the symbol ``$e$'' for its use.
He used the notation $e$  in 
 his 1737 essay on continued fractions (\cite[p. 120]{E71}, cf. \cite{EWW85}),
where he determined its continued fraction expansion 
$$
e = [2; 1, 2,\, 1, 1, 4,\, 1,1, 6, \,1, 1, 8,\, \cdots ]= 2 + \cfrac{1}{1+ \cfrac{1}{2+ \cfrac{1}{1+ \cfrac{1}{1+ \cdots}}}}
$$
and from its form deduced that it was irrational. 
Elsewhere he   found the famous relation $e^{\pi i}= -1$. 
The constant  $e$ is also conjectured not to be a Kontsevich-Zagier period. 
It is well known   to  be a transcendental number, consistent
with this conjecture.

In this paper we review Euler's work on the constant $\gamma$ and  on mathematical
topics connected to it and  survey subsequent developments. 
It is correspondingly divided into two parts, 
 which may be read independently.
 We have also aimed to make  individual subsections  readable out of order.

The first part  of this paper considers  the work of Euler. Its emphasis is historical, and
it  retains some of the original notation of Euler.  It is a tour of part of
Euler's work related to zeta values and other  constants and methods of finding
relations between them.
Euler did  extensive work  related to his constant: he returned
to its study many  times. The basic fact
to take away is that Euler did an enormous amount of work directly on his
constant and related areas, far more than is commonly appreciated.

The second part of this paper  addresses mathematical developments 
 made since Euler's time concerning.
Euler's constant. 
Since Euler's constant is an unusual constant that seems unrelated
to other known constants, its appearance in different mathematical subjects
can be a signal of possible  connections between these subjects.
We  present  many  contexts  in which Euler's constant appears.
These include  connections to the Riemann hypothesis, to  random matrix theory,
to probability theory and to other subjects.  
There is a set of strong analogies between   factorizations of
random integers in the interval $[1, n]$ and  cycle structures of random permutations
in the symmetric group $S_N$,  taking  $N$ to be proportional to $\log n$.
In this analogy there appears another
constant that might have appealed to Euler: the {\em Golomb-Dickman constant},
which can be defined as 
$$
\lambda := \int_{0}^{1} e^{Li(x)} dx = 0.62432 \, 99885 \, 43550 \dots
$$
where $Li(x) := \int_{0}^x \frac{dt}{\log t}$,
see Section \ref{sec38}. The constant $\lambda$ appears in connection with the
distribution of the longest cycle of a random permutation, while the constant
$e^{-\gamma}$ appears in connection with the distribution of the shortest
cycle. The  discovery of this analogy  is described at the beginning of Section \ref{sec38a}.
It was found by  Knuth and Trabb-Pardo \cite{KTP76}
 through noticing a numerical coincidence to $10$ decimal places
 between two computed constants,
  both subsequently identified with  the Golomb-Dickman constant.

In passing we  present many 
striking  and elegant formulas, which exert a certain fascination in
themselves. Some of the
 formulas presented were obtained  by specialization
of more general  results in the literature.
Other works presenting striking  formulas  include
the  popular book of J. Havil \cite{Hav03}  devoted extensively to
mathematics around Euler's constant,  and   
Finch \cite[Sections 1.5 and 4.5]{Fin03}. 

In Section \ref{sec4} we make concluding remarks and briefly  describe
other directions of work related to Euler's constant.

 \bigskip

\renewcommand{\theequation}{\arabic{section}.\arabic{subsection}.\arabic{equation}}
%
%
%
%
\section{Euler's work}\label{sec2}
\setcounter{equation}{0}

Euler (1707--1783) was born in Basel, and entered the University
of Basel at age 13, to study theology. He also
 received private tutorials  in mathematics   from   Johann Bernoulli (1667--1748),
 who was Professor of Mathematics at the University of Basel. 
 Johann recognized his mathematical gift and
convinced his family to let him switch to mathematics. 
Euler was excellent at computation and had  an essentially photographic memory. 
He had great persistence, and returned over and over again to problems that
challenged him.

Here we 
only scratch the surface of Euler's enormous output,
which contains over 800 works.
 These include at least   $20$ books,
surveyed in  Gautschi \cite{Ga08b}.
We trace a   thread  comprising results
corresponding to one of 
his  interests,  the determination of specific constants
given by infinite sums, definite integrals,  infinite products,
and  divergent series, 
and  finding  identities giving  interrelations among  these constants.
Ancillary to this  topic is the related problem of determining numerical values
of these constants;  computing such values allow one to  guess and test possible
identities.  

In his work Euler made extensive calculations, deriving approximations of numbers to
many decimal places. He often reported  results of calculations 
obtained using various approximation schemes
and sometimes compared the  values he obtained by the different methods. 
In Euler's numerical results reported below
 the symbol $\approx$ is used to mean  ``the result of an approximate calculation''. 
 The symbol $\approx$  does not imply agreement to the number of digits provided.
 An underline of a digit in such an expansion indicates the first place where it disagrees with the
 known decimal expansion.

We refer to Euler's  papers by their   index numbers
in the En\"{e}strom catalogue (\cite{Ene13}).
These papers may be found online in the Euler archive (\cite{EA}). For discussions of Euler's early
work up to 1750 see Sandifer \cite{San07b}, \cite{San07}.
For  his work on the Euler(-Maclaurin) summation formula see Hofmann \cite{Hof57} and Knoebel et al.
\cite[Chapter 1]{KLLP07}.
For other discussions of his work on the zeta function  see
St\"{a}ckel \cite{Sta07}, Ayoub \cite{Ay74}, 
Weil \cite[Chap. 3] {We84} and Varadarajan \cite[Chap. 3]{Var08}. For Euler's life and work 
see Calinger \cite{Cal07} and other articles in Bradley and Sandifer \cite{BS07}.

\subsection{Background}\label{sec20}
\setcounter{equation}{0}

Euler spent much effort evaluating the sums of reciprocal powers
\begin{equation}\label{200b}
\zeta(m) := \sum_{j=1}^{\infty} \frac{1}{j^m}.
\end{equation}
for integer $m \ge 2$. 
He  obtained the formula $\zeta(2) = \frac{\pi^2}{6}$, solving the ``Basel problem,''
as well as formulas for all even zeta values $\zeta(2n).$ 
He repeatedly tried to find analogous
formulas for the odd zeta values $\zeta(2n+1)$ which would
explain the nature of these numbers. Here he did not completely succeed,
and the nature of the odd zeta values remains
unresolved to this day. 
Euler's constant naturally arises in this context, in
attempting to assign a value to  $\zeta(1)$.

Euler discovered the constant $\gamma$ in a  study of the {\em harmonic numbers}
\begin{equation}\label{200c}
H_n := \sum_{j=1}^n \frac{1}{j},
\end{equation}
which  are the partial sums of the harmonic series. The series
$``\zeta(1)" = \sum_{n=1}^{\infty} \frac{1}{n}$
diverges, and Euler  approached it by (successfully) 
 finding a function $f(z)$ (of  a real or complex variable $z$) that interpolates the harmonic number
values, i.e. it has  $f(n)=H_n$ for integer $n \ge 1$. 
The resulting function, denoted $H_{z}$ below,  is related to the gamma function
via \eqn{201c}.
In studying it he also studied, for integer $m$, the related sums 
\begin{equation}\label{200d}
H_{n,m}:= \sum_{j=1}^n \,\frac{1}{j^m},
\end{equation}
which we will term {\em $m$-harmonic numbers}; in this notation $\, H_{n,1} = H_n$.

Euler obtained many inter-related  results on his constant $\gamma$.
These are intertwined with his work on  the
gamma function and on values of the (Riemann) zeta function. In Sections 2.2 - 2.4
we present results that he obtained for Euler's constant,
for the gamma function and its logarithmic derivative, and
for  values of the zeta function at integer arguments, concentrating on explicit formulas.
In Section 2.5 we  discuss a paper  of Euler on how to  sum divergent series,
applied to the example $0!-1!+2!- 3!+ \cdots$   His analysis  produced
a new constant,  the Euler-Gompertz constant  $\delta$, given above
in \eqn{104aa}. In Section 2.6 we review the history of the name   Euler-Mascheroni constant,
the adoption of the notation $\gamma$ for it, and  summarize Euler's approach
to research.

Taken together Euler's many papers on his constant $\gamma$, on the related 
 values $\zeta(n)$ for $n \ge 2$,  and other values show his great interest in individual constants 
 (``notable numbers")  and on 
 their arithmetic interrelations.

\subsection{Harmonic series and Euler's constant} \label{sec21}
\setcounter{equation}{0}

One strand of Euler's early work involved 
finding continuous interpolations of various discrete sequences. 
For the sequence of factorials this led to the gamma function, discussed below. 
In  the 1729 paper \cite[E20]{E20} he notes 
that the  
harmonic numbers $H_n= \sum_{k=1}^n \frac{1}{k}$ are given by the integrals
\begin{equation*}
H_n:= \int_{0}^1 \frac{1- x^n}{1-x} dx.
\end{equation*}
He proposes a function interpolating  $H_n$   by treating $n$ as a continuous
variable in this integral.   The resulting interpolating function, valid for real $z \ge 0$,  is
\begin{equation}\label{201aa}
H_{z} := \int_{0}^1 \frac{1- x^z}{1-x} dx.
\end{equation}
Using this definition, he derives the formula
\begin{equation*}
H_{\frac{1}{2}}= 2- 2 \log 2.
\end{equation*}
The function $H_z$ is related to the 
digamma function $\psi(z) = \frac{d}{dz} \log \Gamma(z)$
by 
\begin{equation}\label{201c}
H_{z} = \psi(z+1)+ \gamma,
\end{equation}
see Section \ref{sec22}  and also Theorem~\ref{th30}.

In  a paper written in 1731  ("{\em On harmonic progresssions}") Euler \cite[E43]{E43} summed the harmonic series in terms
of zeta values, as follows,  and computed Euler's constant
to $5$ decimal places. 

\begin{theorem} \label{th21}
{\rm (Euler 1731)}
The  limit
$$
\gamma= \lim_{n \to \infty} \left(H_n - \log n \right)
$$
exists. It is given by the (conditionally) convergent series
\begin{equation}\label{202}
\gamma = \sum_{n=2}^{\infty} (-1)^n \frac{\zeta(n)}{n}.
\end{equation}
\end{theorem}

Euler explicitly observes that this series \eqn{202} converges (conditionally) since it is an alternating series with
decreasing terms.  He reports that  the constant  $\gamma \approx 0.57721\underline{8}$.

Euler obtains the formula \eqn{202}
using the expansion 
$\log (1+x) = \sum_{k=1}^{\infty} (-1)^{k+1} \frac{x^k}{k},$
which was found by Nicholas Mercator (c.1620--1687) \cite{Mercator1668} in 1668.
Evaluating this formula at $x= 1, \frac{1}{2}, ..., \frac{1}{n}$ Euler observes the following in tabular form:
\begin{eqnarray*}
\log 2 &=& 1 - ~\frac{1}{2} \cdot \Big(\frac{1}{1}\Big)^2 + \frac{1}{3} \cdot \Big(\frac{1}{1}\Big)^3 - ...\\
\log \frac{3}{2} &=& \frac{1}{2} - \frac{1}{2}\cdot \Big(\frac{1}{2}\Big)^2 + \frac{1}{3} \cdot \Big(\frac{1}{2}\Big)^3 + ...\\
\log \frac{4}{3} &=& \frac{1}{3} -  \frac{1}{2}\cdot \Big(\frac{1}{3}\Big)^2 + \frac{1}{3} \cdot \Big(\frac{1}{3}\Big)^3 + ..\\
\log \frac{5}{4} &=& \frac{1}{4} -  \frac{1}{2}\cdot \Big(\frac{1}{4}\Big)^2 + \frac{1}{4} \cdot \Big(\frac{1}{4}\Big)^3 + ..\\
\end{eqnarray*}
Summing by columns the first $n$ terms, he obtains
\begin{equation}\label{203}
\log (n+1) = H_n - \frac{1}{2} H_{n,2} + \frac{1}{3} H_{n,3} - \cdots
\end{equation}
One may rewrite  this as
$$
H_n - \log (n+1) = \frac{1}{2} H_{n,2} - \frac{1}{3} H_{n,3} + \cdots
$$
Now one observes that for $j \ge 2$,
$$
\lim_{n \to \infty} H_{n,j} = H_{\infty,j} = \zeta (j)
$$
Taking this limit as $n \to \infty$ term by term in \eqn{203} formally gives the result \eqn{202}.

We note that such representations of infinite sums in a square array,
to be summed in two directions,  were already used repeatedly by Johann Bernoulli's
older brother  Jacob Bernoulli (1654--1705) in his 1689 book
on summing infinite series (\cite{Bernoulli1689}).
Another  nice example of an identity found by  array summation is
\begin{equation}\label{220}
(\zeta(2)-1) + (\zeta(3)-1)+ (\zeta(4)-1) + \cdots = 1,
\end{equation}
using
\begin{eqnarray*}
\zeta(2) -1 &=& \frac{1}{2^2} + \frac{1}{3^2} + \frac{1}{4^2} + \frac{1}{5^2} + \cdots \\
\zeta(3) -1 &=& \frac{1}{2^3} + \frac{1}{3^3} + \frac{1}{4^3} + \frac{1}{5^3} + \cdots \\
\zeta(4) -1 &=& \frac{1}{2^4} + \frac{1}{3^4} + \frac{1}{4^4} + \frac{1}{5^4} + \cdots \\
\end{eqnarray*}
since the column sums yield the telescoping series 
$$
\frac{1}{2^2} \times 2 + \frac{1}{3^2} \times \frac{3}{2} + \frac{1}{4^2} \times \frac{4}{3}+ \cdots
= \frac{1}{1\cdot 2} + \frac{1}{2 \cdot 3} + \frac{1}{3 \cdot 4} + \cdots = 1.
$$
Compare \eqn{220}  to Euler's expression \eqn{228} below for Euler's constant.

In  1732 Euler \cite[E25]{E25} stated his first form of his  summation formula,
without a proof. He then gives it  in detail in his 1735 paper \cite[E47]{E47}, published 1741.
Letting $s$ denote the sum of the function $t(n)$ over some set of integer values (say $n=1$ to $N$),
he writes
\begin{equation*}
s= \int t dn + \alpha t + \beta\frac{dt}{dn} + \gamma \frac{d^2 t}{dn^2} + \delta \frac{d^3 t}{dn^{3}} + \mbox{etc.}.
\end{equation*}
where $\alpha= \frac{1}{2}$, 
$\beta = \frac{\alpha}{2} - \frac{1}{6}$, 
$\gamma= \frac{\beta}{2} - \frac{\alpha}{6} + \frac{1}{24}$, $\delta= \frac{\gamma}{2} -\frac{\beta}{6} + \frac{\alpha}{24} - \frac{1}{120}.$
etc. These linear equations for the coefficients
are solvable to give $\alpha=\frac{1}{2}, \beta= \frac{1}{12}, \gamma= 0, \delta= -\frac{1}{720}$ etc.
Thus
$$
s= \int t dn + \frac{t}{2} + \frac{1}{12} \frac{dt}{dn} - \frac{1}{720}  \frac{d^3 t}{dn^{3}} + \mbox{etc}.
$$
The terms on the right must be evaluated at both endpoints of this interval being summed.
In modern terms, 
$$
\sum_{n=1}^N t(n) = \int_{0}^N t(n)\, dn + \frac{1}{2}\Big(t(N)- t(0)\Big) + \frac{B_2}{2!}\Big( \frac{d\,t(N)}{dn} - \frac{d\,t(0)}{dn}\Big)
- \frac{B_4}{4!}\Big( \frac{d^3\, t(N)}{dn^3} - \frac{d^3\, t(0)}{dn^3} \Big) + \cdots
$$
where the $B_{2k}$ are Bernoulli numbers,  defined in \eqn{221} below,
and the sum must be cut off at a finite point with a remainder term (cf. \cite[Sec. 1.0]{Ten95}).
Euler's formulas  omit any remainder term, and he applies them by cutting
 off the right side sum at a finite place, and then using the computed value of the
 truncated right side sum as an
approximation to the sum on the left,  in effect making the  assumption that the associated remainder term is small.

A similar summation formula was obtained independently by Colin Maclaurin (1698--1746),
who presented it in his 1742  two-volume work on Fluxions \cite{Mac1742}.  
A comparison of their formulas is made by Mills \cite{Mil85}, 
who observes  they are  substantially identical,
and that neither author keeps track of the remainder term.
It is appropriate that this formula, now with the remainder term included,   
is   called the {\em Euler-Maclaurin summation formula.}

 Jacob Bernoulli 
 introduced the  numbers now bearing his name
  (with no subscripts), in {\em Ars Conjectandi} \cite{Bernoulli1713}, published posthumously
in 1713. They arose  in evaluating (in modern notation) the sums of $n$-th powers
$$
\sum_{k=1}^{m-1}   k^n=  \frac{1}{n+1} 
\left(\sum_{k=0}^n \left( {{n+1}\atop{k}}\right) B_k \, m^{n+1-k}. \right)
$$
We  follow the convention  that  the {\em Bernoulli numbers} $B_n$ are  those given  by the expansion
\begin{equation}\label{221}
\frac{t}{e^t - 1} = \sum_{n=0}^{\infty} \frac{B_n}{n!} t^n,
\end{equation}
so that  
\[
B_1= -\frac{1}{2},~ B_2= \frac{1}{6}, ~B_3=0, ~B_4 = -\frac{1}{30}, ~B_5=0, ~B_6= \frac{1}{42}, ...
\]  
Euler  later named the Bernoulli numbers 
 in his 1768 paper (in \cite[E393]{E393}) (``{\em On the sum of series involving the Bernoulli numbers}"),
 and there  listed the first few as $\frac{1}{6}, \frac{1}{30}, \frac{1}{42}, \frac{1}{30}, \frac{5}{66}.$
 His definition led to the  convention\footnote{However in \cite{E393} and elsewhere Euler
does not use subscripts, writing ${\mathfrak A}, {\mathfrak B}, {\mathfrak C}$ etc.}  that the notation $B_n$ corresponds 
  to $|B_{2n}|$ in \eqn{221}. 
  This convention  
   is  followed in the classic text of 
Whittaker and Watson \cite[p. 125]{WW63}.

In his 1755 book on differential calculus \cite[E212, Part II, Chap. 5 and 6]{E212}
Euler  presented his summation formula in detail, with many worked examples.
(See Pengelley \cite{Pe00},  \cite{Pe07} for a  translation 
and for a detailed discussion of this work.)
 In Sect. 122 Euler computes the Bernoulli numbers  $|B_{2n}|$ through $n=15$. 
In Sect. 129 he remarks that:

\begin{quote}
The [Bernoulli numbers] form a highly divergent sequence, which
grows more strongly than any geometric series of growing terms.
\footnote{``pariter ac Bernoulliani ${{\mathfrak A}, \mathfrak B}, {\mathfrak C}, {\mathfrak D}$ \&c.
serium maxime divergentem, quae etiam magis increscat, quam ulla series geometrica terminis
crescentibus procedens.'' [Translation: David Pengelley \cite{Pe00}, \cite{Pe07}.]}
\end{quote}

\noindent In Sect. 142 he  considers 
the  series for the harmonic numbers $s=H_{x}= \sum_{n=1}^x \frac{1}{n}$, 
where his summation formula
reads  (in modern notation)
\begin{equation}\label{222}
s  = \log x + \left( \frac{1}{2x}- \frac{|B_2|}{2x^2} + \frac{|B_4|}{4x^4} - \frac{|B_6|}{6x^6} +  \cdots \right) + C,
\end{equation}
where the constant $C$ to be determined is Euler's constant. 
He remarks that this series is divergent. He substitutes $x=1$, stating that  $s=1$ in \eqn{222}, obtaining formally
\[ 
C = \frac{1}{2} +  \frac{|B_2|}{2} - \frac{|B_4|}{4} + \frac{|B_6|}{6} - \frac{|B_8|}{8} + \cdots
\]
although the right side diverges.
Substituting this  expression in \eqn{222} yields
\begin{eqnarray*}
s  &=& \log x + \left( \frac{1}{2x}- \frac{|B_2|}{2x^2} + \frac{|B_4|}{4x^4} - \frac{|B_6|}{6x^6} + \frac{|B_8|}{8x^8}-  \cdots \right)  \\
&& ~\quad\quad + \left( \frac{1}{2}\, + \frac{|B_2|}{2} - \frac{|B_4|}{4} + \frac{|B_6|}{6} - \frac{|B_8|}{8} + \cdots\right).
\end{eqnarray*}
In Sect. 143 he sets $x=10$ in \eqn{222}  so that    $s= H_{10}$,  
and gives the numerical value $s=2.92896~82539~68253~968$. The right side contains $C$
as an unknown, and by evaluating the other terms (truncating at a suitable term) he finds 
$$
C \approx 0.57721~56649~01532~\underline{5}
$$
a result accurate to $15$ places. 
We note that the divergent series on the right side of \eqn{222}  is  the
asymptotic expansion of the digamma function $\psi(x+1) +\gamma$,
where $\psi(x)= \frac{\Gamma'(x)}{\Gamma(x)}$, see Theorem~\ref{th32}, and
 in effect Euler truncates the asymptotic expansion
at a suitable term
to  obtain a good numerical approximation.

  Over his lifetime Euler obtained
many different numerical approximations to Euler's constant.
He calculated approximations
 using truncated divergent series, and
 his papers report the value obtained, 
 sometimes with a comment on accuracy, sometimes not.
In contrast, some of the convergent series he obtained for Euler's constant
converge slowly and proved of little value for numerics.  He  also explored various
series acceleration methods, which he applied to both convergent
and divergent series.
He  compared and cross-checked his work  with
various different numerical methods.
In this way he obtained confidence as to how many
digits of accuracy he had obtained using the various methods, and also obtained information on the reliability
of his calculations using divergent series.

In 1765 in a paper  studying the  gamma function \cite[E368]{E368}
(which will be discussed in more detail in  Section 2.3)  he  derives  the formula 
$\Gamma'(1) = - \gamma.$
It is given in Section 14 of his paper as  equation \eqn{273} in Section 2.3, which is   equivalent to the integral formula
\begin{equation}\label{225c}
-\gamma = \int_{0}^{\infty} e^{-x} \log x \, dx.
\end{equation}
This result follows  from differentiation
under the integral sign of  Euler's integral \eqn{272} in Section \ref{sec23}.

In a 1768 paper  mainly devoted to relating
Bernoulli numbers to values of  $\zeta(n)$,  Euler  \cite[E393]{E393}
obtains more formulas for  $\gamma$ (denoting it $O$).
In Section 24 he reports the  evaluation of $\gamma$
found earlier and remarks: 

\begin{quote}
$O= 0,57721~56649~01532~5$.  This number seems also the more noteworthy because
even though I have spent much effort in investigating it,
I have not been able to reduce it to a known kind of quantity.
\footnote{``$O=0.5772156649015325$ qui numerus eo maiori 
attentione dignus videtur, quod eum, cum olim in hac
investigatione multum studii consumsissem, nullo modo ad cognitum quantitatum genus
reducere valui.'' [Translation: Jordan Bell]}
\end{quote}
 
\noindent He derives in Section  25 the integral formula
\begin{equation}\label{226a}
\gamma = \int_{0}^{\infty} \left( \frac{e^{-y}}{1-e^{-y}} - \frac{e^{-y}}{y} \right) dy.
\end{equation}
Here we may note that the integrand
\[
 \frac{e^{-y}}{1-e^{-y}} - \frac{e^{-y}}{y}= \frac{1}{y}\Big( \frac{y}{e^y -1} - e^{-y} \Big)= 
 \sum_{n=1}^{\infty} \Big(B_n - (-1)^n \Big)\frac{y^{n-1}}{n!}.
\]
In Section 26 by change of variables he obtains
\begin{equation}\label{227a}
\gamma = \int_{0}^{1} \left(  \frac{1}{1-z}+ \frac{1}{\log z} \right) dz
\end{equation}
In Section 27   he gives the following new formula for  Euler's constant
\begin{equation}\label{228}
\gamma= \frac{1}{2} (\zeta(2)-1) + \frac{2}{3}(\zeta(3) -1) + \frac{3}{4}(\zeta(4) -1) + \cdots
\end{equation}
and in Section 28 he obtains a formula in terms of $\log 2$ and the odd zeta values:
\begin{equation*}
\gamma = \frac{3}{4} - \frac{1}{2}\log 2 + \sum_{k=1}^{\infty} \Big( 1-\frac{1}{2k+1} \Big)\Big(\zeta(2k+1)-1\Big)
\end{equation*}
He concludes in Section 29: 

\begin{quote}
Therefore the question remains of great moment, of what character the number $O$ is
and among  what species of quantities it can be classified.\footnote{``Manet ergo quaestio magni momenti, cujusdam indolis sit numerus iste $O [:=\gamma]$
et ad quodnam genus quantitatum sit referendus.''[Translation: Jordan Bell]}
\end{quote}

In a 1776 paper \cite[E583]{E583}
({\em On a  memorable number naturally occurring in the summation of the harmonic series}),
\footnote{De numero memorabili in summatione progressionis harmonicae naturalis occurrente}
  Euler  
studied the constant $\gamma$  in its own right. This paper was not published until 1785, after his death.
In Section 2 he
speculates that the number $N= e^{\gamma}$  should
be a notable number: 

\begin{quote}
 And therefore, if $x$ is taken to be an infinitely large number it will then be
 $$
  1+ \frac{1}{2} + \frac{1}{3} + \cdots + \frac{1}{x} = C + \log x.
 $$
One may suspect from this that the number $C$ is the hyperbolic [= natural]
logarithm of some notable number, which we put $=N$, so that $C= \log N$
and the sum of the infinite series is equal to the logarithm of the number $N\cdot x$.
Thus it will be worthwhile to inquire into the value of this number $N$, which indeed
it suffices to have defined to five or six decimal figures, since then one will be
able to judge without difficulty whether this agrees with any known number or 
not.\footnote{``Quod si ergo numerus $x$ accipiatur magnus, tum erit
 $$
 1+ \frac{1}{2} + \frac{1}{3} + \cdots + \frac{1}{z} = C + lx
 $$
 istum numerum $C [:= \gamma]$ esse logarithmum hyperbolicum cuiuspiam numeri
notabilis, quem statuamus $=N$, ita ut sit $C= \log N$, et summa illius seriei infinitae
aequeture logarithmo numeri $Nx$, under operae pretium erit in valorem huius numeri $N$ 
inquirerre, quem quidem sufficiet ad quinque vel fex figuras decimales definiuisse, quoniam hinc
non difficulter iduicari poterit, num cum quopiam numero cogito conueniat nec ne.''
[Translation: Jordan Bell]}
\end{quote}

\noindent
Euler  evaluates  $N = 1.7810\underline{6}$ and says  he cannot connect it to any number he knows.
He then gives
several new formulae for  $\gamma$, and tests their usefulness for calculating its value, comparing
the values obtained numerically  with the value $0.57721~56649~01532~5$ that he previously obtained.
His first new formula in Sect. 6 is
\begin{equation*}
1- \gamma = \frac{1}{2}(\zeta(2) -1) + \frac{1}{3} (\zeta(3)-1) + \frac{1}{4}(\zeta(4)-1) + \cdots\
\end{equation*}
Using the first $16$ terms of this expansion, he finds $\gamma \approx  0,57721~\underline{6}9$.
He finds several other formulas, including,  in Sect. 15,  the formula
\begin{equation*}
1 - \log \frac{3}{2} - \gamma= \frac{1}{3\cdot 2^2} (\zeta(3)-1) + \frac{1}{5 \cdot 2^4}(\zeta(5)-1) +
 \frac{1}{7\cdot 2^6} (\zeta(7) -1) + \cdots,
\end{equation*}
from which he finds $\gamma \approx 0.57721~56649~01\underline{7}91~3$,
a result accurate to $12$ places.
He concludes the paper with a  list of eight formulas relating  Euler's constant,
$\log 2$, and even and odd zeta-values. One striking formula (number VII), involving 
the odd zeta values  is 
\begin{equation*}
\gamma= \log 2 - \sum_{k=1}^{\infty} \frac{\zeta(2k+1)}{(2k+1)2^{2k}}.
\end{equation*}

In another 1776 paper  \cite[E629]{E629} ({\em The expansion of the integral formula 
$\int \partial x( \frac{1}{1-x} +\frac{1}{\ell x})$ with the
term extended from $x=0$ to $x=1$}), which was 
not published until 1789,  Euler
further  investigated the   integral formula \eqn{227a} 
for his constant.
He denotes its value by $n$, and says:

\begin{quote}
the number $n$,  [...] 
which is  found to be approximately $ 0, 57721~56649~01532~5$, whose value I have
been able in no way to reduce to already known transcendental measures;
therefore, it will hardly be useless to try to resolve the formula in many different
ways.\footnote{``numerus $n$ [...], et quem per approximationem olim inveni esse
$=0, 5772156649015325$, cuius valorem nullo adhuc modo ad mensuras
transcendentes iam cogitas redigere potui; unde haud inutile erit resolutionem huius
formulae propositae pluribus modis tentare.''[Translation: Jordan Bell]}
\end{quote}

\noindent He then finds several expansions for   related integrals. One of these starts from
the identity, valid for $m, n \ge 0$, 
\begin{equation*}
\int_{0}^1 \frac{x^m -x^n}{\log x}dx = \log \left( \frac{m+1}{n+1}\right) ,
\end{equation*}
and from this he deduces, for $n \ge 1$, that 
\begin{equation*}
\int_{0}^{1} \frac{(1-x)^n}{\log x} dx = \log \left( \prod_{j=0}^{n}  
(j+1)^{ (-1)^{j} \left({{n}\atop{j}}\right)}\right).
\end{equation*}
This formula gives the moments of the measure
$\frac{x}{\log(1-x)} dx$
on $[0,1]$.

\subsection{The gamma function} \label{sec22}
\setcounter{equation}{0}
Euler was not the originator of work on the factorial function.
This  function  was initially studied 
by John Wallis (1616--1703), followed by   Abraham de Moivre  (1667--1754)
and James Stirling (1692--1770). 
Their  work,  together with  work  of 
 Euler, Daniel Bernoulli (1700--1782) and Lagrange (1736--1813)  is reviewed in Dutka \cite{Dut91}.

In 1729, in correspondence with Christian Goldbach, Euler discussed 
 the problem
of obtaining continuous interpolations of discrete series.  He described
a (convergent)  infinite product interpolating
the factorial function, $m!$, 
\begin{equation*}
\frac{1 \cdot 2^m}{1+m} \cdot \frac{2^{1-m}\cdot 3^m}{2+m}\cdot
 \frac{3^{1-m}\cdot 4^m}{3+m} \cdot \frac{4^{1-m}\cdot 5^m}{4+m} \cdot etc.
\end{equation*}
This infinite product, grouped as shown, converges absolutely on the region $\CC \smallsetminus \{ -1, -2, ...\}$,
to the gamma function, which we denote  $\Gamma(m+1)$, following the later notation of Legendre.
John Wallis (1616--1703) had earlier called this  function  the
``hypergeometric series".
Euler substitutes $m=\frac{1}{2}$ and finds 
$$
\Gamma \Big(\frac{3}{2}\Big)=\frac{1}{2} \sqrt{\pi}, 
$$
by recognizing a relation of this infinite product to Wallis's infinite product for $\frac{\pi}{2}$.
In this letter he also introduced the notion of fractional integration, see Sandifer  \cite{San07}. 

Also in 1729 Euler \cite[E19]{E19} wrote 
a  paper presenting  these results
( {\em On transcendental progressions, that is, those whose general term cannot be given algebraically}),
but it was not published until 1738.
 This paper describes the infinite products, suitable
for interpolation. It then  considers {\em Euler's first integral}\footnote{Here the letter
$e$ denotes an integer. Only in  a later paper, [E71], written in 1734,  does Euler use the symbol  
$e$ to mean the constant $2.71828...$.}
\begin{equation*}
 \int_{0}^1 x^e (1-x)^n dx
\end{equation*}
which in modern terms defines the {\em Beta function} $B(e+1, n+1)$, derives
a recurrence relation from it, and in Sect. 9 derives
$B(e+1, n+1) = \frac{e! n!}{(e+n+1)!}$. More generally, in the later notation of
Legendre this extends to the Beta function  identity
\begin{equation}\label{250c}
B(r, s) : =  \int_{0}^1 x^{r-1} (1-x)^{s-1} dx = \frac{\Gamma(r) \Gamma(s)}{\Gamma( r+s)}.
\end{equation}
relating it to  the gamma function.
Euler \cite[Sect. 14]{E19}  goes on 
to obtain a  continuous interpolation of the factorial function
$f(n) = n!$, given by the  integral 
\begin{equation}\label{251}
f(n) = \int_{0}^{1} ( - \log x)^n dx,
\end{equation}
where $n$ is to be interpreted as a continuous variable, $n >-1$.
He then computes formulas for   $\Gamma \left( 1+ \frac{n}{2} \right)$
at half-integer values  \cite[Sect. 20]{E19}.

In 1765, but not published until 1769,  Euler \cite[E368]{E368} 
(``{\em  On a hypergeometric curve expressed by the equation $y=1* 2 * 3* \dots *x$}")   studied the  ``hypergeometric curve"   given
(using Legendre's notation)  by $y = \Gamma(x+1)$, where 
$$
\Gamma(x) := \int_{0}^{\infty} t^{x-1} e^{-t}dt.
$$
Euler interpolates the factorial function by an infinite product.  In Sect. 5 he gives
the  product formula
\begin{equation*}
y= a^x \prod_{k=1}^{\infty}\left( \frac{k}{k+x}\left(\frac{a+k}{a+k-1}\right)^x\right)
\end{equation*}
with any fixed $a$ with $1<a<x$, and with $a= \frac{1+x}{2}$ obtains
$$
y =  \left(\frac{1+x}{2}\right)^x\prod_{k=1}^{\infty} \left(\frac{k}{k+x} \left(\frac{2k+1 +x}{2k-1+x}\right)^x\right)
$$
In Sect. 9 he derives a formula, which in modern notation reads
\begin{equation}\label{271}
 \log y = -\gamma x + \sum_{n=2}^{\infty} (-1)^n \frac{\zeta(n)}{n} x^n,
\end{equation}
where $\gamma$ is Euler's constant.
In Sect. 11 he makes an exponential change of variables from his formula \eqn{251} to obtain
the integral formula
\begin{equation}\label{272}
y= \int_{0}^{\infty} e^{-v} v^x dv,
\end{equation}
which is (essentially) the standard integral representation for the gamma function.
This formula shows that  the ``hypergeometric curve" 
 is indeed given as\footnote{If one instead used the notation $\Pi(x)$ for
this integral used by Gauss, cf. Edwards \cite[p. 8]{Ed74}, then the hypergeometric curve is $y= \Pi(x).$
Gauss made an extensive study of hypergeometric functions.}
\begin{equation*}
y =  \Gamma( x+1).
\end{equation*}
In Sect. 12 he obtains a formula for the digamma function $\psi(x) = \frac{\Gamma'}{\Gamma}(x)$,
which is
\begin{equation*}
\frac{\Gamma'(x+1)}{\Gamma(x+1)} := \frac{dy}{y\,dx} = -\gamma + \sum_{k=1}^{\infty} \frac{x}{k(x+k)}.
\end{equation*}
In formula VII in this section he gives the asymptotic expansion of $\log \Gamma(x+1)$,
(Stirling's series), expressed  as 
\begin{eqnarray}\label{stirling}
\quad  \log y  &= &  \frac{1}{2}\log 2 \pi +  (x +\frac{1}{2}) \log x \nonumber\\
\\
&&  -  \, x + \frac{A}{2x} - \frac{1\cdot 2}{2^3} \frac{B}{x^3}
+ \frac{1 \cdot 2 \cdot 3\cdot 4}{2^5} \frac{C}{x^5} -
 \frac{ 1\cdot 2 \cdot 3 \cdot 4\cdot 5 \cdot 6}{2^7} \frac{D}{x^7}+ \mbox{etc.} \nonumber
\end{eqnarray}
with $A= \frac{1}{6}, $ $B= \frac{1}{90}$, $C= \frac{1}{945}$, $D= \frac{1}{9450}$, ...
Euler also  gives here the value of Euler's constant as
$\gamma \approx 0, 57721~56649~01\underline{4}22~5.$
In Sect. 14 he finds the explicit value 
\begin{equation}\label{273}
 \Gamma'(1) :=\frac{dy}{dx}\, \large{\mid}_{x=0}\, = - \gamma.
\end{equation}
In Sect. 15 he  uses $\Gamma(\frac{3}{2})= \frac{1}{2} \sqrt{\pi}$ to derive the formula
$$
\frac{\Gamma'}{\Gamma}\left(\frac{3}{2}\right) = -\gamma + 2 - 2 \log 2.
$$
In Sect. 26 he observes that, for $x=n$ a positive integer, and  with $a$ treated as a variable,
\begin{equation}\label{274}
\Gamma(x+1) = \sum_{k=0}^{\infty} (-1)^k \left( {{x}\atop{k}}\right) (a-k)^x,
\end{equation}
the series terminates at $k=n$, and the resulting function is independent of $a$.
This leads him to  study in the remainder of the paper the series
$$
s= x^n - m(x-1)^{n} + \frac{m(m-1)}{1\cdot 2} (x-2)^n +
\frac{m(m-1)(m-2)}{1\cdot 2\cdot 3 } (x-3)^n +\&\mbox{c}
$$
 In this formula he sets $n = m + \lambda$, and he investigates  the expansion $s= s(x, \lambda)$ as a
 function of two variables $(\lambda, x)$ which corresponds to the variables $(x-n, a)$ 
 in \eqn{274}. 
 These expansions involve Bernoulli numbers.

\subsection{Zeta values}\label{sec23}
\setcounter{equation}{0}

Already as a young student of Johann Bernoulli (1667--1748)
in the 1720's Euler was shown problems on the summation of infinite
series of arithmetical functions, including in particular the sums of
inverse $k$-th powers of integers.  
\noindent The problem of finding a closed form expression for this value, $\zeta(2)$,  became
a celebrated problem, called ``The Basel Problem."   
The Basel problem had  been raised 
by Pietro Mengoli (1628-1686) in his 1650 book on the
arithmetical series arising in geometrical quadratures (see \cite{ME06}, \cite{Giu91}). 
In that book Mengoli \cite[Prop. 17]{Mengoli1650} summed in closed form the 
series  of (twice the) inverse triangular numbers
\begin{equation*}
\sum_{n=1}^{\infty}\, \frac{1}{n(n+1)} = \frac{1}{2} + \frac{1}{6} + \frac{1}{12} + \cdots = 1,
\end{equation*}
which is a telescoping sum. 
The problem became well known because it was also raised by 
Johann Bernoulli's older brother 
Jacob (also called James or Jacques) Bernoulli (1654--1705),  
earlier holder of the professorship of mathematics at the
University of Basel.   He wrote a book on infinite series 
(\cite{Bernoulli1689}) in 1689;  and several additional
works on them. A collection of this work appeared posthumously
in 1713 \cite{Bernoulli1713}. He could tell that 
 the sum of inverse squares  converged to a value between $1$ and  $2$, 
 using the triangular numbers, via
 $$
1 <  \sum_{n=1}^{\infty} \frac{1}{n^2} < 
 2 \left(\sum_{n=1}^{\infty} \frac{1}{n(n+1)} \right)= 
 2 \sum_{n=1}^{\infty}\left(\frac{1}{n}- \frac{1}{n+1}\right) = 2.
$$
 He was unable to relate this value to known constants, and
 wrote (\cite[art. XVII, p. 254]{Bernoulli1713}):

\begin{quote}
And thus by this Proposition, the sums of series can be found when
the denominators are either triangular numbers less a particular triangular
number or square numbers less a particular square number; by 
[the results of art.] XV,  this happens
for pure triangular numbers, as in the series
$\frac {1}{1}+\frac{1}{3}+\frac{1}{6}+\frac{1}{10}+\frac{1}{15}$
$\&$c., while on the other hand, when the numbers are pure squares, as in
the series $\frac{1}{1}+\frac{1}{4}+\frac{1}{9}+\frac{1}{16}+\frac{1}{25}$
$\&$c., it is more difficult than one would have expected, which is
noteworthy.
If someone should succeed in finding what till now withstood our efforts and communicate
it to us,  we would be much obliged to 
them.\footnote{``Atque ita per hanc Propositionem, inveniri possunt summae serierum, cum
denominatores sunt vel numeri Trigonales minuti alio Trigonali, vel Quadrati minuti
alio Quadrato; ut  \& per XV. quando sunt puri Trigonales, ut in serie
$\frac{1}{1} + \frac{1}{3} + \frac{1}{6} + \frac{1}{10} + \frac{1}{15} $ $\&$c.
at, quod notatu dignum, 
quando sunt puri Quadrati, ut in serie $\frac{1}{1}+ \frac{1}{4} + \frac{1}{9}+ \frac{1}{16}+ \frac{1}{25}$ 
$\&$c
difficilior est, quam quis expectaverit, summae pervestigatio, quam tamen finitam esse, ex altera,
qua manifesto minor est, colligimus: Si quis inveniat nobisque communicet, quod industriam
nostram elusit hactenus, magnas de nobis gratias feret.''[Translation: Jordan Bell]}
\end{quote}

Here he says that  if $ a \ge 1$ 
is an integer and  the  denominators range over $n(n+1) - a(a+1)$ or 
alternatively $n^2- a^2$, with
$n \ge a+1$, then he can sum the series, but he cannot sum the
second series when $a=0$.

Euler solved the Basel problem by  1735 with his famous result that 
$$\zeta(2) = \frac{\pi^2}{6}.$$
Euler's first contributions on the Basel problem involved
getting  numerical approximations of $\zeta(2)$. 
In \cite[E20]{E20}, written in 1731, integrating his representation for the harmonic numbers
led to   a representation (in modern terms) 
\begin{equation*}
\int_{0}^1 \frac{dx}{x} \int_{0}^x \frac{1- t^{n}}{1-t} dt = \sum_{j=1}^{n}\, \frac{1}{j^2}.
\end{equation*}
By a series acceleration method he obtained the formula 
\begin{equation*}
\zeta(2) = (\log 2)^2 + \sum_{n=1}^{\infty} \, \frac{1}{2^{n-1}n^2}, 
\end{equation*}
in which the final sum on the right  is recognizable as the dilogarithm value $2\Li_2(\frac{1}{2}).$
Euler  used this formula  to calculate $\zeta(2) \approx 1.64492~4,$ cf. Sandifer \cite{San07}. 
He next obtained a much better approximation using
Euler summation (the Euler-Maclaurin expansion).
He announced his summation method in \cite[E25]{E25} in 1732, and gave
details in   \cite[E47]{E47}, a paper probably written in 1734.  In this paper he computed
Bernoulli numbers and polynomials, and also
 applied his summation  formula to obtain
$\zeta(2)$ and $\zeta(3)$ to 20 decimal places.

In December 1735, in \cite[E41]{E41}, published in  1740, Euler
gave three  proofs  that $\zeta(2) = \frac{\pi^2}{6}$.
The most well known of these
treated $\sin  x$ as though it were a polynomial with infinitely many roots
at $\pm \pi, \pm 2 \pi, \pm 3\pi, \cdots$. He combines these roots in pairs, and
asserts the infinite product formula
\begin{equation}\label{262a}
\frac{\sin x}{x} = \left( 1- \frac{x^2}{\pi^2}\right)\left( 1- \frac{x^2}{4\pi^2}\right)\left( 1- \frac{x^2}{9\pi^2}\right) \cdots,
\end{equation}
He then equates certain algebraic combinations of  coefficients of the Taylor series expansion at $x=0$ on
both sides of the equation.  The right side coefficients  give the power sums of the roots of the polynomials.  Writing
these Taylor coefficients  as
$1 - \alpha x^2 + \beta x^4- \gamma x^6 + \ldots$, he creates algebraic combinations giving the power sums of
the roots, in effect obtaining\footnote{Euler only writes down the right sides of these equations.}
\begin{eqnarray*}
\frac{1}{\pi^2}\zeta(2) & = &\alpha\\
\frac{1}{\pi^4} \zeta(4) & = &\alpha^2 - 2\beta\\
\frac{1}{\pi^6} \zeta(6) & =  & \alpha^3- 3 \alpha \beta + 3 \gamma,
\end{eqnarray*}
and so on. Now the coefficients $\alpha, \beta, \gamma$
may be evaluated as rational numbers from the power series terms of the expansion of $\frac{\sin x}{x}$
which Euler computes directly.  In this way he  determined formulas for $\zeta(2n)$ for $1\le n \le 6$.
He also evaluated by a similar approach  alternating sums of odd powers, obtaining in
Sect. 10  the result of Leibniz  (1646--1716) that 
\[
\sum_{k=0}^{\infty} (-1)^k\frac{1}{2k+1} = \frac{\pi}{4}.
\]
He views this as confirming his method, saying:

\begin{quote}
And indeed this is the same series discovered some time ago by {\em Leibniz}, by
which he defined the quadrature of the circle. From this, if our method should appear to some
as not reliable enough, a great confirmation comes to light here; thus there should not
be any doubt about the rest that will be derived from this method.
\footnote{``Atque haec est ipsa series a {\em Leibnitio} iam pridem prolata, qua circuli quadraturam
definiuit. Ex quo magnum huius methodi, si cui forte ea non satis certa videatur, firmamentum elucet;
ita ut de reliquis, quae ex hac methodo deriuabantur, omnino non liceat dubitari''
[Trans: Jordan Bell]
}
\end{quote}
In  Sect. 12 he obtained a parallel result for cubes, 
\[
\sum_{k=0}^{\infty}(-1)^k \frac{1}{(2k+1)^3}= \frac{\pi^3}{32}.
\]
These proofs are  not strictly rigorous, because
the infinite product expansion is not completely justified,  but Euler makes
some numerical cross-checks, and his formulas work. 
Euler was sure his answer was right from his previous numerical
computations. In this same paper he gives two more derivations of the result,
still susceptible to criticism, 
 based  on the use of  infinite products. 
 
 We note that directly equating power series coefficients on both sides of \eqn{262a}, for the 
the $2n$-th coefficient  would give the identity
$$
\frac{(-1)^n}{(2n+1)!} = \frac{(-1)^n}{\pi^{2n}}\sum_{1 \le m_1 < m_2 < \cdots < m_n} \frac{1}{ m_1^2 m_2^2 \cdots m_n^2}.
$$
The sum on the right hand side is the multiple zeta value $\zeta({2, 2, ..., 2})$ (using $n$ copies of $2$, see \eqn{MZV} below),
and the case $n=1$ gives $\zeta(2)$. Relations like this may have inspired Euler's later study of multiple
zeta values.
 
Euler communicated his  original proof to 
his correspondents Johann Bernoulli, Daniel Bernoulli (1700--1782), Nicolaus Bernoulli (1687--1759)
and  others.  
Daniel  Bernoulli criticized it
on the grounds that $\sin \pi x$ may have other
complex roots, cf.  \cite[Sec. 4]{Ay74}, \cite[p. 264]{We84}.  In response  to this criticism Euler  eventually justified his infinite product
expansion for $\sin \pi x$, obtaining a version that can be made rigorous in modern terms, cf. Sandifer \cite{San07}.

In a  paper \cite[E61]{E61} published in 1743,
 Euler obtained more results using infinite products, indicating how to derive
formulas for $\zeta(2n),$ for all $n\ge 1$, as well as for the sums 
\[
[ L( 2n+1, \chi_{-4}) :=] \, \sum_{k=0}^{\infty} (-1)^k \frac{1}{(2k+1)^{2n+1}},
\]
for $n \ge 0$.
In this paper he  studied the function $e^x= \sum_{n=0}^{\infty} \frac{x^n}{n!}$.
He recalls his formula for $\sin s$, as
\[ 
s -\frac{s^3}{1\cdot 2\cdot 3} + \frac{s^5}{1\cdot 2 \cdot 3 \cdot 4\cdot 5} -\cdots
= s(1 -\frac{s^2}{\pi^2}) (1- \frac{s^2}{4\pi^2}) (1- \frac{s^2}{9 \pi^2})(1-\frac{s^2}{16\pi^2}) \cdots
\]
Then he  obtains the formula $\sin z= \frac{e^{iz}- e^{-iz}}{2i}$, 
introducing  the modern notation for $e$, saying: 

\begin{quote}
For, this expression is equal to 
$\frac{e^{s \sqrt{-1}} - e^{-s \sqrt{-1}}}{2\sqrt{-1}},$
where $e$ denotes that number whose
logarithm $= 1$,  and 
$$
e^z= \left( 1+ \frac{z}{n}\right)^n,
$$
with $n$ being an infinitely large
 number\footnote{``Haec enim expressio aequivalet isti
 $\frac{e^{s \sqrt{-1}} - e^{-s \sqrt{-1}}}{2 \sqrt{-1}}$
 denotante $e$ numerorum, cujus logarithmus est $=1, \& $ cum sit 
$e^z=\left(1+\frac{z}{n}\right)^n$ existente $n$ numero infinito, '' [Translation: Jordan Bell]}.
\end{quote}

\noindent 
Now he introduces a new principle.
 He evaluates the following integrals, 
 where $p$ and $q$ are integers with  $0< p<q$:
\[
\int_{0}^1 \frac{x^{p-1} + x^{q-p-1}}{1+x^q} dx = \frac{\pi}{ q \sin \frac{p\pi}{q}}
\]
and 
\[
\int_{0}^1 \frac{x^{p-1} - x^{q-p-1}}{1-x^q} dx =
 \frac{\pi \cos \frac{p\pi}{q}}{ q \sin \frac{p\pi}{q}}
\]
Expanding the  first integral as an indefinite integral
 in power series in $x$,  evaluating term by term,
and setting $x=1$ he obtains the formula
\[
\frac{\pi}{ q \sin \frac{p\pi}{q}}=\frac{1}{p} + \frac{1}{q-p} - \frac{1}{q+p} - \frac{1}{2q-p} +\frac{1}{2q+p} +
\frac{1}{3q-p} - \frac{1}{3q+p} + \cdots
\]
and similarly for the second integral.
Setting $\frac{p}{q} =s$, and rescaling, he obtains
\[ 
\frac{\pi}{\sin \pi s}= \frac{1}{s} + \frac{1}{1-s} - \frac{1}{1+s} -\frac{1}{2-s} + \frac{1}{2+s}
+\frac{1}{3-s} - \cdots
\]
and
\[
\frac{\pi \cos \pi s}{\sin \pi s} = \frac{1}{s} - \frac{1}{1-s} + \frac{1}{1+s} - \frac{1}{2-s} + \frac{1}{2+s}
-\frac{1}{3-s} + \cdots
\]
By differentiating with respect to $s$, he obtains
\[
\frac{\pi^2 \cos \pi s}{(\sin \pi s)^2} = \frac{1}{s^2} - \frac{1}{(1-s)^2} - \frac{1}{(1+s)^2}
+ \frac{1}{(2-s)^2} + \frac{1}{(2+s)^2} - \frac{1}{(3-s)^2} + \cdots
\]
and 
\[ 
\frac{ \pi^2}{(\sin \pi s)^2} = \frac{1}{s^2} + \frac{1}{(1-s)^2} + \frac{1}{(1+s)^2} + \frac{1}{(2-s)^2}
+ \frac{1}{(2+s)^2} + \cdots
\]
Now he substitutes $s= \frac{p}{q}$ with $0< p < q$ integers and obtains many identities.
Thus
\[ 
\frac{\pi^2}{ q^2 (\sin \frac{p \pi}{q})^2} = \frac{1}{p^2} + \frac{1}{(q-p)^2} + \frac{1}{(q+p)^2}
+ \frac{1}{(2q-p)^2} + \frac{1}{(2q+p)^2} + \cdots
\]
For $q=4$ and $p=1$  he obtains in this way
\[ 
\frac{\pi^2}{8 \sqrt{2}} = 1 -\frac{1}{3^2} - \frac{1}{5^2} + \frac{1}{7^2} + \frac{1}{9^2} - \frac{1}{11^2}
- \frac{1}{13^2} + \cdots
\]
and
\[ 
\frac{\pi^2}{8} = 1 + \frac{1}{3^2} + \frac{1}{5^2} + \frac{1}{7^2} +\frac{1}{9^2} + \frac{1}{11^2}+\cdots
\]
He observes that it is not difficult to derive on like
principles the value $\zeta(2) = \frac{\pi^2}{6}$. (This identity follows for example from the 
observation that  the series above is   $(1-\frac{1}{2^2}) \zeta(2)$.)
He concludes  that one may continue differentiating in $s$ to obtain formulas for various 
series in $n$-th powers, and gives several formulas for the derivatives. 
In this way one  can obtain a formula for $\zeta(2n)$ that expresses it as
 a rational multiple of $\pi^{2n}$. However this approach does not 
 make evident the closed formula  for $\zeta(2n)$ 
expressed in terms of Bernoulli numbers (given below as \eqn{exact}).

In a paper  \cite[E63]{E63}, published 
in 1743, Euler  gave a  new derivation of $\zeta(2) = \frac{\pi^2}{6}$,
using standard methods in calculus
involving trigonometric integrals, which is easily made  rigorous by
today's standards.   He takes $s$ to be arclength on a circle of  radius $1$
and sets $x= \sin s$ so that $s = \arcsin x$. He observes $x=1$ corresponds to $s= \frac{\pi}{2}$
and here he uses  the symbol $\pi$ with  its contemporary
meaning :

\begin{quote}
 It is clear that I employ here the letter $\pi$ to denote the number
of  {\em Ludolf van Ceulen}  $3.14159265...$\footnote{``Il est clair que j'emplois la lettre $\pi$ pour marquer le nombre de
 LUDOLF \`{a} KEULEN 3,14159265 etc.''}
\end{quote} 

\noindent  Here Euler refers to  Ludolf van Ceulen (1540-1610), who  was a professor at
 Leiden University and who  computed $\pi$ to 35 decimal places. 
 In his calculations Euler works  with differentials
 and writes $ds =\frac{dx}{\sqrt{1-xx}}$, so that $s = \int \frac{dx}{\sqrt{1-xx}}$. Multiplying these expressions 
 gives 
 \[
 s ds = \frac{dx}{\sqrt{1-xx}} \int \frac{dx}{\sqrt{1-xx}}.
 \]
 He integrates both sides from $x=0$ to $x=1$. The left side is $\int_{0}^{\frac{\pi}{2}} s ds = \frac{\pi\pi }{8}.$
 For the right side, he uses the binomial theorem 
 \[
 \frac{1}{\sqrt{1-xx}} [ = (1- xx)^{-\frac{1}{2}} ]= 1 + \frac{1}{2} x^2 + \frac{1\cdot 3}{2 \cdot 4} x^4 + 
 \frac{1\cdot 3 \cdot 5}{2 \cdot 4\cdot 6} x^6 + \mbox{etc.}
 \]
 In fact this expansion converges for $|x|<1$. Euler integrates it term by term to get
 \[ 
 \int \frac{dx}{\sqrt{1-xx}} = x +\frac{1}{2\cdot 3} x^3 + \frac{1\cdot 3}{2 \cdot 4\cdot 5} x^5 + 
 \frac{1\cdot 3 \cdot 5}{2 \cdot 4\cdot 6\cdot 7} x^7+ \mbox{ etc.}
\]
He then obtains
\[
s ds = \frac{x dx}{\sqrt{1-xx}} + \frac{1}{2 \cdot 3} \frac{x^3 dx}{\sqrt{1-xx}} +
\frac{1\cdot 3}{2 \cdot 4\cdot 5} \frac{x^5 dx}{\sqrt{1-xx}}+ 
 \frac{1\cdot 3 \cdot 5}{2 \cdot 4\cdot 6\cdot 7} \frac{x^7 dx }{\sqrt{1-xx}} + \mbox{ etc.}
 \]
 Integrating from $x=0$ to $x=1$ would give
 \[
 \frac{\pi\pi}{8} =  \int_{0}^1 \frac{x dx}{\sqrt{1-xx}} + 
 \frac{1}{2\cdot 3} \int_{0}^1\frac{x^3dx}{\sqrt{1-xx}} + \frac{1\cdot 3}{2 \cdot 4\cdot 5} \int_{0}^1\frac{x^5 dx}{\sqrt{1-xx}} + 
  \mbox{etc.}
 \]
 The individual terms on the right can be integrated by parts
 \[ 
 \int \frac{x^{n+2}dx}{\sqrt{1-xx}} = \frac{n+1}{n+2} \int \frac{x^n dx}{\sqrt{1-xx}} - \frac{x^{n+1}}{n+2} \sqrt{1-xx}.
 \]
 When integrating from $x=0$ to $x=1$, the second term on the right vanishes and Euler gives
 a table
 \begin{eqnarray*}
 \int_{0}^1 \frac{x dx}{\sqrt{1-xx}} & =& 1- \sqrt{1-xx} = 1\\
 \int_{0}^1 \frac{x^3 dx}{\sqrt{1-xx}} &=& \frac{2}{3} \int_{0}^1 \frac{xdx}{\sqrt{1-xx}} = \frac{2}{3}\\
 \int_{0}^1 \frac{x^5 dx}{\sqrt{1-xx}} &=& \frac{4}{5} \int_{0}^1 \frac{x^3dx}{\sqrt{1-xx}} = \frac{2\cdot 4}{3\cdot 5}\\
 \int_{0}^1 \frac{x^7 dx}{\sqrt{1-xx}} &=& \frac{6}{7} \int_{0}^1 \frac{x^5dx}{\sqrt{1-xx}} = \frac{2\cdot 4\cdot 6}{3\cdot 5\cdot 7}\\
 \int_{0}^1 \frac{x^9 dx}{\sqrt{1-xx}} &=& \frac{8}{9} \int_{0}^1 \frac{x^7dx}{\sqrt{1-xx}} = 
 \frac{2\cdot 4\cdot 6\cdot 8}{3\cdot 5\cdot 7 \cdot 9}
 \end{eqnarray*}
  Substituting these values in the integration above yields
  \[
  \frac{\pi\pi}{8} = 1+ \frac{1}{3 \cdot 3} + \frac{1}{5 \cdot 5} + \frac{1}{7\cdot 7} + \frac{1}{9 \cdot 9} + \mbox{etc.}
 \]
 This gives the sum of reciprocals of odd squares, and Euler observes that multiplying by $\frac{1}{4}$ gives
 \[
   \frac{1}{4} + \frac{1}{16} +\frac{1}{36} +\frac{1}{64} + \mbox{etc.}
\]
from which follows 
\[
\frac{\pi\pi}{6} = 1 + \frac{1}{4} + \frac{1}{9} +\frac{1}{16} +\cdots
\]
This proof is rigorous and unarguable. 

 He goes on to  evaluate sums  of larger even powers by the same methods, and
gives a table of explicit evaluations for $\zeta(2n)$ for $1 \le n \le 13$, in the
 form 
 \[
 [\zeta(2n)= ]\, \frac{2^{2n-1}}{(2n+1)!} C_{2n} \pi^{2n}.
 \]
 In this expression, he says part of the formula is known and explainable, and that the 
 remaining difficulty  is to explain the nature of the fractions
[  $C_{2n}$]  which  take values of a different character:
 \[ 
 \frac{1}{2}~~~\, \frac{1}{6} ~~~\, \frac{1}{6} ~~~\,  \frac{3}{10}  ~~~\,\frac{5}{6}  ~~~\, \frac{691}{210} ~~~\,  \frac{35}{2} ~~  \mbox{etc.}
 \]
 He says that he has two other methods for finding these numbers, whose nature
 he does not specify.  As we know,  these involve Bernoulli numbers.

In this same period, around 1737,  not published until 1744, Euler \cite[E72]{E72}
obtained the  ``Euler product'' expansion of the zeta function. In \cite[Theorem 8]{E72}
he proved:

\begin{quote}
The expression formed from the sequence of prime numbers
$$
\frac{2^n \cdot 3^n \cdot 5^n \cdot 7^n \cdot 11^n \cdot \mbox{etc.}}
{(2^n-1)(3^n -1)(5^n -1) (7^n-1) (11^n -1) ~~\mbox{etc.} }
$$
has the same value as the sum of the series
$$
1+ \frac{1}{2^n} + \frac{1}{3^n} + \frac{1}{4^n}+\frac{1}{5^n}+\frac{1}{6^n} +\frac{1}{7^n} + \mbox{etc.}
\footnote{``Si ex serie numerorum primorum sequens formetur espressio
$\frac{2^n \cdot 3^n \cdot 5^n \cdot 7^n \cdot 11^n \cdot \mbox{etc.}}
{(2^n-1)(3^n -1)(5^n -1) (7^n-1) (11^n -1) ~~\mbox{etc.} }$
erit eius aequalis summae huius seriei
$
1+ \frac{1}{2^n} + \frac{1}{3^n} + \frac{1}{4^n}+\frac{1}{5^n}+\frac{1}{6^n} +\frac{1}{7^n} + \mbox{etc.} ''
$
[Translation: David Pengelley \cite{Pe00}, \cite{Pe07}.] Here $n$ is an integer.}
$$
\end{quote} 
This infinite product  formula for $\zeta(n)$
supplies  in principle another way to approximate the values $\zeta(n)$
for $n \ge 2$, using finite products.

Euler made repeated attempts throughout his life
to find closed forms for the odd zeta values, especially
$\zeta(3)$.  In a 1739 paper \cite[E130]{E130}, not published until 1750,  he  developed partial fraction decompositions
and   a method equivalent to Abel summation. In particular,
letting $\theta(s) = \sum_{n=0}^{\infty} \frac{1}{(2n+1)^s}$ and 
$\phi(s) = \sum_{n=1}^{\infty}\frac{(-1)^{n-1}}{n^s}$,
 he obtained the formula
$$
\phi(1-2m) = \frac{(-1)^{m-1} 2 \cdot (2m-1)!} {\pi^{2m} }\theta(2m)
$$
for $m=1, 2, 3, 4$, in which the left side is determined as the Abel sum
$$
\phi(m) := \lim_{x \to 1^{-}} \sum_{n=1}^{\infty} \frac{(-1)^{n-1}x^n}{n^m};
$$
in fact the function $\phi(s)$ is Abel summable  for all $s \in \CC$. 

In a paper written in 1749 \cite[E352]{E352}, 
not presented until 1761, and published in 1768, Euler derived the functional equation for the zeta function
at integer points, and he conjectured
$$
\frac{\phi(1-s)}{\phi(s)} = \frac{ - \Gamma(s)(2^s-1) \cos \frac{\pi s}{2} } {(2^{s-1}-1)\pi^s}.
$$
He apparently hoped to use this functional relation to deduce a closed form for the 
odd zeta values, including
$\zeta(3)$, and says: 

\begin{quote}
As far as the [alternating] sum of reciprocals of powers is concerned,
$$
1 - \frac{1}{2^n}+ \frac{1}{3^n}-\frac{1}{4^n} + \frac{1}{5^n} - \frac{1}{6^n} + \mbox{\&c}
$$
I have already observed the sum [$\phi(n)$] can be assigned a value only when $n$ is even
and that when $n$ is odd, all my efforts have been useless up to 
now.\footnote{``A l'\'{e}gard des s\'{e}ries r\'{e}ciproques des puissances
$$
1 - \frac{1}{2^n}+ \frac{1}{3^n}-\frac{1}{4^n} + \frac{1}{5^n} - \frac{1}{6^n} + \mbox{\&c}
$$
j'ai d\'{e}j\`{a} observ\'{e}, que leurs sommes ne sauroient \^{e}tre assign\'{e}es que
lorsque l'exposant $n$ est un nombre entier par, \& que pour les cas
ou $n$ est un nombre entier impair, tous mes soins ont \'{e}t\'{e} jusqu'ici inutiles.''}
\end{quote}

 In his 1755 book on differential calculus Euler \cite[E212, Part II, Chap. 6, Arts. 141-153]{E212}
 gave closed forms for $\zeta(2n)$ explicitly expressed in terms of the Bernoulli numbers.
These are equivalent to the modern form
\begin{equation}\label{exact}
\zeta(2n) =  (-1)^{n+1} \frac{B_{2n} (2 \pi)^{2n}}{2 (2n!)}.
\end{equation}
He also used his 
 summation formula to numerically evaluate the sums using
divergent series, as follows. To estimate  $\zeta(2)$
he considers the finite sums
$$
s= 1+ \frac{1}{4} + \frac{1}{9} + \cdots + \frac{1}{x^2},
$$
where $x$ is fixed, and  obtains the 
summation
\begin{equation}\label{256a}
s= C - \frac{1}{x} + \frac{1}{2x^2} -\frac{|B_2|}{x^3} + \frac{|B_4|}{x^5} - \frac{|B_6|}{x^7} +\cdots,
\end{equation}
where the constant $C$ 
 is to be determined by finding all the other terms at one value of $x$.
Choosing $x=\infty$ we have $s= \zeta(2)$ and the summation formula 
formally gives $s= C= \zeta(2).$ 
On setting  $x=1$, the finite sum gives $s=1$, therefore one has, formally,  
\begin{equation*}
C= 1+ 1  -\frac{1}{2} +|B_2|-|B_4|+|B_6|-\cdots
\end{equation*}
The right side of this expression  is a divergent series and 
Euler  remarks:

\begin{quote}
But this series alone does not give the value
of $C$, since it diverges strongly.
Above [Section 125] we  demonstrated that the  sum of  the series to infinity is  $\frac{\pi \pi}{6}$;
and therefore setting $x= \infty$ and $s= \frac{\pi \pi}{6}$, we have
$C= \frac{\pi \pi}{6}$, because then all other terms vanish. Thus it follows that
$$
1+ 1 -\frac{1}{2} +|B_2|-|B_4|+|B_6|-\cdots= \frac{\pi\pi}{6}.\footnote{``quae series autem cum sit maxime divergens, 
valorem constantis
$C [:=\zeta(2)]$ non ostendit. Quia autem supra demonstavimus summam huius seriei in infinitum 
continuatae esse $= \frac{\pi \pi}{6}$; facto $x=\infty$, si ponatur $s= \frac{\pi \pi}{6}$,
fiet $C= \frac{\pi\pi}{6},$ ob reliquos terminos omnes evanescentes. Erit ergo
$
1+ 1 -\frac{1}{2} + {\mathfrak U}  - {\mathfrak B} + {\mathfrak C} - {\mathfrak D} +{\mathfrak E} 
  -\&c= \frac{\pi\pi}{6}.
$
'' [Translation: David Pengelley \cite[Sec. 148]{Pe00}, \cite{Pe07}]}
$$
\end{quote}

\noindent
Euler's idea is to use these divergent series
to obtain good numerical approximations to $C$ by substituting a finite
value of $x$  and to add up the terms
of the series ``until it begins to diverge", i.e. to truncate it
according to some recipe.
We do not address  the question of what rule Euler used to determine
where to truncate. 
This procedure has received modern justification in the theory of
asymptotic series, cf. Hardy \cite{Ha49}, Hildebrand \cite[Sec. 5.8]{Hil56}, \cite[Sec 6.4]{Ed74}.  
By taking large $x$ one can get
arbitrarily accurate estimates (for some functions). 
If the series is alternating, one 
may sometimes estimate the error as being no larger than the smallest term.
Euler demonstrates in 
this particular case how to find the value $C= \frac{\pi^2}{6}$ numerically
 using the value $x=10$, with $s=\sum_{i=1}^{10} \frac{1}{n^2}$,
 by truncating the expansion \eqn{256a} after power $\frac{1}{x^{17}}$, obtaining
$C \approx 1,64493~40668~48226~430$. He writes: 

\begin{quote}
And this number is at once the value of the expression $\frac{\pi^2}{6}$, as will be
apparent by  doing the calculation with the known value of $\pi$. 
Whence it is at once understood that even though
the series [$|B_2|, |B_4|, |B_6| \cdots$]
diverges, still the true sum is produced.\footnote{``Hicque numerus simul est valor 
expressionis $\frac{\pi \pi}{6}$,
quemadmodem ex valore ipsius $\pi$ cogito calculum institutenti patebit. Unde simul
intelligitur, etiamsi series ${\mathfrak A}, {\mathfrak B}$, ${\mathfrak C}$ \& c. divergat,
tamen hoc modo veram prodire summan.'' [Translation: Jordan Bell]}.
\end{quote}

\noindent
Euler also applies this method to the odd zeta values, 
obtaining numerical approximations for all zeta values up to $\zeta(15)$,  
in particular $\zeta(3)\approx  \frac{\pi^3}{25.79436}$.
He also (implicitly) introduces  in Sec. 153 a ``Bernoulli function" $b(z)$ which interpolates 
at integer values 
the variable $z$ his version of the even Bernoulli numbers, i.e. 
$$
b(n)=  |B_{2n}| = (-1)^{n+1} B_{2n}, ~~~~n \in \ZZ_{\ge 1}.
$$
Euler  remarks: 

\begin{quote}
The series of Bernoulli numbers [$|B_{2}|, |B_4|, |B_6|, \cdots$~] 
however irregular, seem to be interpolated from this source; that is,
terms constituted in the middle of two of them can be defined: for if the
middle term lying between the first $A$ and the second $B$, corresponding 
to the index $1 \frac{1}{2}$ were $=p$, then it would certainly be 
$$
1 +\frac{1}{2^3} + \frac{1}{3^3} + \&\mbox{tc}= \frac{2^2 p}{1 \cdot 2\cdot 3} \pi^3
$$
and hence
$$
p = \frac{3}{2 \pi^3} (1+ \frac{1}{2^3} + \frac{1}{3^3} + \&c) = 0,05815~227.\footnote{``Ex hoc fonte series numerorum Bernoulliarum 
${{1}\atop{\mathfrak U}} {{2}\atop{\mathfrak B}} {{3}\atop{\mathfrak C}} [\cdots]$
quantumvis irregularis videatur
interpolari, seu termini in medio binorum quorumcunque constituti assignari poterunt; se
enim terminus medium interiacens inter primum ${\mathfrak A}$ and secundum ${\mathfrak B}$,
sue indici $1 \frac{1}{2}$, respondens fuerit $=p$; erit utique 
$1 + \frac{1}{2^3}+ \frac{1}{3^3} +  \mbox{\&tc.} = \frac{2^2 p}{1\cdot 2 \cdot 3} \pi^3$
ideoque
$p = \frac{3}{2 \pi^{3}} (1+ \frac{1}{2^3} + \frac{1}{3^3} + \&c)= 0,05815227.$''
[Translation: Jordan Bell]} 
$$
\end{quote}

\noindent
That is, Euler shows the interpolating  value  $p:=b(\frac{3}{2})$ is related to $\zeta(3)$,
by the relation
$$
p = \frac{3}{2}\frac{\zeta(3)}{\pi^3}, 
$$
and he finds the
 numerical value  $p \approx 0.05815227$. He similarly obtains a value for $b(\frac{5}{2})$
in terms of $\zeta(5)$.

 In a 1771 paper, published in 1776, Euler  \cite[E477]{E477} ({\em ``Meditations about a singular type
 of series"}) introduced multiple zeta values
and with them obtained another formula for $\zeta(3)$. He had considered this topic
 much earlier. In Dec. 1742 Goldbach
wrote Euler a letter   about multiple zeta values, 
and they exchanged several further letters on their properties.
The 1771 paper recalls  this correspondence
with Goldbach, and considers the sums, for $m, n \ge 1$ involving
the $n$-harmonic numbers 
\begin{equation}\label{265a}
S_{m,n} := \sum_{k=1}^{\infty} \frac{H_{k,n}}{k^m}
=\sum_{k \ge 1} \left( \sum_{1 \le l \le k} \frac{1}{k^m l^n} \right) = \sum_{l \ge 1} \sum_{k \ge l} 
\left(\frac{1}{k^m l^n}\right).
\end{equation}
Euler lets $\int \frac{1}{z^m}$ denote $\zeta(m)$ and lets $\int \frac{1}{z^m}(\frac{1}{z})^n$
denote the sum $S_{m,n}.$ He considers all  values $1\le m, n \le 8$.
This  includes divergent series, since all sums $S_{1, n}$  
are divergent.
The modern definition of multiple zeta values, valid for positive
integers $(a_1, ..., a_k)$ with  $a_1 \ge 2$,  is
\begin{equation}\label{MZV}
\zeta(a_1, a_2, ..., a_k) = 
\sum_{n_1>n_2 > \cdots > n_k > 0}  \frac{1}{n_1^{a_1} n_2^{a_2} \cdots n_k^{a_k}}.
\end{equation}
In this   notation  \eqn{265a} expresses the identities (valid for $m \ge 2, n \ge 1$)
$$
S_{m,n} = \zeta(n+m) + \zeta(m,n).
$$
 Euler presents three methods for obtaining identities, 
 and  establishes a number of additional identities for each range $m+n=j$ for $2 \le j \le 11$.
 For $j=3$ he obtains
 $$S_{2,1} = 2\, \zeta(3).$$ 
This is equivalent to
\begin{equation*}
\zeta(3) = \zeta(2,1) = \frac{1}{2} \Big(\sum_{k=1}^{\infty} \frac{H_k}{k^2}\Big).
\end{equation*}
Being thorough, Euler also  treats  cases $m=1$ where the series diverge, 
and observes that  for $j=2$ he cannot derive useful identities.

 In a 1772 paper
\cite[E432]{E432} (``{\em Analytical Exercises}"), published in 1773,  Euler obtained new formulas for values $\zeta(2n+1)$
via divergent series related to the functional equation for the zeta function.
 As an example, in Sect. 6 he writes, treating $n$ as a positive integer and $\omega$ as
 an infinitely small quantity, the formula
$$
 1 + \frac{1}{3^{n+\omega}} + \frac{1}{5^{n+ \omega}} + \cdots=  \frac{-1}{2 \cos( \frac{n+\omega}{2}\pi)}
 \frac{\pi^{n+\omega}}{1 \cdot 2 \cdots (n-1+\omega)} \Big( 1- 2^{n-1+\omega} + 3^{n-1+\omega} - 
  \mbox{etc}\Big).
$$
Here the  right side contains a divergent series. In Sect. $7$,  he in effect takes a
derivative in $\omega$ and obtains, for example in  the case $n=3$, the formula
$$
1 + \frac{1}{3^3} + \frac{1}{5^3} + \frac{1}{7^3} + \mbox{etc.} = -\frac{\pi^2}{2} (2^2 \log 2 - 3^2 \log 3
+ 4^2 \log 4 - 5^2 \log 5 + \mbox{etc.} \Big)
$$
The right side is another divergent series,  which he names
$$
Z := 2^2 \log 2 - 3^2 \log 3
+ 4^2 \log 4 - 5^2 \log 5 + 6^2 \log 6 - 7^2 \log 7+ \mbox{etc.}
$$
He applies various series acceleration techniques on this divergent series, 
compare the discussion of divergent series in Section 2.5 below. In Sect.  20 he transforms $Z$ into 
a rapidly convergent series,
\begin{equation*}
Z = \frac{1}{4} - \frac{ \alpha \pi^2}{3 \cdot 4 \cdot 2^2} - \frac{\beta \pi^4}{ 5\cdot 6 \cdot 2^4} - \frac{ \gamma \pi^6}{7\cdot 8 \cdot 2^6}
-\frac{\delta \pi^8}{9\cdot 10 \cdot 2^{8}} - \frac{\epsilon \pi^{10}}{ 11\cdot 12 \cdot 2^{10}} - \mbox{etc.}
\end{equation*}
In his notation $\zeta(2) = \alpha \pi^2$, $\zeta(4) = \beta \pi^4$ etc, so he thus obtains the valid, rapidly convergent formula
\begin{equation*}
\frac{7}{8} \zeta (3) = \frac{\pi^2}{2} \Big( \frac{1}{4} - \sum_{n=1}^{\infty} \frac{\zeta(2n)}{(2n+1)(2n+2) 2^{2n}} \Big).
\end{equation*}
After further changes of variable, he 
deduces in Sect.  21  the formula
\[
1 + \frac{1}{3^3} + \frac{1}{5^3} + \frac{1}{7^3} + \mbox{etc.} = \frac{\pi^2}{4} \log 2 + 2 \int_{0}^{\frac{\pi}{2}} \Phi \log( \sin \Phi )d\Phi.
\] 
This may be rewritten
\begin{equation*}
\frac{7}{8}\zeta(3) = \frac{\pi^2}{4} \log 2 + 2 \int_{0}^{\frac{\pi}{2}} x \, \log (\sin x) dx.
\end{equation*}
Euler notes also the ``near miss" 
that a variant of the last integral is exactly evaluable:
\begin{equation*}
 \int_{0}^{\frac{\pi}{2}}  \log (\sin x) dx = -\frac{\pi}{2} \log 2.
 \end{equation*}

Euler  continued to consider  zeta values. In the 1775 paper 
\cite[E597]{E597} (``{\em A new and most easy method for 
summing series of reciprocals of powers}"), published posthumously
in 1785, he gave  a new derivation of his formulas for the even zeta values $\zeta(2n)$.
In the 1779 paper \cite[E736]{E736},  published  in 1811, he studied the dilogarithm function 
$Li_2(x) := \sum_{n=1}^{\infty} \frac{x^n}{n^2}$. 
 In Sect.  4 he establishes the functional equation
\begin{equation*}
Li_2(x) + Li_2(1-x) = \frac{\pi^2}{6} - (\log x)( \log (1-x)).
\end{equation*}
The special case $x=0$ of this functional equation 
gives again   $\zeta(2)= \frac{\pi^2}{6}$. 
Euler deduces several other functional equations for the dilogarithm, and by specialization, 
a number of new series identities.
\subsection{Summing divergent series: the Euler-Gompertz constant}\label{sec25}
\setcounter{equation}{0}

In the course  of his investigations Euler encountered many divergent
series. His work on divergent series  was a by-product of his aim to obtain accurate numerical estimates for the
sum of various convergent series, and he learned and developed 
 summability methods which could be applied to both convergent and divergent series.
 By  cross-checking
 predictions  of various summability methods 
 he developed useful methods to sum divergent series.
Of  particular interest here,  Euler's summation of a particular divergent series,
 the Wallis hypergeometric series,  produced an  
interesting new constant, the  Euler-Gompertz constant.

 Euler's  near contemporary James Stirling (1692--1770)  developed methods of acceleration of convergence of series by
differencing, starting in 1719, influenced by Newton's work (\cite[pp. 92-101]{Newton1711}) on interpolation,
see  Tweddle \cite{Twe92}. 
 Euler obtained  Stirling's  1730 book ( ``{\em Differential methods and summation of series}")  
 (\cite{Stirling1730}, translated in \cite{Twe03}) which  discussed
 infinite series, summation, interpolation and quadrature, in which
 Proposition $14$  gives three examples of a method of accelerating convergence of series
of hypergeometric type. On 8 June 1736 Euler wrote to Stirling, saying (\cite[p. 229]{Twe03}, \cite[p. 141]{Twe88}):

\begin{quote}
I searched all over for your  excellent book on the method of differences, 
a review of which I had seen in the {\em Acta Lipenses}, until I have achieved
my desire. Now that I have read through it diligently, I am truly astonished at the
great abundance of excellent methods in such a small volume, by which you
show to sum slowly convergent series with ease and how to interpolate
progressions which are very difficult to deal with.
But especially pleasing to me was Prop. 14 of Part I, in which you 
present a method for summing so easily series whose law of progression
is not even established, using only the relation of the last terms; certainly
this method applies very widely and has the greatest use. But the demonstration
of this proposition, which you seemed to have concealed from study, caused
me immense difficulty, until at last with the greatest pleasure I obtained it 
from the things which had gone before...
\end{quote}

\noindent   A key point 
is that  Proposition 14 applies when the 
members $a_n$ of the series  satisfy  linear recurrence relations
whose coefficients are rational functions of the parameter $n$. This point was not made explicit by Stirling.
 It is known that there exist  counterexamples to
Proposition 14 otherwise (cf. Tweddle \cite[pp. 229--233]{Twe03}).
The convergence acceleration method of Stirling is really a suite of methods of great flexibility,
as is explained\
 by  Gosper \cite{Gos76}.\footnote{Gosper \cite[p.122]{Gos76} says: ``We will be taking up almost exactly
where James Stirling left off, and had he been granted use of a symbolic mathematics
system such as MIT MACSYMA, Stirling would probably have done most of this work
before 1750."} 
Euler  incorporated these  ideas
 in his  subsequent work.\footnote{We do not attempt to untangle the exact
 influence of Stirling's work on Euler. In the rest of his 1736 letter Euler explains his own
 Euler summation method for summing series. In a  reply made  16 April 1738 Stirling tells Euler that
 Colin Maclaurin will be  publishing a book on Fluxions that has a similar summation formula
 that he found some time ago. Euler replied on 27 July  1738  
 that `` Mr. Maclaurin probably came
 upon his summation formula before me, and consequently deserves to be
 named as its first discoverer. For I found that theorem about four years ago, at which time
 I also described its proof and application in greater detail to our Academy."  Euler also
 states that he independently found some summability methods in Stirling's book, before
 he knew of its existence (\cite[Chap. 6]{Twe88}).}.

In a 1760 paper \cite[E247]{E247} ({\em ``On divergent series"})
  Euler formulated his views on  the use and philosophy of assigning a finite value
to some divergent series.  This paper includes work Euler had done as early as 1746.
In it  Euler presents  four  different methods
for summing the particular divergent series
\begin{equation}\label{554a}
 1- 1 +2 -6  +24 - 120 +\cdots  [= 0! - 1! + 2! - 3! + 4! -5! +\cdots],
\end{equation}
which he calls  ``Wallis'  hypergeometric series,"
and shows they each give (approximately) the same answer,
which he denotes $A$.
Barbeau and Leah \cite{BL76} give a  translation of the initial, philosophical part of this paper, 
followed by  a summary of the rest of its  contents, and Barbeau \cite{Bar79} gives a detailed exposition of 
Euler's treatment of this example.

The series \eqn{554a} is termed ``hypergeometric" because each term in the
series is multiplied by a ratio which varies from term to term,
rather than by a constant ratio  as in a  geometric series.
Series of this type appeared in the Scholium to Proposition 190
in Wallis's 1656 book  {\em Arithmetica Infinitorum} \cite{Wal1656}\footnote{This reference was noted by Carl Boehm,
an editor of Euler's Collected Works, cf. \cite[p. 157]{BL76}.}.

Euler's four methods were: 
\begin{enumerate}
\item[(1)] \cite[Sect. 13--16]{E247}
The iteration of a series acceleration method
involving a differencing operation applied to this series. 
Euler  iterates four times, truncates the 
resulting divergent
series and obtains an approximation to its ``sum"  $A$ of
$A \approx 38015/65536 = 0.58006...$;  
\item[(2)] \cite[Sect. 17-18]{E247}
The interpolation  of  solutions to 
the difference equation $P_{n+1} = n P_{n} +1,$ with $P_1 =1$.
Euler  obtains
the closed formula 
$$P_n = 1 + (n-1) + (n-1)(n-2) + (n-1)(n-2)(n-3) +\cdots,$$
which is a finite sum for  each $n \ge 1$, and  observes that  substituting $n=0$ formally gives
the Wallis series. He considers interpolation methods that 
rescale $P_{n}$ by an invertible function $f(x)$ to obtain values $f(P_n)$, and  then attempts
to interpolate the value $f(P_0).$ For the choice $f(x)= \frac{1}{x}$ he obtains $A \approx 0.6$ and for 
$f(x)= \log x$ he obtains
$A \approx 0.59966$; 
\item[(3)] \cite[Sect. 19--20]{E247}
The construction and solution of an associated differential equation.
Euler views the  Wallis series as the special case $x=1$
of the power series
$$
s(x) = 1 - 1!\,x + 2! \,x^2 - 3! \,x^3 + 4!\,x^4 - 5! \,x^5 + \cdots.
$$
Although this series diverges for all nonzero $x \in \CC$, he observes
 that $s(x)$ formally satisfies a linear differential equation.
 This differential equation   has a convergent solution
given by an  integral with variable endpoint $x$,  and by specializing it to 
$x=1$ he obtains an exact  answer $A$.
  We discuss this method in detail below. 
\item[(4)] \cite[Sec. 21--25]{E247}
The determination of  a continued fraction expansion for the
series $s(x)$ given by the integral in (3). This  is 
\begin{equation}\label{eulerfrac}
s(x) = \cfrac{1}{1+ \cfrac{x}{1+ \cfrac{x}{1+\cfrac{2x}{1+ \cfrac{2x}{1+\cfrac{3x}{1+ \cfrac{3x}{1+\cdots}}}}}}}
\end{equation}
This continued fraction is convergent at $x=1$. Using it, he computes a good numerical approximation,
 given in \eqn{559d} below, to the answer $A$.
\end{enumerate}  
.

In the third of four summation methods,  Euler  considers the
series to be the ``value" at $x=1$ of the asymptotic
series
\begin{equation}\label{555a}
s(x) \sim \sum_{n=0}^{\infty} (-1)^n n! \,x^{n+1}.
\end{equation}
He notes  this series satisfies  the 
inhomogeneous linear differential equation
\begin{equation}
\label{556a}
s'(x) + \frac{1}{x^2} s(x)= \frac{1}{x}.
\end{equation}
He writes down an explicit solution to the differential equation \eqn{556a}, namely
\begin{equation}\label{557a}
s(x) := e^{\frac{1}{x}}\int_{0}^x \frac{1}{t} e^{-\frac{1}{t}} dt.
\end{equation}
The expansion of this integral near $x=0$ has the asymptotic expansion \eqn{555a}.
Euler transforms the integral \eqn{557a} into other forms.
Using the change of variable $v= e^{1- \frac{1}{t}}$ he obtains
\begin{equation}\label{557b}
s(x) = e^{(\frac{1}{x} -1)} \int_{0}^{e^{1-\frac{1}{x}}} \frac{dv}{1-\log v}.
\end{equation}
He then obtains as his  proposed  summation for the divergent series the value
\begin{equation}\label{559a}
\EG := s(1) = e \int_{0}^1 \frac{1}{t} e^{- \frac{1}{t}} dt = \int_{0}^{1} \frac{dv}{1- \log v}.
\end{equation}
He numerically estimates this integral using the trapezoidal rule
and  obtains $\EG \approx 0.5963\underline{7} 255.$
Using his fourth method, in Sec. 25 Euler  obtained a more precise numerical value for
this constant,  finding 
 \begin{equation}\label{559d}
 \EG \approx 0. 59634~7362\underline{1}~237.
 \end{equation}
The  constant $\EG$ defined by \eqn{559a} is of some interest  and we call it the
{\em Euler-Gompertz constant,} following 
Finch \cite[Sect. 6.2]{Fin03}. 
This constant  was earlier named  the {\em Gompertz constant}  by Le Lionnais \cite{LL83}.
\footnote{Le Lionnais \cite{LL83} gives no  explanation or reference for his choice of  name ``Gompertz".}
Benjamin Gompertz (1779--1865) was a mathematician and actuary who in 1825
proposed  a functional model for approximating mortality tables, cf. Gompertz \cite{Gom1825},\cite{Gom1862}.
The {\em Gompertz function} and {\em Gompertz distribution} are named in his honor. 
 I have not located
 the Euler-Gompertz constant in his writings.
  However Finch \cite[Sect. 6.2.4]{Fin03}  points out a  connection 
 of the Euler-Gompertz constant with a Gompertz-like  functional model.

 The Euler-Gompertz constant has some
 properties analogous to Euler's constant.  The   integral representations for the Euler-Gompertz
 constant  given by  \eqn{557a} and \eqn{557b} evaluated at $x=1$
 resemble  the  integral 
 representations  \eqn{226a} and \eqn{227a}
 for Euler's constant. An elegant parallel was noted by Aptekarev \cite{Apt09},
 \begin{equation}\label{559f}
 \gamma = - \int_{0}^{\infty} \log x \, e^{-x}dx, \quad\quad \delta= \int_{0}^{\infty} \log (x+1) \,e^{-x} dx.
\end{equation} 
 There are also 
 integral formulae in which both constants appear, see \eqn{353EG} and \eqn{362EG}.

 In 1949 G. H. Hardy \cite[p. 26]{Ha49} 
gave a third integral representation for the function $s(x)$,  obtained by
using the change of variable $t= \frac{x}{1+xw}$, 
namely
\begin{equation*}
s(x) = x \int_{0}^{\infty} \frac{e^{-w}}{1+xw} dw.
\end{equation*}
Evaluation at $x=1$ yields the integral representation
\begin{equation}\label{560aa} 
\EG = \int_{0}^{\infty} \frac{e^{-w}}{1+w} dw.
\end{equation}
From this representation Hardy obtained a relation of
$\EG$ to Euler's constant given by 
\begin{equation}\label{560b}
\EG = -e\left(\gamma - 1 + \frac{1}{2 \cdot 2!} - \frac{1}{3 \cdot 3!} + \frac{1}{4 \cdot 4!} + \cdots\right).
\end{equation}
One can also deduce from \eqn{560aa} that
\begin{equation}\label{561a}
\EG = - e \, Ei(-1)
\end{equation}
where $Ei(x)$ is the {\em exponential integral} 
\begin{equation}\label{561b}
Ei(x) := \int_{-\infty}^x  \frac{e^t}{t}dt,
\end{equation}
which is defined as a function of a complex variable $x$ on
the complex plane cut along the nonnegative real axis.\footnote{This is sufficient
to define $Ei(-1).$ For applications the  function $Ei(x)$ is usually extended to positive
real values by using the Cauchy principal value at the point $x=0$
where the integrand has a singularity. The real-valued definition 
using the Cauchy principal value is related to the complex-analytic extension by
$$
Ei(x) := \lim_{\epsilon \to 0^{+}} \frac{1}{2} \Big(Ei(x+ \epsilon i) +  Ei(x- \epsilon i)\Big).
$$.
}

Here we  note that Euler's integral \eqn{557a}
is transformed by  the change of variable $t=\frac{1}{u}$ to
another simple form, which on replacing $x$ by $1/x$ becomes
\begin{equation}\label{562a}
 s(\frac{1}{x}) = e^{x} \int_{x}^{\infty} \frac{e^{-u}}{u} du = e^x E_1(x).
 \end{equation}
 Here  $E_1(z)$ is the {\em principal exponential integral}, defined for $|\arg z| < \pi$
 by a contour integral 
 \begin{equation}\label{expint}
 E_1(z) := \int_{z}^{\infty} \frac{e^{-u}}{u} du,
 \end{equation}
 that eventually goes to $+\infty$ along the positive real axis, see \cite[Chap. 5]{AS}.
 Choosing $x=1$ yields the identity
 \begin{equation}\label{249a}
 E_{1}(1) =  \int_{1}^{\infty} \frac{e^{-t}}{t} dt = \frac{\delta}{e}.
 \end{equation}
Further formulas involving the Euler-Gompertz constant are given
 in Sections \ref{sec34},  \ref{sec35},  \ref{sec311} and  \ref{sec312}.

Regarding method (4) above, Euler had previously done much work on continued fractions
in the late 1730's  in \cite[E71]{E71} and in particular  in his sequel paper \cite[E123]{E123}, in which
he  expands some analytic functions and integrals in continued fractions. This  work of Euler is discussed in the book of 
Khrushchev \cite{Khr08}, which includes an English translation of \cite{E123}. 
 In 1879 E. Laguerre \cite{Laguerre1879} found an analytic continued fraction expansion 
 for the principal exponential integral,  which  coincides with  $e^{-x} s(1/x)$by \eqn{562a}, obtaining
 \begin{equation}\label{249c}
\int_{x}^{\infty} \frac{e^{-t}}{t} dt= \cfrac{e^{-x}}{x+1- \cfrac{1}{x+3-\cfrac{1}{\frac{x+5}{4}-  \cfrac{\frac{1}{4}}{\frac{x+7}{9}-\cfrac{\frac{1}{9}}{\frac{x+9}{16}- \cdots}}}}}.
\end{equation}
Laguerre noted that this expansion converges for positive real $x$.
His expansion simplifies to\footnote{This form of  fraction with minus signs between convergents
is  often  called a {\em $J$-continued fraction}, see Wall \cite[p. 103]{Wall48}.}
\begin{equation}\label{249d}
\int_{x}^{\infty} \frac{e^{-t}}{t} dt = 
 \cfrac{e^{-x}}{x+1- \cfrac{1^2}{x+3-\cfrac{2^2}{x+5-  \cfrac{3^2}{x+7-\cfrac{4^2}{x+9- \cdots}}}}},
\end{equation}
which is the form in which it appears in Wall \cite[(92.7)]{Wall48}. 
On taking $x=1$ it yields another continued fraction for the Euler-Gompertz constant, 
\begin{equation}\label{249e}
\delta  = e \int_{1}^{\infty} \frac{e^{-t}}{t} dt =
 \cfrac{1}{2- \cfrac{1^2}{4-\cfrac{2^2}{6-  \cfrac{3^2}{8-\cfrac{4^2}{10- \cdots}}}}}. 
 \end{equation}

In his great 1894 work 
T. J. Stieltjes (\cite{Sti1894}, \cite{Sti93}) 
developed a general theory  for analytic continued fractions representing
functions $F(x)$ given by 
$\int_{0}^{\infty} \frac{d \Phi(u)}{x+ u}$, which takes the form
\begin{equation}\label{249f}
\int_{0}^{\infty} \frac{d\Phi(u)}{x+ u}  =
\cfrac{1}{a_1 x+ \cfrac{1}{a_2+ \cfrac{1}{a_3 x +  \cfrac{1}{a_4+\cfrac{1}{a_5 x+\cfrac{1}{a_6+ \cdots}}}}}}.
\end{equation}
for positive real $a_n$. He  formulated the  (Stieltjes)  moment problem for measures on the nonnegative real axis,
and related his continued fractions to the moment problem.
In the introduction of his paper he  specifically gives  Laguerre's continued fraction \eqn{249d}
as an example, and observes that
it may be converted to a continued fraction in his form \eqn{249f} with  $a_{2n-1}=1$ and $a_{2n}= \frac{1}{n}.$ 
On taking $x=1$ his converted fraction is closely related to  Euler's continued fraction \eqn{eulerfrac}.
Stieltjes notes that his theory extends the provable convergence region
for Laguerre's continued fraction \eqn{249d} to all complex $x$, excluding
 the negative real axis.
In  \cite[Sect. 57]{Sti1894} he remarks that associated measure in
the moment problem for this continued fraction must be unbounded, and in 
 \cite[Sect. 62]{Sti1894} he  obtains a generalization of
 (his form of)  Laguerre's  continued fraction to add extra parameters;
cf. Wall \cite[eqn. (92.17)]{Wall48}.

It is not known whether the  Euler-Gompertz constant is rational or irrational. 
However it has recently been established that at least one of $\gamma$ and $\EG$ is transcendental,
as presented in  Section \ref{sec312}.

\subsection{Euler-Mascheroni constant; Euler's approach to research }\label{sec26}
\setcounter{equation}{0}

Euler's constant is often called the {\em Euler-Mascheroni constant,}
after  later work of Lorenzo Mascheroni \cite{Mas1790}, \cite{Mas1792}.
In 1790, in a book proposing to answer some
problems raised by Euler, Mascheroni \cite{Mas1790}
 introduced the notation $A$
for this constant. In 1792  Mascheroni  \cite[p.11]{Mas1792} also
reconsidered integrals in Euler's paper \cite[E629]{E629},
using  both the  notations $A$ and $a$  for Euler's constant.
In his book, using an expansion of the logarithmic integral,
 he gave a numerical expression for  it to $32$ decimal places
(\cite[p. 23]{Mas1790}), which stated   
$$
A \approx 0. 57721~56649~01532~8606\underline{1}~ 81120~ 90082~39
$$
This computation attached
his name to the problem; Euler had previously computed it accurately to $15$ places.
At that time  constants were often named after a  person who had done the labor of computing
them to the most digits. 
 However in  1809  Johann von Soldner \cite{vS1809},
 while preparing  the first book on the logarithmic integral function $li(x) = \int \frac{dx}{\log x}$, 
found it necessary to recompute Euler's constant, which he called $H$ (perhaps from
the harmonic series).
He computed it to $22$ decimal places by a similar method, 
and  obtained (\cite[p. 13]{vS1809})
\[
H \approx 0. 57721~56649~01532 ~86060~65
\]
Von Soldner's answer 
disagreed with Mascheroni's in the $20$-th decimal place. To settle the conflict
von Soldner asked the aid of C. F. Gauss, and Gauss engaged the
19-year old calculating prodigy F. G. B. Nicolai (1793---1846) to
recheck the computation. Nicolai used the Euler-Maclaurin summation
formula to compute $\gamma$ to $40$ places, finding agreement
with von Soldner's calculations to $22$ places (see \cite[p. 89]{Hav03}).
Thus the conflict was settled in von Soldner's favor, and  Mascheroni's value is  accurate
to  only $19$ places. 
In the current era, which credits the mathematical  contributions,
it seems most appropriate to name the constant after Euler alone.

The now-standard notation $\gamma$ for Euler's constant 
arose subsequent to the work of both Euler and Mascheroni. 
 Euler himself used various notations for his constant, including $C$ and $O$ and $n$;
Mascheroni used $A$ and $a$; von Soldner used $H$ and $C$.
The  notation $\gamma$ appears  in
an 1837  paper on the logarithmic integral of Bretschneider \cite[p. 260]{Bret1837}.
It also appears in a   calculus textbook
of Augustus de Morgan \cite[p. 578]{ADM}, published about the same time. 
This choice of symbol $\gamma$ is perhaps based on the association of  Euler's constant with the gamma function.

To conclude our discussion of Euler's work,  
it seems useful  to  contemplate the 
approach  to research used by
Euler in his long successful career. C. Truesdell \cite[Essay 10]{Tru84}, \cite[pp. 91--92]{Tru87}, 
makes the following observations about 
the  methods used  by  Euler,
 his teacher Johann Bernoulli, and  Johann's teacher and brother
Jacob  Bernoulli.
\begin{enumerate}
\item
Always attack a special problem. If possible solve the
special problem in a way that leads to a general method.
\item
Read and digest every earlier attempt at a theory of
the phenomenon in question.
\item
Let a key problem solved be a father to a key problem posed.
The new problem finds its place on the structure provided
by the solution of the old; its solution in turn will provide
further structure.
\item
If two special problems solved seem cognate, try to unite
them in a general scheme. To do so, set aside the differences,
and try to build a structure on the common features.
\item
Never rest content with an imperfect or incomplete argument.
If you cannot complete and perfect it yourself, lay bare its
flaws for others to see.
\item 
Never abandon a problem you have solved. There are
always better ways. Keep searching for them, for they
lead to a fuller understanding. While broadening,
deepen and simplify.
\end{enumerate}
 Truesdell speaks here about  Euler's work  in the area of foundations of mechanics,
 but his  observations apply equally well to 
 Euler's work  in number theory and analysis discussed above.

%
%
%
%

\section{Mathematical Developments}\label{sec3}
\setcounter{equation}{0}

We now consider a collection of mathematical topics in which Euler's constant
appears.

%
%
%
%
\subsection{Euler's constant and the gamma function}\label{sec31}
\setcounter{equation}{0}

Euler's constant appears in several places in connection with the gamma function.
 Already Euler noted a basic occurrence of Euler's constant with the gamma function, 
\begin{equation}\label{201}
 \Gamma'(1)= -\gamma.
\end{equation}
Thus $\gamma$ appears in the Taylor series expansion of $\Gamma(z)$ around
the point $z=1$.
Furthermore Euler's constant also appears (in the form $e^{\gamma}$) in the 
 Hadamard product expansion for the entire function $\frac{1}{\Gamma(z)},$ which states:
$$
\frac{1}{\Gamma(z)} = z e^{\gamma z} \prod_{n=1}^{\infty} \left(1 +\frac{z}{n} \right) e^{-\frac{z}{n}}.
$$

Euler's constant  seems  however more tightly associated  to the logarithm of the gamma function,
which was studied by Euler \cite[E368]{E368} and many subsequent workers.
The function  $\log \Gamma(z)$ is  multivalued  under analytic continuation,
and many results are 
more elegantly expressed in terms of the {\em digamma function}, defined by
\begin{equation}\label{300}
\psi(z) := \frac{\Gamma'(z)}{\Gamma(z)} = \frac{d}{dz} \log \Gamma(z).
\end{equation}
The theory of the digamma function is simpler in some respects because it is a meromorphic
function. Nielsen \cite[Chap. 1]{Nie06} presents results and some early history, 
and Ferraro \cite{Fer07} describes the work of Gauss on these functions.

We recall a few basic facts.
 The digamma function  satisfies the difference equation
\begin{equation}\label{300b}
\psi(z+1) = \psi(z) + \frac{1}{z},
\end{equation}
which is inherited from the functional equation $z\Gamma(z)= \Gamma(z+1)$. 
It has an  integral representation, valid for $Re(z) >0$,  
\begin{equation*}
\psi(z) = \int_{0}^{\infty} \Big( \frac{e^{-t}}{t} - \frac{e^{-zt}}{1- e^{-t}}\Big) dt,
\end{equation*}
 found in 1813 by Gauss \cite[Art. 30 ff]{Gau1813} as part
of his investigation of the hypergeometric function ${}_2 F_{1}(\alpha, \beta; \gamma; x)$ (see \cite{Fer07}).
The difference equation  gives
\[
\psi(x) = \psi(x+n) -  \sum_{j=0}^{n-1} \frac{1}{x+j}
\]
from which follows 
\begin{equation*}
\psi(x) = \lim_{n \to \infty} \Big( \log n - \sum_{j=0}^{n-1} \frac{1}{x+j}\Big).
\end{equation*}
This  limit is  valid for all $x \in \CC \smallsetminus \ZZ_{\le 0}$, using the fact that
$ \psi(x) - \log x$ goes uniformly to $0$ in the right half-plane as $Re(x) \to \infty$
(\cite[p. 85, 93]{Nie06}).
This formula  gives $\psi(1) = -\gamma,$
a fact which also follows directly from \eqn{201}.

 Euler's constant
appears in the  values of the digamma function $\psi(x)$ at all positive integers
and  half-integers.  In 1765 Euler \cite[E368]{E368} found the values
$\Gamma(\frac{1}{2}), \Gamma(1)$ and $\Gamma'(\frac{1}{2}), \Gamma'(1)$.
These values determines  $\psi(1/2), \psi(1)$, and then, by  using 
the difference equation \eqn{300b},
the values of  $\psi(x)$ for all half-integers.

\smallskip
\begin{theorem} ~\label{th30} {\em (Euler 1765)} 
 The digamma function $\psi(x) := \frac{\Gamma'(x)}{\Gamma(x)} $
 satisfies, for integers $n \ge 1$,
\begin{equation*}
\psi(n)=   -\gamma + H_{n-1}
\end{equation*}
using the convention that $H_0=0$.
It also satisfies, for half integers $n + \frac{1}{2}$, and $n \ge 0$,
\begin{equation*}
\psi(n+ \frac{1}{2})=  -\gamma - 2\log 2+ 2 H_{2n-1} -  H_{n-1},
\end{equation*}
using the additional convention $H_{-1}=0.$
 \end{theorem}

\begin{proof} The first  formula  follows from $\psi(1)=-\gamma$ using
 \eqn{300b}. The second formula will follow from the identity
\begin{equation}\label{302c}
\psi(\frac{1}{2}) = -\gamma - 2 \log 2,
\end{equation}
because the  difference equation \eqn{300b} then gives
$$
\psi(n+ \frac{1}{2}) = \psi(\frac{1}{2}) + 2\left( 1+ \frac{1}{3} + \cdots + \frac{1}{2n-1}\right).
$$
To obtain \eqn{302c}  we start from the duplication formula for  the gamma function
\begin{equation}\label{307a}
\Gamma(2z) =
 \frac{1}{\sqrt{ \pi}}  2^{2z- 1}\Gamma(z)\Gamma(z+ \frac{1}{2}).
\end{equation}
 We logarithmically differentiate this formula at $z= \frac{1}{2},$
  and solve for $\psi(\frac{1}{2}).$ 
  \end{proof}

The next two results present  expansions of the digamma function 
that connect it with integer zeta values. 
The first one  concerns its Taylor series expansion around the point $z=1$.

\bigskip
\begin{theorem}~\label{th31} {\em (Euler 1765)}
The digamma function has  Taylor series expansion around $z=1$ given by
\begin{equation}\label{311a}
\psi(z+1) = -\gamma + \sum_{k=1}^{\infty} (-1)^{k+1} \zeta(k+1) \,z^k
\end{equation}
This expansion converges absolutely  for $|z| < 1$. 
\end{theorem}

\paragraph{\bf Remark.}
The  pattern of Taylor
coefficients in \eqn{311a} has the $k=0$ term formally assigning to $``\zeta(1)"$
the value $\gamma$.\medskip

\begin{proof}
 Euler \cite[E368, Sect. 9]{E368} obtained a  Taylor series expansion  around $z=1$
for $\log \Gamma(z)$, given as \eqn{271} (with $x=z-1$).   Differentiating it term by term 
with respect to $x$ yields \eqn{311a}.
\end{proof}

The second relation to zeta values concerns the asymptotic expansion of the digamma
function around $z= +\infty$, valid in the right half-plane. 
In this formula  the non-positive integer values of the zeta function appear.

\bigskip
\begin{theorem}~\label{th32}  
 The digamma function $\psi(z)$ has the asymptotic expansion
\begin{equation}\label{321}
\psi(z+1) \sim  \log z   +
\sum_{k=1}^{\infty} (-1)^{k} \zeta(1-k)\left( \frac{1}{z}\right)^k, 
\end{equation}
which is valid  on the sector $-\pi +\epsilon \le \arg z \le \pi - \epsilon$,
for  any given $\epsilon>0$.
This asymptotic expansion does not converge for any value of $z$.
\end{theorem}

The statement that  \eqn{321} is an 
 asymptotic expansion means: For  each finite $n \ge 1$,   the
estimate
\begin{equation*}
\psi(z+1) =  \log z   +
\sum_{k=1}^{n}  (-1)^{k} \zeta(1-k)\left( \frac{1}{z}\right)^k  + O ( z^{-n-1}) 
\end{equation*}
holds on the sector 
$-\pi + \epsilon \le  \arg{z} \le  \pi - \epsilon$, with $\epsilon >0$ 
where the implied constant in
the $O$-symbol depends on both $n$ and $\epsilon$.

Theorem ~\ref{th32} does not seem to have been previously stated in this form, with 
coefficients  containing $\zeta(1-k)$. However it is quite an old result, 
stated in the alternate form
 \begin{equation}\label{323cc}
\psi(z+1) \sim \log z  - \frac{B_1}{z} - \sum_{k=1}^{\infty}  \frac{B_{2k}}{2k}\frac{1}{z^{2k}}.
\end{equation}
The conversion from \eqn{323cc} to \eqn{321} is simply a matter of substituting
the formulas for the zeta values at negative integers: $\zeta(0) = -\frac{1}{2}=B_1$
and $\zeta(1-k) = (-1)^{k+1} \frac{B_{k}}{k}$, noting $B_{2k+1}=0$ for $k \ge 1$, see
 \cite[Sec. 1.18]{EMOT53}, \cite[(6.3.18)]{AS}.
As a  {\em formal} expansion  \eqn{323cc} follows  either from Stirling's expansion of
sums of logarithms obtained in  1730 (\cite[Prop. 28]{Stirling1730},
 c.f. Tweddle \cite[Chap. 1]{Twe88}),
   or   from Euler's expansion of $\log \Gamma(z)$ given in \eqn{stirling}, by
differentiation  term-by-term. Its rigorous interpretation  as an   asymptotic expansion 
having a remainder term estimate valid on the large region given above   is due to Stieltjes.
In his 1886 French doctoral thesis  
Stieltjes \cite{Sti1886}  treated a general notion of  an asymptotic expansion with remainder term, including   the case of  an asymptotic expansion for  $\log \Gamma(a)$ (\cite[Sec. 14--19]{Sti1886}). In 1889 
Stieltjes \cite{Sti1889}   presented a detailed  derivation of an asymptotic series
for $\log \Gamma(a)$, whose methods can be used
to  derive the result \eqn{323cc} above.
\medskip

\begin{proof}
 The 1889  formula of Stieltjes \cite{Sti1889}  for $\log \Gamma(a)$ embodys Stirling's
expansion with an error estimate.  Stieltjes first introduced the function $J(a)$ defined
by
$$
\log \Gamma(a) = (a- \frac{1}{2}) \log a - a + \frac{1}{2} \log (2 \pi) + J(a).
$$
His key observation is  that $J(a)$ has an integral representation
\begin{equation}\label{3002a}
J(a) = \int_{0}^{\infty} \frac{P(x)}{x+a} dx,
\end{equation}
in which $P(x)$ is the periodic function 
\[
P(x) = \frac{1}{2} -x + \lfloor x\rfloor = -B_1(x-\lfloor x\rfloor), 
\] 
where $B_1(x)$
is the first Bernoulli polynomial. He
observes that the integral formula \eqref{3002a} has the merit that it 
  converges on
the complex $a$-plane aside from the negative real axis, unlike two earlier 
integral formulas for 
 $J(a)$ of  Binet,
\[
J(a) = \int_{0}^{\infty} \Big( \frac{1}{1-e^{-x}} - \frac{1}{x} - \frac{1}{2} \Big) \frac{e^{-ax}}{x} dx,
\]
and
\[
J(a) = \frac{1}{\pi} \int_{0}^{\infty} \frac{a}{a^2 + x^2} \log \Big(\frac{1}{1-e^{-2 \pi x}}\Big) dx,
\]
which are valid on the  half-plane  $Re(a)>0$. He now repeatedly integrates \eqn{3002a} by parts to
obtain  a formula \footnote{In Stieltjes's formula $B_j$ denotes
 the Bernoulli number $|B_{2j}|$ in  \eqn{221}. We have altered his notation to match \eqn{221}.}
  for $\log \Gamma(a)$,
valid on the entire complex plane aside from the negative real axis, 
\begin{equation}\label{3001}
\log \Gamma(a) = (a -\frac{1}{2} )\log a - a + \frac{1}{2} \log (2 \pi)
+ \frac{|B_2|}{1 \cdot 2 a} 
-\frac{|B_4|}{3 \cdot 4 a^3} 
+ \cdots +  \frac{ (-1)^{k-1}  |B_{2k}|}{(2k-1) 2k \, a^{2k-1}} + J_k(a),
\end{equation}
in which  $J_k(a) = \int_{0}^{\infty} \frac{P_k(x)}{(x+a)^{2k+1}} dx$ where $P_k(x)$  is
periodic and
given by its Fourier series
$$
P_k(x) = (-1)^k (2k)! \sum_{n=1}^{\infty} \frac{1}{2^{2k} (\pi n)^{2k+1}} \sin (2 \pi n x),
$$
and is  related to the $(2k+1)$-st Bernoulli polynomial.
He obtains an error estimate for $J_k(a)$  (\cite[eqn. (34)]{Sti1889}), 
valid for $-\pi + \epsilon < \theta< \pi - \epsilon$, that
$$
|J_k(Re^{ i\theta})| < \frac{|B_{2k+2}|}{(2k+2)(2k+1)} (\sec \frac{1}{2} \theta)^{2k+2} R^{-2k-1}.
$$
We obtain a formula for $\psi(a)$ by
differentiating  \eqn{3001}, and using $B_{2k} = (-1)^{k-1} |B_{2k}|$ gives
$$
\psi (a) = \log a - \frac{1}{2a} -\frac{B_2}{2\, a^2} - \frac{B_4}{4 \, a^4} - \cdots - \frac{B_{2k}}{2k \, a^{2k}}+ J_k^{'}(a).
$$
This formula gives the required asymptotic expansion for $\psi(a)$, as soon as one obtains a suitable error estimate for
$J_k^{'}(a)$, which may be done on the same lines as for $J_k(a)$.
The asymptotic expansion \eqn{323cc} follows after using 
$\psi(a+1) = \psi(a) + \frac{1}{a}$, which has the effect
to reverse the sign of the  coefficient
of $\frac{1}{a}$. 

The formula \eqn{321} follows by direct
 substitution of the zeta  values $\zeta(0)= -\frac{1}{2}=B_1$ and $\zeta(1-k)= (-1)^{k+1}\frac{B_{k}}{k}$ for $k \ge 2$,
noting $B_{2k+1} =0$ for $k \ge 1$, see \cite[Sec. 1.18]{EMOT53}, \cite[(6.3.18)]{AS}.

Finally, the known growth rate for  even Bernoulli numbers,
$$
|B_{2n}| \sim \frac{2}{(2\pi)^{2n}} (2n)! \sim 4 \sqrt{\pi n} (\pi e)^{-2n} n^{2n}.
$$
implies that the asymptotic series \eqn{321} diverges for all finite values of $z$.
\end{proof}

A consequence of the  asymptotic expansion \eqn{321} for $\psi(x)$ is  an asymptotic expansion
for the harmonic numbers, given by 
\begin{equation}\label{322b}
H_n  \sim \log n  + \gamma+ \sum_{k=1}^{\infty} (-1)^k \zeta(1-k) \frac{1}{n^k}.
\end{equation}
This is obtained on taking $z=n$ in  \eqn{311a} and substituting $\psi(n+1) = H_n -\gamma$.
Substituting the zeta values  allows this expansion to be expressed in the more well-known form
\begin{eqnarray*}
H_n  & \sim &  
\log n +\gamma - \frac{B_1}{n} - \sum_{k=2}^{\infty}  \frac{B_{2k}}{2k}\frac{1}{n^{2k}}, \nonumber\\
& \sim & \log n + \gamma + \frac{1}{2n} - \frac{1}{12n^2} + \frac{1}{120 n^4} - \frac{1}{252n^6} + \cdots
\end{eqnarray*}
Note that although  
$H_n$ has jumps of size  $\frac{1}{n}$ at each integer value,  nevertheless  the asymptotic
expansion  is valid to all orders $(\frac{1}{n})^k$ for $k \ge 1$, although the higher order terms are
of magnitude much smaller than the jumps. 

Theorem \ref{th31} and Theorem \ref{th32} (2) have
   the feature that  these two expansions 
  exhibit between them the full set of integer values 
 $\zeta(k)$ of
the Riemann zeta function,
with Euler's constant appearing as a substitute for $``\zeta(1)"$.
This is a remarkable fact.

%
%
%
%
\subsection{Euler's constant and the   zeta  function}\label{sec32}
\setcounter{equation}{0}

The function $\zeta(s)$ was studied at  length by Euler, but is named after 
 Riemann \cite{Riemann1859}, commemorating his epoch-making 1859 paper.
 Riemann showed that $\zeta(s)$   extends to a meromorphic function in the plane with its
 only singularity being a simple pole at $s=1$.
Euler's constant appears in 
the Laurent series expansion of $\zeta(s)$ around the point $s=1$,
as follows.\smallskip


 \begin{theorem}~\label{th41} {\em (Stieltjes 1885)}
 The  Laurent expansion of $\zeta(s)$ around
 $s=1$ has the form
 \begin{equation}\label{401}
 \zeta(s) = \frac{1}{s-1} + \gamma_0 + \sum_{n=1}^{\infty}  \frac{(-1)^n}{n!} 
 \gamma_n (s-1)^n
 \end{equation}
 in which $\gamma_0=\gamma$ is Euler's constant, and the coefficients
  $\gamma_n$ for $n \ge 0$ are  defined by
 \begin{equation}\label{402}
 \gamma_n := 
  \lim_{m \to \infty} \left( \sum_{k=1}^m \frac{(\log k)^n}{k} - \frac{(\log m)^{n+1}}{n+1} \right)
 \end{equation}
 \end{theorem}

\begin{proof}
We follow a proof  of  Bohman and Fr\"{o}berg \cite{BF88}.
It starts from the identity\\
 $(s-1) \zeta(s) = \sum_{k=1}^{\infty} (s-1)k^{-s},$ 
 which is valid for $Re(s)>1$.  For real $s>1$, one has 
 the telescoping sum 
$$
\sum_{k=1}^{\infty} (k^{1-s} - (k+1)^{1-s})=1.
$$
Using it, one  obtains 
\[
(s-1)\zeta(s) = 1 + \sum_{k=1}^{\infty} \{ (k+1)^{1-s} -k^{1-s} + (s-1)k^{-s} \}.
\]
Noting that at  $s=1$ the telescoping sum above  is $0$, one has
\begin{eqnarray*}
(s-1) \zeta(s) &=& 1 + \sum_{k=1}^{\infty} \{ \exp(- (s-1) \log (k+1)) - \exp( -(s-1)\log k)
 \\&& 
\quad \quad\quad\quad\quad
 + (s-1) \frac{1}{k} \exp(- (s-1) \log k)\} \\
& = & 1 + 
\sum_{k=1}^{\infty} \Big\{ \sum_{n=0}^{\infty} 
\frac{(-1)^n (s-1)^n}{n!} [ (\log(k+1))^n - (\log k)^n] \\
&& \quad\quad\quad \quad\quad
+ \frac{s-1}{k} \sum_{n=0}^{\infty} \frac{(-1)^n (s-1)^n(\log k)^n}{n!} \Big\}.
\end{eqnarray*}
Dividing by $(s-1)$ now yields
\[ 
\zeta(s) = \frac{1}{s-1} +  \sum_{n=0}^{\infty} \frac{(-1)^n }{n!}\gamma_n (s-1)^n,
\]
with $\gamma_0= \gamma$, and with
\begin{equation}\label{402a}
\gamma_n = \sum_{k=1}^{\infty} \Big\{ \frac{(\log k)^n}{k} - \frac{ (\log (k+1))^{n+1} - (\log k)^{n+1}}{n+1} \Big\}.
\end{equation}
This last formula  is  equivalent to  \eqn{402}, since the first $m$ terms in the sum
add up to the $m$-th term in \eqn{402}.
\end{proof}

   Theorem~\ref{th41} gives  the  coefficients 
 $\gamma_n$ of the zeta function   in a form resembling the   original definition of 
 Euler's constant.     
These constants arose in  a  correspondence that Stieltjes carried
on with his mentor Hermite (cf. \cite{HS05} )
 from 1882  to  1894.
   In Letter 71 of that correspondence, written in June 1885, Stieltjes stated that  he proposed  to calculate
 the first few terms in the Laurent expansion\footnote{In Letter 71 Stieltjes uses the notation $C_0, C_1, ...$,
 but in Letter 75 he uses the same notation with a different meaning. We substitute $A_0, A_1, ...$ for
 the Letter 71 coefficients
 following \cite{BF88}.}  of $\zeta(s)$
 around $s=1$, 
 $$ 
 \zeta(z+1) = \frac{1}{z} + A_0 + A_1 \,z +...
 $$
 noting that $A_0= 0.57721 5665...$ is known and that
he finds $A_1= -0.07281 5\underline{5}20...$
 In Letters 73 and  74,  Hermite  formulated \eqn{402}, 
  which he called ``very interesting," and inquired about rigorously establishing it.\footnote{The result \eqn{402}
   is misstated in Erd\"{e}lyi et al \cite[(1.12.17)]{EMOT53},
making the right side equal  $\frac{(-1)^n}{n!} \gamma_n$, 
a mistake tracing back to a paper of G. H. Hardy \cite{Ha12}.} 
 In  Letter 75 Stieltjes \cite[pp. 151--155]{HS05}  gave a complete derivation of  
 \eqn{402}, writing
 \[
 \zeta(s+1) = \frac{1}{s} + C_0 - C_1 \,s + \frac{C_2}{1\cdot 2} \,s^2 - \frac{C_3}{1\cdot 2 \cdot 3}\, x^3+ \cdots
 \]
 The Laurent series coefficients $A_n$ are now named the  Stieltjes constants 
 by some authors who follow Letter 71 (\cite{BF88}, \cite{BL99}),  while other
 authors  follow Letter 75 and call the scaled values $\gamma_n= C_n$ the {\em Stieltjes constants} (\cite{Ke92}, \cite{Cof06}).

%
%
%

 \begin{table}\centering
\renewcommand{\arraystretch}{.85}
\begin{tabular}{|r|r|
}
\hline
\multicolumn{1}{|c|}{$n$} &
\multicolumn{1}{c|}{$\gamma_n$}\\ \hline
0 &     $+$ 0.57721 ~56649 ~ 01532\\ 
 1 &   $-$ 0.07281~ 58454 ~ 83676\\ 
 2 &    $-$ 0.00969~ 03631 ~28723     \\
 3 &     $+$ 0.00205 ~38344~ 20303 \\ 
 4 &      $+ $ 0.00232 ~53700 ~65467\\ 
 5 &     $+ $ 0.00079 ~33238~ 17301 \\ 
 6 &      $-$ 0.00023 ~87693~ 45430 \\ 
 7 &      $-$ 0.00052  ~72895 ~67057 \\ 
 8 &       $-$ 0.00035 ~21233 ~53803 \\ 
 9 &       $-$ 0.00003 ~43947~ 74418\\ 
10 &     $+ $ 0.00020~53328~14909  \\ \hline
\end{tabular}
\bigskip

\noindent \caption{Values for Stieltjes constants $\gamma_{n}$. }
\label{tab31}
\end{table}

 Table \ref{tab31} presents values of $\gamma_n$
 for small $n$ as computed by Keiper \cite{Ke92}.
  A good deal is  known about the size of these constants.
 There are infinitely many $\gamma_n$ of each sign (Briggs \cite{Br55}). 
 Because $\zeta(s) - \frac{1}{s-1}$ is an entire function, the 
 Laurent coefficients   $A_n$
 must rapidly approach $0$ in absolute value as $n \to \infty$. However 
despite the sizes of  the initial values in Table \ref{tab31},  the scaled values $\gamma_n$ 
 are known to be unbounded (and to sometimes be  exponentially large) as a function of $n$, cf. Cohen  \cite[Sec. 10.3.5]{Coh07b}.
 Keiper estimated that $\gamma_{150}\approx 8.02 \times 10^{35}.$
 More detailed asymptotic bounds  on the $\gamma_n$ are given in Coffey \cite{Cof06} and 
 Knessl and Coffey \cite{KC11}, \cite{KC11b}; see also Coffey \cite{Cof10}.
 At present nothing  is known about the  irrationality or transcendence of 
 any of  the Stieltjes constants.

Various generalizations of the Riemann zeta function  play
a role in theoretical physics, especially in quantum field theory, 
in connection with  a method of 
``zeta function regularization" for assigning  finite values in certain perturbation
calculations. This method was  suggested by S. Hawking \cite{Haw77} in 1977, and
it has now  been widely extended to other questions in physics,
see the book of Elizalde \cite{Elizalde95}  (See also \cite{EOR94},  \cite{EVZ98}.)
Here one  forms a generalized  ``zeta function" from 
spectral data and evaluates it an an integer point, often $s=0$. If the function
is meromorphic at this point, then the constant term in the Laurent expansion at this
value is taken as the ``regularized" value.  This procedure can be viewed as a quite special kind of
``dimensional regularization," in which
perturbation calculations are made viewing  the dimension of space-time as a
complex variable $s$, and at the end of the calculation one specializes the
complex variable to the integer value of the
dimension, often  $s=4$, see Leibbrandt \cite{Lei75}.
Under this convention 
Euler's constant   $\gamma$
 functions as a ``regularized" value $``\zeta(1)"$ of the zeta function at $s=1$,
 because by Theorem \ref{th41} 
 it is the constant term in the Laurent expansion of the
 Riemann zeta function at $s=1$.
  This interpretation as ``$\zeta(1)"$ also matches 
 the constant term in  the Taylor  expansion of the digamma function $\psi(z)$
 around $z=1$  in Theorem~\ref{th31}.
 
 Euler's constant occurs in many deeper ways in connection with the behavior
 of the Riemann zeta function.
  The most tantalizing open problem about the
 Riemann zeta function is the {\em Riemann hypothesis},
 which asserts that the non-real complex  zeros of the Riemann zeta function
 all fall on the line $\text{Re}(s) = \frac{1}{2}$. 
  Sections  \ref{sec36} and  \ref{sec37}  present  formulations of the Riemann hypothesis
 in terms of the relation of  limiting asymptotic behaviors of arithmetic functions to  $e^{\gamma}$. 
   Section
   \ref{sec37a}  presents results relating  the extreme value distribution of the Riemann zeta function
 on the line $\text{Re}(s)=1$ to the constant $e^{\gamma}$.
 
 In 1885 Stieltjes announced a proof of the Riemann hypothesis
 (\cite{Sti1885}, English translation in \cite[p. 561]{Sti93}).
 In Letter 71  to Hermite discussed above,  Stieltjes enclosed
his announcement of   a proof of the Riemann hypothesis, requesting
 that it be  communicated  to   Comptes Rendus, if Hermite approved,
 and Hermite did so.  
 In Letter 79, written later in 1885, Stieltjes sketched his approach,  saying that  his proof 
 was very complicated and that he hoped to simplify it.
 Stieltjes  considered the Mertens function 
 $$
 M(n) =\sum_{j=1}^n \mu(j),
 $$
  in which $\mu(j)$ is the M\"{o}bius function
 (defined in Section \ref{sec33}) 
 and asserted that he could prove  that $|M(n)|/\sqrt{n}$ is a bounded function,
 a result which would certainly imply the Riemann hypothesis.
 Stieltjes had further correspondence in 1887 with Mittag-Leffler, also  making 
 the assertion that he could prove  $\frac{M(n)}{\sqrt{n}}$ is bounded between two limits,
 see \cite[Tome II, Appendix]{HS05}, \cite[Sect. 2]{tR93}.
He additionally  made  numerical calculations\footnote{Such calculations
 were found among his posthumous possessions, according to  te Riele \cite[p. 69]{tR93}.}
  of the Mertens function $M(n)$ for ranges $1 \le n \le 1200$, $2000 \le n \le  2100$, 
 and  $6000\le n \le 7000$, verifying it was small over these ranges.
 However he  never produced a written proof of the Riemann hypothesis
 related to this announcement.  
  Modern work relating the Riemann zeta zeros and random matrix theory now
 indicates  that 
  $|M(n)|/\sqrt{n}$ should be  an unbounded function, although this is unproved, 
 see te Riele \cite[Sect. 4]{tR93} and 
  Edwards \cite[Sect. 12.1]{Ed74}.

%
%
%
%
\subsection{Euler's constant and prime numbers} \label{sec33}
\setcounter{equation}{0}

 Euler's constant appears in various statistics associated to prime 
 numbers. 
 In 1874 Franz Mertens \cite{Mertens1874} determined the 
 growth behavior of the sum of reciprocals of prime numbers,
 as follows.

\begin{theorem}\label{th50}
{\em (Mertens's Sum Theorem 1874)}
For $x \ge 2$,  
\begin{equation*}
\sum_{p \le x} \frac{1}{p} = \log\log x + B +R(x),
\end{equation*}
with $B$ being a constant
\begin{equation}\label{500ab}
B = \gamma + \sum_{m=2}^{\infty} \mu(m) \frac{\log \zeta(m)}{m}\approx 0.26149\, 72129 \cdots
\end{equation}
and with remainder term bounded by
\begin{equation*}
|R(x)| \le \frac{4}{\log (P +1)} + \frac{2}{P \log P},
\end{equation*}
where $P$ denotes the largest prime smaller than $x$.
\end{theorem}

Here $\mu(n)$ is the {\em M\"{o}bius function}, which takes the values $\mu(1)=1$ and
$\mu(n)= (-1)^k$  if $n$ is a product of $k$ distinct prime factors, and $\mu(n)=0$ otherwise.
This  result asserts the existence of the constant
\begin{equation*}
B := \lim_{ x\to \infty} \left( \sum_{p \le x} \frac{1}{p} - \log\log x\right),
\end{equation*}
which is then given explicitly by \eqn{500ab}.
Mertens's  proof  deduced that $B= \gamma- H$ with
\begin{equation}\label{500d}
H := - \sum_{p}\{ \log(1- \frac{1}{p}) + \frac{1}{p} \} = 0.31571\, 84519 \dots .
\end{equation}
For modern derivations which determine the constant $B$   see Hardy and Wright \cite[Theorem 428]{HW08} 
or  Tenenbaum and Mendes France \cite[Sect. 1.9]{TM00}.

 Mertens \cite[Sect. 3]{Mertens1874} deduced from this result the  following product theorem,
 involving Euler's constant alone.
\begin{theorem}\label{th51}
{\em (Mertens's Product Theorem 1874)}
One has 
\begin{equation}\label{501a}
\lim_{x \to \infty} ~~(\log x) \prod_{p \le  x} \left(1- \frac{1}{p} \right) = e^{-\gamma}. 
\end{equation}
More precisely, for $x \ge 2$, 
\begin{equation*}
\prod_{p \le  x} \left(1- \frac{1}{p} \right) = \frac{e^{-\gamma+ S(x)}}{\log x}
\end{equation*}
in which
\begin{equation*}
|S(x)| \le \frac{4}{\log (P+1)} + \frac{2}{P \log P} + \frac{1}{2P},
\end{equation*}
where $P$ denotes the largest prime smaller than $x$.  Thus $S(x) \to 0$ as $x \to \infty$.
\end{theorem}

In this formula the constant $e^{-\gamma}$ encodes a property 
of the entire ensemble of primes in the following sense:   if one could vary a single prime in
the product above, holding the remainder
fixed, then the limit on the right side of \eqn{501a} would change.

In Mertens's  product theorem   the  finite product 
\begin{equation}\label{504b}
D(x) := \prod_{p \le x}\left(1- \frac{1}{p} \right)
\end{equation}
is the inverse of the partial Euler product 
for the Riemann zeta function, taken over all primes below $x$, 
evaluated at the point $s=1$. The number $D(x)$ gives 
 the limiting density of integers  having no prime factor smaller than $x$,
 taken on the interval $[1, T]$ and  letting
  $T \to \infty$ while holding $x$ fixed.    Now the set of
  integers below  $T= x^2$  having no prime smaller than $x$ consists  of exactly
the primes between $x$ and $x^2$.
Letting $\pi(x)$ count the number of primes below $x$, the prime number
theorem gives $\pi(x) \sim \frac{x}{\log x}$, and this yields the asymptotic formula. 
$$
\pi(x^2)- \pi(x)  \sim  \frac{x^2}{\log x^2} - \frac{x}{\log x} \sim  
\frac{1}{2} \frac{x^2}{\log x}.
$$
On the other hand, if the 
density $D(x)$ in Merten's theorem were asymptotically correct already at $T= x^2$ then 
it  would predict the
 number of such primes to be about $e^{-\gamma} \frac{x^2}{\log x}$.
The fact that $e^{-\gamma} = 0.561145 ...$ does not equal $\frac{1}{2}$, but is
slightly larger,  reflects the failure of  the inclusion-exclusion formula
to give an asymptotic formula in this range. This is a basic difficulty studied by sieve methods.

A subtle consequence of this
failure is that there must be occasional unexpectedly
large fluctuations in the number of primes  in
short intervals away from their expected number. This phenomenon was
discovered by H. Maier \cite{Mai85} in 1985. 
 It is interesting to note that
Maier's argument used a ``double counting" argument in a
tabular array of integers that resembles the  array summations
of Euler and Bernoulli exhibited in section \ref{sec21}.
His argument also used properties of the Buchstab function discussed in
section \ref{sec35}.
See Granville \cite{Gra94} for a detailed discussion of 
this large fluctuation phenomenon.

%
%
%
%
\subsection{Euler's constant and arithmetic functions}\label{sec33a}
\setcounter{equation}{0}

Euler's constant 
also appears in  the  behavior
of three  basic arithmetic functions, 
the divisor function $d(n)$, {Euler's totient function} $\phi(n)$
and the {sum of divisors function} $\sigma(n)$.
For the divisor function $d(n)$ it  arises in its average behavior, while for $\phi(n)$ and
$\sigma(n)$ it concerns extremal asymptotic behavior and there appears
in the form $e^{\gamma}$. 
These extremal behaviors  can be deduced starting  from  Mertens's results.

A basic arithmetic function is 
the  {\em divisor function}
\begin{equation*}
d(n) := \#\{ d:  d|n, ~~~ 1 \le d \le n\}.
\end{equation*}
which counts the number of divisors of $n$.
In 1849 Dirichlet \cite{Dirichlet1849} 
introduced a method to estimate the average size of the divisor function
and other arithmetic functions.  In the statement of Dirichlet's result 
 $\{ x\} := x - [x]$ denotes the fractional part $x ~(\bmod \, 1)$. 

\begin{theorem} \label{th51a}
{\em (Dirichlet 1849)}
(1) The partial sums of  the   divisor function satisfy
\begin{equation}\label{500bb}
 \sum_{k=1}^{n}  d(k) =  n\log n  + (2 \gamma -1)n  + O(\sqrt{n})
\end{equation}
for $1 \le n < \infty$, where $\gamma$ is Euler's constant.

(2) The averages of the fractional parts of $\frac{n}{k}$ satisfy
\begin{equation}\label{500ff}
  \sum_{k=1}^n  \{ \frac{n}{k} \}  = (1-\gamma) n + O (\sqrt{n})
\end{equation}
for  $1 \le n < \infty.$
\end{theorem}

\begin{proof}
Dirichlet's proof starts from the identity
\begin{equation}\label{500gg}
 \sum_{k=1}^n d(k) = \sum_{k=1}^n \lfloor \frac{n}{k} \rfloor,
\end{equation}
which relates the divisor function to integer parts of a scaled harmonic series.
The right side of this identity is approximated using the harmonic sum estimate 
\begin{equation*}
\sum_{k=1}^n \frac{n}{k} = n \log n + \gamma n + O (1),
\end{equation*}
in which Euler's constant appears.  The difference term
 is the sum of fractional parts.
In  Sect.  2  Dirichlet formulates the  ``hyperbola method" to count  the sum
 of the lattice points on the right side of \eqn{500gg}, and in Sect.  3 he applies the method to get (1).
 In  Sect. 4 he compares the answer \eqn{500bb}  with \eqn{500gg} and obtains the formula 
   \eqn{500ff} giving (2). 
A modern treatment of the hyperbola method appears in Tenenbaum \cite[I.3.2]{Ten95}.
\end{proof}

The problem of obtaining  bounds
for the  size of the remainder term in \eqn{500bb}, 
\begin{equation}\label{500hh}
\Delta(n):=  \sum_{k=1}^{n}  d(k) - ( n\log n  + (2 \gamma -1)n),
\end{equation}
is  now called the {\em Dirichlet divisor problem.}  
This problem seeks the best possible bound for the exponent $\theta$
  of the form $ |\Delta(n)| =O( n^{\theta+\epsilon})$, valid for  any given $\epsilon >0$ for $1\le n < \infty$,
  with $O$-constant depending on $\epsilon$.
  This problem appears to be very difficult.
 Dirichlet's argument shows that  the size of the remainder term $\Delta(n)$  is dominated by
the remainder term in the fractional part sum \eqn{500ff}.
  In 1916 Hardy \cite{Ha16} established the lower bound $\theta \ge \frac{1}{4}$.
  The exponent $\theta= 1/4$   is conjectured
  to be the correct answer, and would correspond to an expected square-root cancellation. 
 The best $\Omega$-result  on fluctuations of  $\Delta(n)$, which does not
 change the lower bound exponent,  is that of Soundararajan \cite{Sou03} in 2003. 
  Concerning upper bounds, a long sequence of improvements in exponent bounds has arrived at  
  the current record $\theta \le \frac{131}{416} \approx 0.31490$,
  obtained  by    Huxley (\cite{Hux03}, \cite{Hux05}) in 2005.
Tsang \cite{Tsa10} gives a  recent survey on recent work on $\Delta(n)$ and related problems.

 Dirichlet  noted that   \eqn{500ff} 
implies that the fractional parts $\{\{ \frac{n}{k} \};  1 \le k \le n \}$
are far from being  uniformly distributed $(\bmod \, 1)$.
In Sect.  4 of \cite{Dirichlet1849} he proved   that the set   of fractional parts 
$k$ having  $0 \le \{ \frac{n}{k} \} \le \frac{1}{2}$ 
has a limiting frequency
$(2- \log 4)n  = (0.61370 \dots )n$.
and those with $\frac{1}{2} < \{ \frac{n}{k} \}< 1$ have
the complementary density $(\log 4 -1)n =  (0.38629 \dots) n$, 
as $n \to \infty.$

In 1898 de La Vall\'{e}e Poussin \cite{dLVP1898} generalized the fractional part sum
estimate \eqn{500ff} above  to summation over $k$ restricted to an arithmetic progression. He showed
for each integer $a \ge 1$ and integer $0 < b \le  a$, 
\begin{equation*}
 \sum_{k=0}^{(n-b)/a}  \{ \frac{n}{ak+b} \}  = \frac{1}{a}(1-\gamma) n + O (\sqrt{n}),
\end{equation*}
for $1 \le n < \infty$, with $O$-constant depending on $a$. 
He comments that it is very remarkable that 
these fractional parts approach the same average value $1-\gamma$  no matter which arithmetic progression
we restrict them to. Very recently Pillichshammer \cite{Pil10}
presented further work in this direction.

We next consider {\em  Euler's totient function}
\begin{equation*}
\phi(n) := \# \{ j:  1\le j \le n, ~~\mbox{with}~~\gcd(j, n)=1\},
\end{equation*}
which has $\phi(n)= |(\ZZ/n\ZZ)^{*}|.$ 
Euler's constant appears in the asymptotics of  the minimal growth rate (minimal order)
of $\phi(n)$. In 1903 Edmund Landau \cite{Lan03} showed  the following result, which
he also presented later  in his textbook \cite[pp. 216--219]{Lan09}. 

\begin{theorem} \label{th52}
{\em (Landau 1903)}
Euler's totient function $\phi(n)$ has minimal order $\frac{n}{\log\log n}$.
Explicitly,  
\begin{equation*}
\liminf_{n \to \infty} \frac{\phi(n) \log\log n}{n} = e^{-\gamma}.
\end{equation*}
\end{theorem}

Since $\phi(n)= n \prod_{p|n}(1-\frac{1}{p})$, its  extremal behavior is  attained by
the {\em primorial  numbers}, which are the numbers that are  products of
the first $k$ primes, i.e. $N_k := p_1p_2\cdots p_k$, for example  
$N_4=2\cdot 3\cdot 5\cdot7= 210.$ 
The result then follows using  Mertens's product theorem.\smallskip

The  {\em sum of divisors function}  is given by
\begin{equation*}
\sigma(n) := \sum_{d|n}  d,
\end{equation*}
so that  e.g. $ \sigma(6)=12.$
Euler's constant  appears in the asymptotics of the maximal growth rate
(maximal order)  of  $\sigma(n)$. 
In 1913 Gronwall \cite{Gr13} obtained
the following result.

\bigskip
\begin{theorem}~\label{th53}
{\em (Gronwall 1913)}
The sum of divisors function $\sigma(n)$ has maximal order $n \log\log n$.
Explicitly, 
\begin{equation*}
\limsup_{n \to \infty} \frac{\sigma(n)}{n \log\log n} = e^{\gamma}.
\end{equation*}
\end{theorem}

For $\sigma(n)$ the extremal numbers 
form a very thin set, with
a more complicated description, which was analyzed by 
Ramanujan \cite{Ram15} in 1915 (see also 
Ramanujan \cite{Ram97})  and Alaoglu and
Erd\H{o}s \cite{AE44} in 1944. They have the property
that their divisibility by each prime $p^{e_p(n)}$ has
 exponent $e_p(n) \to \infty$ as $n \to \infty$,
see Section 3.7 and   formula \eqn{551d}.

The Euler totient function $\phi(n)$ and the sum of divisors function $\sigma(n)$
are related by the  inequalities
\begin{equation*}
\frac{6}{\pi^2} n^2  < \phi(n) \sigma(n) \le n^2,
\end{equation*}
which are easily derived from the identity\footnote{ $p^e || n$ means $p^e |n $ and $p^{e+1} \nmid n$.}
$$
\phi(n) \sigma(n) = n^2 \prod_{{p^e || n}\atop{e \ge 1}} \left( 1- \frac{1}{p^{e+1}}\right).
$$
We note that the ratio of extremal behaviors  of the ``smallest" $\phi(n)$,
resp. the  ``largest" $\sigma(n)$, given in Theorem~\ref{th52}
and Theorem~\ref{th53} have a product asymptotically growing like $n^2$.
The extremal numbers $N_k = p_1 p_2 \cdots p_k$ in Landau's theorem are
easily shown to satisfy 
$\phi(N_k) \sigma(N_k) \sim \frac{6}{\pi^2} {N_k}^2$ as $N_k \to \infty$.
The extremal
numbers $\tilde{n}$ in Gronwall's theorem can be shown to satisfy the other bound
$$ \phi(\tilde{n}) \sigma(\tilde{n}) \sim {\tilde{n}}^2$$
as $\tilde{n} \to \infty$.
The latter result  implies that the Gronwall extremal numbers satisfy 
$$
\lim_{\tilde{n} \to \infty} \frac{ \phi(\tilde{n}) \log\log \tilde{n}}{\tilde{n}} = e^{-\gamma},
$$
showing they are asymptotically extremal in Landau's theorem as well.

In Sections \ref{sec36} and  \ref{sec37}  we discuss more recent
results showing  that the Riemann hypothesis is encoded as an assertion
as to how the extremal limits are approached in Landau's theorem and 
Gronwall's theorem.

%
%
%
%
\subsection{Euler's constant and sieve methods: the Dickman function}\label{sec34}
\setcounter{equation}{0}

 Euler's constant  appears in connection
 with statistics on the sizes of prime factors of general integers.
  Let $\Psi(x,  y)$ count the number of integers $n \le x$ having largest prime
 factor no larger than $y$.  These are the numbers remaining if one sieves
 out all integers divisible by some prime exceeding $y$.
 The Dirichlet series associated to
 the set of integers having no prime factor $>y$ is  the {\em partial zeta function} 
 \begin{equation}\label{523bb}
 \zeta(s  ; y) := \prod_{p \le y} \left( 1- \frac{1}{p^s} \right)^{-1},
 \end{equation}
 and we have 
 $$
 \zeta(1; y)= \frac{1}{D(y)},
 $$
 where $D(y)$ is given in \eqn{504b}. 
  It was shown by
 Dickman \cite{Dic30} in 1930 that, for each $u >0,$ a positive fraction
 of all integers $n \le x$ have all prime factors less than $x^{1/u}$.
 More precisely, one has the asymptotic formula
 \begin{equation*}
 \Psi ( x, x^{\frac{1}{u}}) \sim \rho(u) x, 
 \end{equation*}
 for a certain function  $\rho(u)$, now called the {\em Dickman function}.
 The notation $\rho(u)$  was introduced by de Bruijn \cite{deB51a}, \cite{deB51b}.)
This function is determined for $u \ge 0$ by the properties: 
\begin{enumerate}
\item 
(Initial condition) For $0 \le u \le1$, it satisfies
\begin{equation*}
\rho(u) =1.
\end{equation*}
\item
 (Differential-difference equation) For $u \ge 1$, 
\begin{equation*}
u \rho'(u) = - \rho(u-1).
\end{equation*}
\end{enumerate}
The Dickman function is alternatively  characterized  
as the solution on the real line to the 
integral equation
\begin{equation}\label{526d}
u \rho(u) = \int_{0}^{1} \rho (u-t) dt,
\end{equation}
(for $u \rho(u)$) with initial condition as above, extended to require that $\rho(u)=0$ for
all $u < 0$, as shown in de Bruijn \cite[(2.1)]{deB51a}.  Another explicit form for
$\rho(u)$ is the iterated integral form
\begin{equation}\label{526e}
\rho(u) = 1 + \sum_{k=1}^{\lfloor u \rfloor} \frac{(-1)^k}{k!} \int_{ { t_1+ \cdots _+ t_k \le u}\atop t_1, ..., t_k \ge 1} \frac{dt_1}{t_1} 
\frac{dt_2}{t_2} \cdots \frac{dt_k}{t_k},
\end{equation}
which also occurs
in connection with random permutations, see \eqref{gonch} ff.

It is known that for $u \ge 1$ the Dickman function $\rho(u)$ is a strictly decreasing positive function
that decreases rapidly, satisfying
\begin{equation*}
 \rho(u) \le \frac{1}{\Gamma(u+1)},
 \end{equation*}
 see Norton \cite[Lemma 4.7]{Nor71}.  The key relation of the Dickman function to Euler's constant
 is that it has total mass $e^{\gamma}$, see Theorem~\ref{th52b} below.
 
     In 1951 N. G. de Bruijn \cite{deB51a} 
  gave an exact expression for the Dickman function as a contour integral,
 \begin{equation}\label{528d}
 \rho(u) = \frac{1}{2 \pi i } \int_{-i\infty}^{i \infty} \exp\left(\gamma +  \int_{0}^z \frac{e^s-1} {s} ds\right)  e^{-uz} dz  ~~~(u>0).
 \end{equation}
 
 One may also consider the  one-sided Laplace transform of the Dickman function which
  we denote
   \begin{equation*}
  \hat{\rho}(s) := \int_{0}^{\infty} \rho(u) e^{-us} du,
 \end{equation*}
 following  Tenenbaum \cite[III.5.4]{Ten95}.
 This integral converges absolutely for all $s \in \CC$ and defines $\hat{\rho}(s)$
 as an entire function.   To evaluate it,  recall that the {\em complementary exponential integral }
 \begin{equation}\label{cexpi}
 \Ein(z) := \int_{0}^z \frac{1- e^{-t}}{t} dt = \sum_{n=1}^{\infty}(-1)^{n-1} \frac{1}{n \cdot n!} z^n,
\end{equation}
 is an entire function of $z$.
 \begin{theorem}\label{th52b} 
 {\em (de Bruijn 1951,  van Lint and Richert 1964)}
The one-sided Laplace transform $\hat{\rho}(s)$ of the
 Dickman function $\rho(u)$  is an entire function,
 given by 
 \begin{equation}\label{530d}
 \hat{\rho}(s) = e^{\gamma - \Ein(s)},
 \end{equation}
 in which $\Ein (s)$ is the complementary exponential integral.
  In particular, the total mass of the Dickman function is
 \begin{equation}\label{532d}
 \hat{\rho}(0) = \int_{0}^{\infty} \rho (u) \,du = e^{\gamma}.
 \end{equation}
 \end{theorem}

 \begin{proof}
 Inside  de Bruijn's formula \eqn{528d}  there appears the integral 
 $$
 I(s) := \int_{0}^{s} \frac{e^{t} -1}{t} dt= -\Ein(-s). 
 $$
 This   formula determines the Fourier transform of the Dickman function to be
 $e^{\gamma + I(-it)}$. 
The change of variable $s=-it$ yields the  the one-sided Laplace transform of
 $\rho(u)$  above.
  The total mass identity \eqn{532d} was noted in 1964 
   in van Lint and Richert \cite{vLR64}. Moreover they  
    noted that as $u \to \infty$, 
 \begin{equation*}
 \int_{0}^u \rho (t) dt = e^{\gamma} + O\left( e^{-u}\right).
 \end{equation*} 
 A direct approach to the determination of the 
  Laplace transform given in \eqn{530d}
 was given by  Tenenbaum \cite[III.5.4, Theorem 7]{Ten95} in 1995.
 His proof uses a lemma showing  that for all $s \in \CC \smallsetminus (-\infty, 0]$ one has
 \begin{equation}\label{533d}
  - I(-s)=  \Ein(s) = \gamma + \log s + J(s),
 \end{equation}
 with
 \begin{equation}\label{534d}
 J(s) := \int_{0}^{\infty} \frac{e^{-s-t}}{s+t} dt.
 \end{equation}
 Tenenbaum sketches a direct
 arithmetic proof of the identity \eqn{532d}  in \cite[III.5 Exercise 2, p. 392]{Ten95};
 in this argument the quantity $e^{\gamma}$ is derived using Mertens's product theorem.
 We  note the identity 
 \begin{equation*}
 J(s) = \Eone(s), 
 \end{equation*}
 in which 
 $ \Eone(s) := \int_{s}^{\infty} \frac{e^{-t}}{t} dt$ 
 is the principal exponential integral. The function $\Eone(s)$  
  appeared in connection with Euler's divergent
 series in Section \ref{sec25}, see \eqn{562a}.
\end{proof}

  In 1973 Chamayou \cite{Cha73} found a probabilistic interpretation
  of the Dickman function $\rho(x)$ which is suitable for computing it
  by Monte Carlo simulation.
  
 \begin{theorem}\label{th53aa} {\em (Chamayou 1973)}
  Let $X_1, X_2, X_3, ...$ be a sequence of independent
 identically distributed  random variables with uniform
 distribution on $[0,1].$  Set 
 \[
 P(u) := Prob [ X_1 + X_1 X_2 + X_1X_2X_3 + \cdots  \le u.]
 \]
 This function is well defined for $u \ge 0$ and 
 its derivative $P'(u)$ satisfies
 \begin{equation*}
 \rho (u) = e^{\gamma} P'(u).
 \end{equation*}
  \end{theorem}

  A discrete identity  relating the Dickman function and $e^{\gamma}$
  was noted by Knuth and Trabb-Pardo \cite{KTP76}, which complements 
  the continuous identity \eqn{532d}. 
  
 \begin{theorem}\label{th53b} {\em (Knuth and Trabb-Pardo 1976)}
  For $0 \le x \le 1$  the Dickman function satisfies the identity 
  \begin{equation}\label{534f}
  x + \sum_{n=1}^{\infty} (x+n) \rho(x+n)= e^{\gamma}.
  \end{equation}
  In particular, taking $x=0$, 
  \begin{equation*}
  \rho(1) + 2 \rho(2) + 3 \rho(3) + \cdots = e^{\gamma}.
  \end{equation*}
   \end{theorem}

\begin{proof}
 The identity  \eqn{534f}  immediately follows  from \eqn{532d}, 
  using the integral equation \eqn{526d} for $u\rho(u)$. 
  \end{proof}   
  
  We conclude this section by using  formulas above to 
  deduce a curious identity relating  a special
  value of this Laplace transform to  the Euler-Gompertz constant $\EG$
  given in Section 2.4; compare also Theorem~\ref{th102}. 
  
 \begin{theorem}\label{th53c} 
   The Laplace transform $\hat{\rho}(s)$ of the
 Dickman function $\rho(u)$  is given at $s=1$  by 
  \begin{equation*}
\hat{\rho}(1) :=  \int_{0}^{\infty} \rho(u) \, e^{-u} du = e^{-\frac{\EG}{e}}.
 \end{equation*} 
 in which
  $\EG :=  \int_{0}^1\frac{dv}{1-\log v}$ is the Euler-Gompertz constant.
 \end{theorem}

\begin{proof}
By definition $\hat{\rho}(1) = \int_{0}^{\infty} \rho(u) e^{-u}du.$
We start from Hardy's formula \eqn{560b} for the Euler-Gompertz constant:
$$
\EG= s(1) = \int_{0}^{\infty} \frac{e^{-w}}{1+w} dw.
$$
Comparison with \eqn{534d} then yields
\begin{equation*}
 \Eone(1)= J(1)= \int_{0}^{\infty} \frac{e^{-1-t}}{1+t} dt = \frac{\EG}{e}.
\end{equation*}
Now  \eqn{533d} evaluated at $s=1$ gives
\begin{equation}
\label{353EG}
\Ein(1) = \gamma+ \Eone(1)  = \gamma + \frac{\EG}{e}, 
\end{equation}
which can be rewritten  using \eqn{cexpi} as
\begin{equation}\label{353EG2}
\int_{0}^{1} \frac{1-e^{-t}}{t} dt =\gamma + \frac{\EG}{e}.
\end{equation}
We conclude from Theorem \ref{th52b} that
$$
\hat{\rho}(1) = e^{\gamma -Ein(1)} = e^{-\frac{\EG}{e}},
$$
which is the assertion.
\end{proof}

 Hildebrand and Tenenbaum \cite{HT93a} and Granville \cite{Gra08}
 give  detailed surveys including estimates for  $\Psi(x, y)$, a field sometimes
 called {\em psixyology}, see Moree \cite{Mor93}.  It has recently been
 uncovered that  Ramanujan had results on the
 Dickman function in his `Lost Notebook', prior to the work of
 Dickman, see Moree \cite[Sect. 2.4]{Mor13}.

%
%
%
%
\subsection{Euler's constant and sieve methods: the Buchstab function}\label{sec35}
\setcounter{equation}{0}

Euler's constant also arises in sieve methods in number theory
in the complementary problem where  one removes integers having some small prime factor.
Let $\Phi(x,y)$ count the number of integers $n \le x$ having
no prime factor $p \le y$.  In 1937 Buchstab \cite{Buc37} (see also \cite{Buc38}) 
established for $ u >1$  an asymptotic formula 
\begin{equation*}
\Phi(x, x^{\frac{1}{u}})\sim u\, \om(u) \frac{x}{\log x},
\end{equation*}
for a certain function $\om(u)$ named by
de Bruijn \cite{deB50a}  the  Buchstab function.

The {\em Buchstab function} is defined for $u \ge 1$ by the properties:
\begin{enumerate}
\item
(Initial conditions) For $1 \le u \le 2$,  it satisfies
\begin{equation*}
\om(u) = \frac{1}{u}.
\end{equation*}
\item
(Differential-difference equation) For  $u \ge 2$,
\begin{equation*}
(u \, \om(u) )^{'}=  \om(u-1).
\end{equation*}
\end{enumerate}
This function is alternatively  characterized  
as the solution on the real line to the 
integral equation
\begin{equation}\label{546d}
u \, \omega(u) = 1+\int_{1}^{u-1} \omega (t) dt,
\end{equation}
 as shown in de Bruijn \cite[(2.1)]{deB50a}, who
determined properties of this function.
Another explicit formula for this function, valid for $u >2$, is
$$
u \omega(u) = 
 1 + \sum_{2 \le k \le \lfloor u\rfloor} \frac{1}{k!}
 \int_{ {1/u \le y_i \le 1}\atop{ 1/u \le 1- (y_1+y_2 + \cdots + y_{k-1}) \le 1} }
 \frac{dy_1 dy_2 \cdots dy_{k-1}}{y_1 y_2 \cdots y_{k-1} (1- (y_1 +y_2+ \cdots + y_{k-1}) )}.
 $$
It is known that the  function $\om(u)$ is oscillatory  and satisfies
\begin{equation*}
\lim_{u \to \infty} \om(u) = e^{-\gamma}.
\end{equation*}
The convergence to $e^{-\gamma}$ is extremely rapid, and satisfies the estimate
\begin{equation}\label{eq345a}
\om(u) = e^{-\gamma} + O( u^{-u/2})  \quad ~~~\mbox{for}  ~~u \in [2, \infty).
\end{equation}
Proofs of these results can be found in 
Montgomery and Vaughan \cite[Sec. 7.2]{MV07}
and  in Tenenbaum \cite[Chap. III.6]{Ten95}.

The convergence behavior of the Buchstab function to $e^{\gamma}$ as $x \to \infty$ played a crucial
role in the groundbreaking 1985 work of  Helmut Maier \cite{Mai85} showing the existence of large 
fluctuations in the density of  primes in  short intervals;  this work was mentioned at the end
of Section \ref{sec33}.
 Maier \cite[Lemma 4]{Mai85} 
showed that $\om(u)-e^{-\gamma}$ changes sign 
at least once on each interval $[a-1, a]$ for any $a \ge 2$.
In 1990 Cheer and Goldston \cite{CG90}
showed there are at most two sign changes  and
at most two critical points on an interval $[a-1, a]$, and Hildebrand \cite{Hi90}  showed that
 the spacing between such
sign changes approaches $1$ as $u \to \infty$,

A  consequence of the analysis 
of de Bruijn \cite{deB50a} is that the 
 constant $e^{-\gamma}$  appears 
as a universal limiting constant for sieve methods,
for a wide range of sieve cutoff functions $y=y(x)$ growing 
neither too slow nor too fast
compared to $x$. 


\begin{theorem} \label{buchstab} {\rm (de Bruijn 1950)}
Suppose that $y= y(x)$ depends on $x$  in such a way
that both $y \to \infty$ and 
$\frac{\log y}{\log x} \to 0$ hold as $x \to \infty$.
Under these conditions  we have
\begin{equation}\label{361aa}
\Phi(x, y) \sim e^{-\gamma} \frac{x}{\log y}.
\end{equation}
\end{theorem}

\begin{proof}
This result can be deduced from de Bruijn \cite[(1.7)]{deB50a}.
This asymptotic formula also follows as
 a corollary of  very general estimate (\cite[III.6.2, Theorem 3]{Ten95})
 valid uniformly on the region $x \ge y \ge 2$, which states
\begin{equation*}
\Phi(x, y) = \omega(\frac{\log x}{\log y})\frac{x}{\log y}- \frac{y}{\log y} + O \Big( \frac{x}{ (\log y)^2} \Big).
\end{equation*}
The formula \eqn{361aa}  follows using the estimate \eqn{eq345a} for $\omega(u)$,
which applies since the  hypotheses $x/y \to \infty$ when $x \to \infty$
imply $\frac{\log x}{\log y} \to \infty$.
\end{proof}

The one-sided Laplace transform of the Buchstab function is given as  
\begin{equation}\label{355a}
\hat{\omega}(s) := \int_{0}^{\infty} e^{-su} \omega (u) du,
\end{equation}
where we make the convention that $\omega(u) =0$, for $0 \le u <1.$ 
Under this convention, this Laplace
transform has a simple relation with  the Laplace transform of the Dickman function.

\begin{theorem} \label{buchstab0} 
{\em (Tenenbaum 1995)} 
The one-sided Laplace transform $\hat{\omega}(s)$ defined for $Re(s)>0$ by \eqn{355a},
extends to a meromorphic function on $\CC$,
given explicitly by the form
\begin{equation}\label{356a}
1 + \hat{\omega}(s) = \frac{1}{s \hat{\rho}(s)}, ~~~~\quad\quad (s \ne 0).
\end{equation}
When $s$ is not real and negative, 
\begin{equation}\label{357a}
1+ \hat{\omega}(s) = e^{J(s)},
\end{equation}
with $J(s) = \int_{0}^{\infty} \frac{e^{-s-t}}{s+t} dt = E_1(s).$
\end{theorem}

\begin{proof}
The Laplace transforms  individually have been evaluated many times in
the sieve method literature,
see Wheeler \cite[Theorem 2, ff]{Whe90}, who also gives references  back to the 1960's.
One has explicitly that
\begin{equation}\label{357b}
\hat{\omega}(s) = \frac{1}{s} \exp( -\Ein(s)),
\end{equation}
in terms of the complementary exponential integral.
The identity \eqn{356a} connecting them seems to be first explicitly stated in 
 Tenenbaum \cite[III.6, Theorem 5]{Ten95}.
 \end{proof}

\noindent Combining this result with  Theorem   \ref{th52b} gives
$$
\hat{\omega}(s) \sim   \frac{ e^{-\gamma}}{s}, \quad\quad \mbox{as} ~~ s\to 0.
$$
We also obtain using \eqn{357a}, \eqn{357b} along  with \eqn{249a} that 
\begin{equation}\label{362EG}
\hat{\omega}(1) := \int_{0}^{\infty} \omega(u) e^{-u} du 
= e^{-\gamma} \exp (-\Eone(1))= e^{-\gamma} e^{-\delta/e},
\end{equation}
where $\delta$ is the Euler-Gompertz constant.

Wheeler \cite[Theorem 2]{Whe90} also observes that the solutions to
the differential-difference equations for the Dickman  function and the
Buchstab function are distinguished from those for general initial conditions
by the special property that their Laplace transforms analytically continue
to entire (resp. meromorphic) functions of $s \in \CC$. 

For a retrospective look  at  the work of 
 N. G. de Bruijn on $\rho(x)$ and $\Psi(x,y)$ in Section \ref{sec35}, 
 and of  $\omega(x)$ and $\Phi(x, y)$ in this section, see  Moree \cite{Mor13}.

%
%
%
%
\subsection{Euler's constant  and the Riemann hypothesis }\label{sec36}
\setcounter{equation}{0}

Several formulations of the Riemann hypothesis can be 
given in terms of Euler's constant. Here we describe some
that  involve the approach towards
the extremal behaviors of 
 Euler's totient function $\phi(n)$ and the sum of divisors function $\sigma(n)$
 given in  Section \ref{sec33a}.
The Riemann hypothesis is  also  related, in another way,  to generalized Euler's constants 
 considered  in Section \ref{sec37}.

 In 1981 J.-L. Nicolas (\cite{Ni81}, \cite{Ni83}) proved 
that the Riemann hypothesis is encoded in the property
 that the Euler totient function values $\phi(n_k)$  for the extremal numbers $n_k = p_1p_2\cdots p_k$ given
 in Landau's theorem 
 approach their  limiting value from one side only. Nicolas stated his result using the
 inverse of Landau's quantity, as follows.

\begin{theorem}\label{th54}
{\em (Nicolas 1981) }
The Riemann hypothesis holds if and only if all the primorial numbers 
$N_k = p_1p_2 \cdots p_k$  with $k \ge 2$ satisfy
\begin{equation*}
\frac{N_k}{\phi(N_k) \log\log N_k}> e^{\gamma}.
\end{equation*}
If the Riemann hypothesis is false, then these inequalities will be true for infinitely many 
primorial numbers $N_k$
and false for infinitely many $N_k$.
\end{theorem}

This inequality is equivalent to the statement  that  the Riemann hypothesis implies that all
the primorial numbers $N_k$  for $k \ge 2$ {\em undershoot} the asymptotic lower bound $e^{-\gamma}$
 in Landau's Theorem \ref{th52}. 
As an example, $\phi( 2 \cdot 3 \cdot 5 \cdot 7) = \phi(210)=48$ and 
$$
\frac{\phi(n_4) \log\log n_4}{n_4} = \frac{\phi(210) \log\log 210}{210} = 0.38321 \cdots 
< e^{-\gamma} = 0.56145 \dots .
$$
For most $n$ one will have $\frac{\phi(n) \log\log n}{n}> e^{-\gamma}$, and only
a very thin subset of $n$ will be close to this bound. For example taking
 $n=p_k$, a  single prime, one sees that as $k \to \infty$ this ratio goes to $+\infty$.

Very recently  Nicolas \cite{Ni12}  
obtained a refined encoding of the Riemann hypothesis 
in terms of a sharper asymptotic for small values of the Euler $\phi$-function,
in which both $\gamma$ and $e^{\gamma}$ appear.

\begin{theorem}\label{th374}
{\em (Nicolas 2012) }
For each integer $n \ge 2$ set 
$$
c(n) := \Big(\frac{n}{\phi(n)} - e^{\gamma} \log\log n\Big) \sqrt{\log n}.
$$
Then the Riemann hypothesis is equivalent to the statement that
\begin{equation}\label{374b}
\limsup_{n \to \infty} \,c(n) = e^{\gamma}( 4 + \gamma - \log (4\pi)).
\end{equation}
\end{theorem}

\begin{proof}
The assertion that the Riemann hypothesis
implies \eqn{374b} holds is \cite[Theorem 1.1]{Ni12}, and the converse 
assertion is  \cite[Corollary 1.1]{Ni12}. Here the constant
\[
e^{\gamma}\Big(4+ \gamma -\log (4 \pi) \Big)= e^{\gamma}(2 + \beta)= 3.64441\, 50964...,
\]
in which
\[ 
\beta= \sum_{\rho} \frac{1}{\rho(1-\rho)}=2 +\gamma - \log (4 \pi)=0.04619\, 14179...,
\]
where in the sum $\rho$ runs over the nonreal zeros of the zeta function, counted
with multiplicity.
For the  converse direction, he  shows that if  the Riemann hypothesis fails, then
$$
\limsup_{n \to \infty} c(n) = +\infty.
$$
Note that   $\liminf_{n \to \infty} c(n)= -\infty$ holds unconditionally.
\end{proof}

In 1984 G. Robin \cite{Ro84} showed  that for 
the sum of divisors function $\sigma(n)$ the Riemann hypothesis
is also encoded as  a one-sided approach to the limit in
Gronwall's theorem (Theorem \ref{th53}).

\begin{theorem}\label{th55}
{\em (Robin 1984) }
The Riemann hypothesis holds if and only if the inequalities 
\begin{equation}
\label{551}
\frac{\sigma(n)}{n \log\log n} < e^{\gamma}
\end{equation}
are valid for all $n \ge 5041$. If the Riemann hypothesis is false, then
this inequality will be true for infinitely many $n$ and false for infinitely
many $n$.
\end{theorem}

Robin's result says
that  Riemann hypothesis  is equivalent to the limiting value $e^{\gamma}$ being approached
from below, for all sufficiently large $n$.
The inequality  \eqn{551} fails to hold  for a few small $n$, 
the largest known exception being $n=5040$.
Here the extremal numbers giving record values for $f(n)= \frac{\sigma(n)}{n}$ 
will have a different form than that in 
Nicolas's theorem. 
It is known that infinitely many of the extremal numbers  will be
{\em colossally abundant numbers},
as defined by Alaoglu and Erd\H{o}s \cite{AE44} in 1944.
These are numbers $n$ such that there is some $\epsilon >0$ such that
$$
\frac{\sigma(n)}{n^{1+\epsilon}} \ge \frac{\sigma(k)}{k^{1+\epsilon}}, ~~~\mbox{for}~~~ 1\le k < n.
$$
Alaoglu and Erd\H{o}s showed that the  ``generic"  colossally abundant number is  a product of 
powers of small primes
having a  sequence of exponents  
by a parameter $\epsilon >0$ as
\begin{equation*}
n = n(\epsilon) :=\prod_{p} p^{a_p(\epsilon)},
\end{equation*}
in which 
\begin{equation}\label{551d}
a_p(\epsilon) := \lfloor \frac{ \log(p^{1+\epsilon} -1) - \log(p^{\epsilon}-1)}{\log p} \rfloor -1.
\end{equation}
For  fixed $\epsilon$ the  sequence of
exponents  $a_p(\epsilon)$ is non-increasing as $p$ increases
and become $0$ for large $p$.
However  for fixed $p$ the exponent $a_p(\epsilon)$ increases as $\epsilon$ decreases
 towards $0$, so 
that $n(\epsilon) \to \infty$ as $\epsilon \to 0^{+}.$ 
Colossally abundant and related  numbers had  actually been studied  by Ramanujan \cite{Ram15}
in 1915, but the relevant part of this paper was suppressed by the London Mathematical Society 
to save expense.  The suppressed part of the manuscript was recovered in 
his Lost Notebook, and later published in 1997 ( \cite{Ram97}).

In 2002 this  author (\cite{La02}), starting from Robin's result,  obtained the following  
elementary criterion for the Riemann
hypothesis, involving the sum of divisors function $\sigma(n)$ and the harmonic
numbers $H_n$.

\begin{theorem}~\label{th56}
The Riemann hypothesis is equivalent to the assertion that
for each $n \ge 1$ the  inequality
\begin{equation}\label{561}
\sigma(n) \le e^{H_n} \log H_n + H_n
\end{equation}
is valid.  Assuming the Riemann hypothesis,   equality holds if and only if $n=1$.
\end{theorem}

The additive  term $H_n$ is included in this formula for elegance, as 
it makes the result hold  for all $n \ge 1$ (assuming RH), rather than 
be valid for $n \ge 5041,$
as in Robin's theorem \ref{th55}. 
The converse direction of this
result requires additional proof, which is accomplished in \cite{La02}
using certain asymptotic estimates
obtained in Robin's paper \cite{Ro84}. It shows that if the Riemann hypothesis is false then
the inequalities \eqn{561} will fail to hold for infinitely many positive integers $n$.

%
%
%
%
\subsection{Generalized  Euler constants and the Riemann hypothesis}\label{sec37}
\setcounter{equation}{0}

 In 1961 W. Briggs \cite{Br61} introduced the notion of an {\em Euler constant
 associated to  an arithmetic progression} of integers. This notion was studied in detail by
D. H. Lehmer \cite{Leh75}  in 1975. For the arithmetic progression $h ~(\bmod~k)$ with $0\le   h< k$
we set
\begin{equation*}
\gamma(h,k) := \lim_{x \to \infty} 
\Big(\sum_{\substack{0< n \le x\\n\equiv h ~(\bmod k)}} \frac{1}{n} - \frac{\log x}{k} \Big).
\end{equation*}
These constants were later termed {\em Euler-Lehmer constants} by
Murty and Saradha \cite{MS10}.  Here one has  $\gamma(0,1) = \gamma$, $\gamma(1, 2) =\frac{1}{2} ( \gamma+ \log 2)$
and $\gamma(2, 4) = \frac{1}{4} \gamma$, where $\gamma$ is Euler's constant.
It suffices to study the constants  where $\gcd(h,k)=1$, other cases reduce to these by 
dividing out the greatest common factor.

 We next define {\em generalized Euler constants $\gamma(\Omega)$} associated
 to a finite set of primes
  $\Omega=\{ p_{i_1}, .... , p_{i_k}\} $,  possibly empty.
  These constants  were introduced
   by Diamond and Ford \cite{DF08} in 2008.
  To specify them, we first define a 
  {\em zeta function} $Z_{\Omega}(s)$ by 
\begin{equation*}
Z_{\Omega}(s) := \left( \prod_{p \in \Omega} (1- \frac{1}{p^s}) \right) \zeta(s)\\
= \sum_{n=1}^{\infty} {\bf 1}_{\Omega}(n) n^{-s}= \sum_{gcd(n, P_{\Omega})=1} n^{-s} 
\end{equation*}
where we set
\[
{\bf 1}_{\Omega}(j) = 
\begin{cases}
1 & ~\mbox{if} ~~(n, P_{\Omega}) = 1,\\
0 & ~\mbox{otherwise,}
\end{cases}
\]
in which 
$$
P_{\Omega} := \prod_{p \in \Omega} p.
$$
The function $Z_{\Omega}(s)$ defines a meromorphic function on the entire plane which has
a simple pole at $s=1$ with residue 
\begin{equation*}
D(\Omega):= \prod_{p \in \Omega} \left(1- \frac{1}{p}\right).
\end{equation*}
The  {\em generalized Euler constant} $\gamma(\Omega)$ associated to $\Omega$
is  the constant term in the  Laurent expansion of $Z_{\Omega}(s)$ around $s=1$, namely
$$
Z_{\Omega}(s) = \frac{D(\Omega)}{s-1} + \gamma(\Omega) + \sum_{j=1}^{\infty} \gamma_k(\Omega)(s-1)^k.
$$ 
In the case that $\Omega=\emptyset$ is the empty set,  this function $Z_{\Omega}(s)$ is exactly the
Riemann zeta function, and 
$\gamma(\emptyset)= \gamma$ by Theorem \ref{th41}.
These numbers $\gamma(\Omega)$ generalize  the characterization of Euler's constant in 
\eqn{402} in the sense that
\begin{equation*}
\gamma(\Omega) = \lim_{n \to \infty} \Big( \sum_{j=1}^n  \frac{{\bf 1}_{\Omega}(j)}{j} - D(\Omega) \log n \Big).
\end{equation*}
The constants $\gamma(\Omega)$ are easily shown to be  finite sums of Euler-Lehmer constants
\begin{equation}\label{604aa}
\gamma(\Omega) = \sum_{{1 \le h < P_{\sP}}\atop{gcd(h, P_{\Omega})=1}} \gamma(h, P_{\Omega}).
\end{equation}

Diamond and Ford  \cite[Theorem 1]{DF08} show that these generalized Euler constants  are related to 
Euler's constant in a second way which
involves the constant $e^{-\gamma}$ rather than $\gamma$. 
Let  $\Omega_r$ denote
the set of the first $r$ primes and set $\gamma_r := \gamma(\Omega_r)$.
\begin{theorem}\label{th61}
{\em (Diamond and Ford 2008)}
Let $\Gamma= \inf \{ \gamma (\sP): ~~\mbox{all finite}~~\sP\}$. Then:

(1) The  values of $\gamma(\sP)$ are dense in $[\Gamma, \infty).$

(2) The constant $\Gamma$ satisfies 
\begin{equation}\label{381k}
0.56 \le \Gamma \le e^{-\gamma}.
\end{equation}
If $\Gamma < e^{-\gamma}$ then 
the value $\Gamma$ is attained
at some $\gamma_r$. 
\end{theorem}

\begin{proof}
(1) This is shown as \cite[Prop. 2]{DF08}.

(2) Diamond and Ford show that  each $\gamma(\sP) \ge \gamma_r$ for some $1 \le r < |\sP|$. 
This implies 
\begin{equation}\label{381m}
\Gamma = \inf  \{ \gamma_r: r \ge 1\}.
\end{equation}
 Merten's formulas for sums and products yield the asymptotic formula
\begin{equation}\label{381h}
\gamma_r \sim e^{-\gamma} ~~\mbox{as}~~~ r \to \infty.
\end{equation}
which by \eqn{381m}   implies that $\Gamma \le e^{-\gamma}$.
If $\Gamma < e^{-\gamma}$ then it must be attained by one of the  $\gamma_{r}$,
For there must be at least one  $\gamma_{r_1} <  e^{-\gamma}$,  
and  by \eqn{381h} there are only  finitely many $\gamma_{r} \le \gamma_{r_1}$.
and  one of these must attain the infimum by the bound above. 
Diamond and Ford established the lower bound $\gamma_r  \ge 0.56$
for all $r \ge 1$.
\end{proof}

 Diamond and Ford \cite[Theorem 2]{DF08} also show   that the behavior of the quantities
$\gamma(\Omega)$ is complicated as the set $\Omega$ increases,  in the sense
that $\gamma_{r} $ is not a monotone function of $r$,
with $\gamma_{r+1} > \gamma_r$ and $\gamma_{r+1} < \gamma_r$ each occurring
infinitely often. 
Nevertheless  Diamond and Ford  \cite[Theorems 3, 4]{DF08} 
obtain the following  elegant reformulation of the Riemann hypothesis. 

\begin{theorem}\label{th62}
{\em (Diamond and Ford 2008)}
The Riemann hypothesis is equivalent to
either one of the following assertions.
\begin{enumerate}

\item
The infimum $\Gamma$ of $\gamma(\sP)$ satisfies 
$$\Gamma = e^{-\gamma}.
$$
\item
For  every finite set $\Omega$
of primes 
\begin{equation}\label{621}
\gamma(\Omega) > e^{-\gamma}.
\end{equation}
\end{enumerate} 
\end{theorem}

Here (2)  implies (1), and (2) says that  if the Riemann hypothesis is true, then \eqn{621} holds for all $\Omega$
and the infimum $\Gamma = e^{-\gamma}$ is not attained,  while if it does not hold,  
 then by Theorem \ref{th61} (1) the  reverse inequality 
  holds  for infinitely many $\Omega$, and the infimum is attained.
Theorem~\ref{th62} seems remarkable: on taking  $\sP=\emptyset$ to be the empty set,
it  already requires, for the truth of the Riemann hypothesis, that
$$
\gamma= 0.57721...> e^{-\gamma}= 0.56145...
$$
In consequence the unique real root 
$x_0\approx 0.56714$ of the equation $x_0= e^{-x_0}$ necessarily
satisfies $\gamma > x_0 > e^{-\gamma}$.

Results on the transcendence  of Euler-Lehmer  constants
$\gamma(h, k)$  and generalized Euler constants $\gamma(\Omega)$ have recently been established, see
 Section \ref{sec312}.

 Finally we note  that  one may  define
 more generally 
 {\em higher order Euler-Lehmer constants}. For $j \ge 1$ Dilcher \cite{Dil92}  defines for $h \ge 1,$
  and $0< k \le h$, the constants
\begin{equation*}
\gamma_j(h, k) := \lim_{x \to \infty} \Big( \sum_{\substack{0 < n \le x \\ n \equiv h (\bmod k)}} \frac{(\log n)^j}{n}- \frac{ (\log x)^{j+1}}{k(j+1)} \Big).
\end{equation*}
Here  $j$ is the {\em order}, with $\gamma_1(h, k) =\gamma(h,k)$ being the Euler-Lehmer constants, and
$\gamma_j(1,1) = \gamma_j$ is the $j$-th Stieltjes constant (see  Theorem \ref{th41}),
Dilcher relates these constants to values of derivatives of the digamma function at rational points, and also 
 to derivatives of Dirichlet $L$-functions evaluated  at $s=1$.

%
%
%
%
\subsection{Euler's constant and extreme values of $\zeta(1+it)$ and $L(1, \chi_{-d})$} \label{sec37a}
\setcounter{equation}{0}

Some of the deepest questions in number theory concern the distribution
of values of the Riemann zeta function and its generalizations, Dirichlet $L$-functions.
In this section we consider such value distributions  at  the special point $s=1$ and on the vertical line $\text{Re}(s) =1$.
Here Euler's constant appears 
in connection with the size of extreme values of $\zeta(1+it)$ as $t \to \infty$. 
It also appears  in the same guise in connection with extreme values of  the Dirichlet
$L$-function  $L(s, \chi_{-d})$  at the point $s=1$ 
where $\chi_{-d}$ is the real primitive Dirichlet
character associated to the quadratic field $\QQ(\sqrt{-d})$, and we let $d \to \infty$.
(The quantity $-d$ denotes a fundamental discriminant of an imaginary quadratic field, which is necessarily  squarefree away  from
the prime $2$.) In both case  these
 connections   were first made by J. E. Littlewood.
Some of his results are conditional on the Riemann hypothesis
or  the generalized Riemann hypothesis, while others are unconditional. 

In  these results Euler's constant occurs by way  of  Merten's Theorem \ref{th51}.
The basic idea,  in the context of large values of $|\zeta(1+it)|$,
is that large values  will occur 
(only) when  a suitable finite truncation of  its Euler product representation
is large, and this in turn will occur  when the phases of the
individual  terms in the product line up properly.
Mertens's theorem is then used in estimating  the size of this Euler product.

We first consider  extreme values of $\zeta(1+it)$.
To place results in context, it  is known unconditionally that 
$|\zeta(1+it)| = O(\log |t|)$ and that $\frac{1}{|\zeta(1+it)|} = O(\log |t|)$ as $|t| \to \infty$.
 In 1926 Littlewood   determined, assuming the Riemann hypothesis, 
 the correct maximal order of growth of the Riemann zeta function on the line $\text{Re}(s)=1$,
and in 1928 he obtained a corresponding  bound for $\frac{1}{\zeta(1+it)}$.

\begin{theorem}\label{th391b}
{\em (Littlewood 1926, 1928)}
Assume the Riemann hypothesis. Then

(1) The maximal order of $|\zeta(1+it)|$ is at most $\log\log t$, with
\[
\limsup_{ t \to \infty} \frac{ |\zeta(1+it)|}{\log\log t} \le    2 e^{\gamma}.
\]

(2) The maximal order of $\frac{1}{|\zeta(1+it)|}$ is at most $\log\log t$, with

\[
\limsup_{t \to \infty} \frac{1}{|\zeta(1+it)| \log\log t} \le \frac{2}{\zeta(2)}\, e^{\gamma}
\]
\end{theorem}

\paragraph{\bf Proof.} 
The bound (1) is    \cite[Theorem 7]{Lit26} and  the bound (2) is   \cite[Theorem 1]{Lit28a}.
 $~~~\Box$. \smallskip

Littlewood  also established unconditionally a 
lower bound for the quantity in (1) and, conditionally 
on the Riemann hypothesis, a lower bound for that in (2), both of which
matched the bounds above up to a factor of 2. In 1949 
a method of Chowla \cite{Cho49}
made this lower bound for the quantity in (2) unconditional. These combined
results are as follows. 

\begin{theorem}\label{th392b}
{\em (Littlewood 1926, Titchmarsh 1933)}
The following bounds  hold unconditionally.

(1) The maximal order of $|\zeta(1+it)|$ is at least $\log\log t$, with
\[
\limsup_{ t \to \infty} \frac{ |\zeta(1+it)|}{\log\log t} \ge     e^{\gamma}.
\]

(2) The maximal order of $\frac{1}{|\zeta(1+it)|}$ is at least $\log\log t$, with
\[
\limsup_{t \to \infty} \frac{1}{|\zeta(1+it)| \log\log t} \ge \frac{1}{\zeta(2)} \,e^{\gamma}
\]
\end{theorem}

\begin{proof}
The bound (1) is Theorem 8 of \cite{Lit26}. The lower bound (2) was proved, assuming the Riemann
hypothesis,  in 1928 by Littlewood \cite{Lit28a}.  An unconditional proof
was later given by Titchmarsh \cite{Tit33}, see also \cite[Theorem 8.9(B)]{TH86}.
\end{proof}

Littlewood's  proof of  (1) used Diophantine approximation properties of the
values $\log p$ for primes $p$, namely the fact that they are linearly independent
over the rationals. This guarantees that there exist suitable values of $t$ where the all
terms in the Euler product
$
\prod_{p \le X} (1- \frac{1}{p^{1+it}})^{-1}
$
have most phases $ t \log p $ near $0 ~(\bmod \,2\pi)$ for  $p \le  X$.

Littlewood was struck by the fact that the 
unconditional lower bound for\\
 $\limsup_{ t \to \infty} \frac{ |\zeta(1+it)|}{\log\log t}$
in Theorem \ref{th392b}(1) and
the conditional upper bound for it given  in Theorem \ref{th391b}(1) differ
by the simple multiplicative factor  $2$.
He discusses this fact at length in \cite{Lit28a}, \cite{Lit28b}.  The
Riemann hypothesis is not sufficient to predict the exact value!
It seems that more  subtle properties of the Riemann zeta function,
perhaps related to Diophantine approximation properties of prime
numbers or of the imaginary parts of zeta zeros, will play a role in determining the exact constant.
Littlewood favored the right answer as being the lower bound in Theorem \ref{th392b}
 saying (\cite[p. 359]{Lit28b}):

\begin{quote}
The results involving $c$ [$=e^{\gamma}$] are evidently final except
for a certain factor $2$. I showed also that on a certain further hypothesis (which there
is, perhaps, no good reason for believing) this factor $2$ disappears.
\end{quote}

\noindent In the fullness of time this statement has been elevated to a conjecture attributed
to Littlewood. 

 There is now relatively strong  evidence favoring the conjecture that the lower bound (1)
 in Theorem \ref{th392b} should be  equality.
In 2006 Granville and Soundararajan \cite{GS06} obtained a bound for
the frequency of occurrence of extreme values  of $|\zeta(1+it)|$, considering
the quantity
\[ 
\Phi_T( \tau) := \frac{1}{T} \meas \{ t \in [T, 2T]: ~|\zeta(1+it)| > e^{\gamma} \tau \},
\]
where $\meas$ denotes Lebesgue measure. They 
established a result \cite[Theorem 1]{GS06} showing that 
for all sufficiently large $T$, uniformly in the range $1 \ll \tau \le \log\log T + 20$,
there holds
\begin{equation}\label{asympbd}
\Phi_T( \tau) = \exp \Big( -\frac{2 e^{\tau -C -1}}{\tau}
 \Big( 1+ O ( \frac{1}{\sqrt{\tau} } + (\frac{e^{\tau}}{\log T})^{1/2}) \Big) \Big).
\end{equation}
where $C$ is a positive constant. 
An extension of these ideas 
established the following result (\cite[Theorem 2]{GS06}).

\begin{theorem}\label{th393b}
{\em (Granville and Soundararajan 2006)}
For all sufficiently large $T $, the measure of the set of points $t$ in $[T, 2T]$ having
\[
 |\zeta(1+it)| \ge e^{\gamma} \big( \log\log T + \log\log\log T - \log\log\log\log T - \log A + O (1)\big)
\]
is at least $T^{1- \frac{1}{A}},$ uniformly for $A \ge 10$.
\end{theorem}

They observe that if the estimate \eqn{asympbd}
 remained valid without
restriction on the range of $\tau$, this 
 would correspond to the following stronger assertion.

\begin{cj}\label{cj394b} 
{(Granville and Soundararajan 2006)}
There is  a constant $C_1$ such that for  $T \ge 10$,
\[
\max_{ T \le t \le 2T} |\zeta(1+it)| = e^{\gamma} \big( \log\log T + \log\log\log T  + C_1 + o(1)\big)
\]
\end{cj}

\noindent This conjecture implies
that  the lower bound (1) in Theorem \ref{th392b} would be an equality.

We next consider the distribution of extreme values of the Dirichlet $L$-functions $L(s, \chi_{-d})$
at $s=1$, 
for real primitive characters $\chi_{-d}(n) = (\frac{-d}{n})$, where $-d$ is the discriminant of 
an imaginary quadratic field $\QQ(\sqrt{-d})$. The Dirichlet $L$-function
\[ 
L(s, \chi_{-d}) = \sum_{n=1}^{\infty} \Big(\frac{-d}{n}\Big)n^{-s},
\]
is a relative of the Riemann zeta function, having an Euler product and
a functional equation.
 Dirichlet's class number
formula  states that for a fundamental discriminant $-d$ with  $d>0,$
$$
L(1, \chi_{-d}) = \frac{2\pi \,h(-d)}{w_{-d}\sqrt{d}},
$$
where $h(-d)$ is the order of the ideal class group in $\QQ(\sqrt{-d})$ and
$w_{-d}$ is the number of units in $\QQ(\sqrt{-d})$ so that $w_{-d} = 2$ for $d>4$,
$w_{-3}=6$ and $w_{-4}=4$. 
Thus the  size of $L(1, \chi_{-d})$ encodes information on the size of the class group,
and the formula shows that $L(1, \chi_{-d}) >0$.

\begin{theorem}\label{th395b}
{\em (Littlewood 1928)}
Assume the Generalized Riemann hypothesis  for all real primitive
characters $L(s, \chi_{-d})$ where $-d$ is the discriminant of
an imaginary quadratic field $\QQ(\sqrt{-d})$. Then:

(1) The maximal order of $L(1, \chi_{-d})$ is at most $\log\log d$, with
\[
\limsup_{ d \to \infty} \frac{ L(1, \chi_{-d})}{\log\log d} \le    2 e^{\gamma}.
\]

(2) The maximal order of $\frac{1}{L(1, \chi_{-d})}$ is at most $\log\log d$, with
\[
\limsup_{d \to \infty}\frac{1}{ L(1, \chi_{-d}) \log\log d} \le  \frac{2}{\zeta(2)}\, e^{\gamma}.
\]
\end{theorem}

\begin{proof}
The bounds (1) and    (2) are the content of  \cite[Theorem 1]{Lit28b}.
\end{proof}

There are corresponding unconditional lower bounds that differ from these
by a factor of $2$, as noted by Littlewood in 1928 for the bound (1), and
completed by a result of 
Chowla \cite{Cho49} for the bound (2).

\begin{theorem}\label{th396b}
{\em (Littlewood 1928, Chowla 1949)}
The following bounds 
hold unconditionally as $d \to \infty$
with $-d$ a fundamental discriminant.

(1) The maximal order of $L(1, \chi_{-d})$ is at least $\log\log d$, with
\[
\limsup_{ d \to \infty} \frac{ L(1, \chi_{-d})}{\log\log d} \ge     e^{\gamma}.
\]

(2) The maximal order of $\frac{1}{L(1, \chi_{-d})}$ is at least   ${\log\log d}$, with
\[
\limsup_{d \to \infty}\frac{1}{L(1, \chi_{-d}) \log\log d} \ge  \frac{1}{\zeta(2)}\, e^{\gamma}.
\]
\end{theorem}

\begin{proof} 
 The bound (1) is due to Littlewood \cite{Lit28b}.
The bound (2) is due to Chowla \cite[Theorem 2]{Cho49}. 
\end{proof}

There is again  a factor of $2$ difference between the unconditional lower
bound and the conditional upper bound.
In 1999 Montgomery and Vaughan \cite{MV99} 
formulated a model for  the general distribution of
sizes of Dirichlet character values at $s=1$, and based on it they 
advanced  very precise
conjectures in favor of the lower bounds being the correct answer, as follows.

\begin{cj}\label{cj397b} { (Montgomery and Vaughan 1999)}

(1) For each $\epsilon >0$, for all $D \ge D(\epsilon)$ there holds
\[
 \max_{d \le D} L(1, \chi_{-d}) \le  e^{\gamma} \log\log D + (1+ \epsilon) \log\log \log D,
\]

(2) For each $\epsilon >0$, for all $D \ge D(\epsilon)$ there holds
\[
\max_{d \le D} \frac{1}{ L(1, \chi_{-d})} \le  \frac{1}{ \zeta(2)} e^{\gamma} \log\log D + O\Big(\frac{1}{(\log\log D) (\log\log\log D)}\Big).
\]
\end{cj}

In 2003 Granville and Soundararajan \cite{GS03} obtained detailed probabilistic
estimates for  the number of characters having extreme values $e^{\gamma} \tau$. 
Their results imply unconditionally that there are infinitely many $d$ with
\[
L(1, \chi_{-d}) \ge e^{\gamma}\big( \log\log d + \log\log\log d - \log\log\log \log d -10\big).
\]
In their later paper \cite[Theorem 3]{GS06} they sketched a proof that 
for any fixed $A \ge 10$ for all sufficiently large primes $q$
there are at least $q^{1- \frac{1}{A}}$ Dirichlet characters $\chi~(\bmod \, q)$ such that
\[
|L(1, \chi)| \ge e^{\gamma}\big( \log\log d + \log\log\log d - \log\log\log \log d -\log A+O(1) \big).
\]

Finally we remark on a related problem, concerning
the distribution of the phase $\text{arg} (\zeta (1+it))$.
In 1972  Pavlov and Faddeev \cite{PF72} observed  that 
the phase of $\zeta(1+ it)$ appears
in the scattering matrix data for the hyperbolic Laplacian operator acting on
the modular surface $X(1) =  PSL(2, \ZZ)\backslash \HH,$  where $\HH=\{ z=x+iy : y= Im(z)>0\}$ denotes
the upper half plane.   Then in 1980 Lax and Phillips \cite[Sect. Theorem 7.19]{LP80}  treated this case  as an
example of  their version of scattering theory for automorphic functions (given in \cite{LP76}).
Recently Y. Lamzouri \cite{Lam08}  obtained interesting results on  the joint distribution of
the modulus and phase  $(|\zeta(1+it)|, \arg(\zeta(1+it)))$.

%
%
%
%
\subsection{Euler's constant and random permutations: cycle structure}\label{sec38}
\setcounter{equation}{0}

Let $S_N$ denote the symmetric group of all permutations on $[1, N] := \{ 1, 2, 3,  \cdots, N\}.$
By a random permutation we mean an element  $\sigma \in S_N$ drawn with the uniform distribution,
picked with probability $\frac{1}{N!}$. We view $\sigma$ as a product of distinct cycles,
and let $c_j(\sigma)$ count the number of cycles of length $j$ in $\sigma$. 
Three interesting statistics on the cycle structure of a  permutation $\sigma \in S_N$ are 
its total number of cycles 
$$
n(\sigma) := \sum_{j=1}^n c_j, 
$$
the length of its longest cycle
$$
M(\sigma) := \max\{j: \,c_j >0\},
$$
 and the length of its shortest cycle
$$
m(\sigma) := \min\{ j: \,c_j >0\}.
$$
Euler's constant $\gamma$ appears in connection with the
the distributions of each of these statistics, viewed as random variables on $S_N$.

In 1939  J.  Touchard \cite[p. 247]{Tou39} expressed the distribution of cycle lengths
using exponential generating function methods.  In 1944 Goncharov \cite[Trans. pp. 31-34]{Gon44}
(announced 1942 \cite{Gon42}) also gave generating functions, derived in a probabilistic context. He showed that
 if one lets $c(N,k)$ denote the number of permutations in $S_N$ having
exactly $k$ cycles, then one has the generating function
\begin{equation}\label{CGF}
\sum_{k=1}^N c(N,k) x^k = x(x+1)(x+2) \cdots (x+N-1).
\end{equation}
Goncharov computed the mean and variance of the number of cycles $n(\sigma)$, as follows.\smallskip


\begin{theorem}~\label{th380} {\rm (Touchard 1939, Goncharov 1944)}
Draw a random permutation  $\sigma$ from the symmetric group $S_N$ on $N$ elements with the uniform distribution.

(1) The expected value of the number of cycles in $\sigma$ is:
$$
E[n(\sigma)] = H_N = \sum_{j=1}^N \frac{1}{j}.
$$
In particular one has the  estimate
$$
E[n(\sigma)] = \log N + \gamma +O(\frac{1}{N}).
$$

(2) The variance of the number of cycles  in $\sigma$ is:
$$
Var[ n (\sigma)] := E[ (n(\sigma) - E[n(\sigma)])^2] =H_{N} - H_{N, 2} = \sum_{j=1}^N \frac{1}{j} - \sum_{j=1}^N \frac{1}{j^2}.
$$
In particular one has the estimate 
$$
 \sqrt{Var[ n (\sigma)]}= \sqrt{ \log N} + \Big(\frac{\gamma}{2} -\frac{ \pi^2}{12}\Big) \frac{1}{\sqrt{\log N}}  + O((\log N)^{-\frac{3}{2}}).
$$
\end{theorem}

\begin{proof}
In 1939 Touchard \cite[p. 291] {Tou39} obtained the formula
$E[n(\sigma)] = H_N$.
In 1944 Goncharov  \cite[Sect. 15]{Gon44}  derived both (1) and (2), using \eqn{CGF}.
A similar proof was given in 1953 by Greenwood \cite[p. 404]{Gre53}.
  Note also that for all $J \ge 1$ the  $J$-th moment of $n(\sigma)$ is given by
$$
E[ n(\sigma)^J] = \frac{1}{N!}(x \frac{d}{dx})^J M_N(x) \, |_{x=1}.
$$
It follows using \eqn{CGF} that the $J$-th moment can be expressed as
a polynomial in the $m$-harmonic numbers $H_{N, m}$ for $1\le m \le J.$

Goncharov \cite[Sect. 16]{Gon44} used the moments to 
  show furthermore that  the   normalized  random variables 
$ \bar{n}(\sigma) := \frac{n(\sigma)- E[n(\sigma)]}{\sqrt{Var[n(\sigma)]}}$
satisfy a central limit theorem as $N \to \infty$, i.e. they converge in distribution to 
standard unit normal distribution.
\end{proof}

 Goncharov \cite{Gon44}  also  obtained information on the distribution
of the maximum cycle $M(\sigma)$ as well. 
Define the  scaled random 
variable $L_1^{(N)} :=\frac{1}{N} M(\sigma)$. Then, as $N \to \infty$, 
Goncharov showed these random variables have a limiting distribution, with 
the  convergence in distribution
$
L_1^{(N)}   \rightarrow_{d} ~~~\LL_1,
$
in which $\LL_1$ is the probability distribution supported on $ [0,1]$ whose cumulative distribution
function 
$$
F_1(\alpha):= \mbox{Prob}[ 0 \le \LL_1 \le \alpha] 
$$
 is given for $0 \le \alpha \le 1$ by 
  \begin{equation}\label{gonch}
 F_1(\alpha) =  1+ \sum_{k=1}^{\lfloor \frac{1}{\alpha} \rfloor} \frac{(-1)^{k}}{k!}
 \int_{{t_1+ \cdots + t_k \le \frac{1}{\alpha}}\atop{t_1, \cdots, t_k \ge 1}} \frac{dt_1}{t_1} \frac{dt_2}{t_2} \cdots \frac{dt_n}{t_n}.
 \end{equation}
with $F_1(\alpha) =1$ for $ \alpha \ge 1$.  
Much later Knuth and Trabb-Pardo \cite[Sec. 10]{KTP76}  in 1976 showed that this
distribution is connected to the Dickman function $\rho(u)$ in Section
 \ref{sec34} by 
\begin{equation}\label{gonch1}
F_1(\alpha) = \rho(\frac{1}{\alpha}), ~~~ 0 \le \alpha \le 1,
\end{equation}
and this  yields the previously given formula \eqn{526e}.
From this connection we deduce that this   distribution has a continuous density  $f_1(\alpha) \,d\alpha$ given by
  $$ f_1(\alpha) = -\rho'(\frac{1}{\alpha}) \frac{1}{\alpha^2}
  = \rho(\frac{1- \alpha}{\alpha})\frac{1}{\alpha},$$
with  the last equality deduced using the differential-difference equation \eqn{524d} for $\rho(u)$.

In 1966 Shepp and Lloyd \cite{SL66} obtained detailed information on  the distribution of both
$m(\sigma)$ and $M(\sigma)$
and also of the $r$-th longest and $r$-th shortest cycles. 
Euler's constant appears in the distributions of longest and shortest cycles, as follows.\smallskip


\begin{theorem}~\label{th92A} {\em  (Shepp and Lloyd 1966)}
Pick a random permutation  $\sigma$ on $N$ elements with the uniform distribution.

(1) Let $M_r(\sigma)$ denote the length of the $r$-th longest cycle under iteration of
the permutation $\sigma \in S_N$. Then the $k$-th moment  $E[M_r(\sigma)^k]$ satisfies,
for $k \ge 1$ and $r \ge 1$,  as $N \to \infty$, 
\begin{equation*}
\lim_{N \to \infty}  \frac{E[M_r(\sigma)^k]}{N^k} = G_{r, k}\, ,
\end{equation*}
for positive limiting constants $G_{r,k}$. These constants  are explicitly given by
\begin{equation}\label{A913}
G_{r,k} =  \int_{0}^{\infty}\frac{ x^{k-1}}{k!} \frac{Ei(-x)^{r-1}}{(r-1)!} e^{ -Ei(-x)} e^{-x} dx,
\end{equation}
in which $Ei(-x)= \int_{x}^{\infty} \frac{e^{-t}}{t} dt$ is the exponential integral.
They satisfy, for fixed $k$, as $r \to \infty$, 
\begin{equation*}
\lim_{r \to \infty} (k+1)^r G_{r,k} = \frac{1}{k!} \, e^{-k \gamma},
\end{equation*}
in which $\gamma$ is Euler's constant.

(2)  Let $m_r(\sigma)$ denote the length of the $r$-th shortest cycle of the permutation $\sigma \in S_N$.
Then the expected value $E[m_r(\sigma)]$ satisfies, as $N \to \infty$, 
\begin{equation}\label{A911}
\lim_{N \to \infty}  \frac{E[m_r(\sigma)]}{(\log N)^r} = \frac{1}{r!} e^{-\gamma}.
\end{equation}
Furthermore  the $k$-th moment  $E[m_r(\sigma)^k]$ for $k \ge 2$
and $r \ge 1$  satisfies, as $N \to \infty$, 
\begin{equation}\label{A911b}
\lim_{N \to \infty}  \frac{E[m_r(\sigma)^k]}{N^{k-1} (\log N)^{r-1}} =
\frac{1}{(r-1)!} \int_{0}^{\infty} \frac{x^{k-1}}{(k-1)!} e^{Ei(-x)} e^{-x} dx.
\end{equation}
\end{theorem}

\begin{proof}
The exponential integral $Ei(x)$, defined for $x<0$ is given in \eqn{561b}.
Here (1) appears  in  Shepp and Lloyd \cite[eqns. (13) and (14) ff., p. 347]{SL66}.
Here (2) appears as   \cite[eqn. (22), p. 352]{SL66}.
\end{proof}

 Shepp and Lloyd \cite{SL66} also obtained the limiting distribution of the
$r$-th longest cycle as $N \to \infty$.
In terms of the scaled variables $L_r^{(N)}:=\frac{1}{N}M_r(\sigma)$, they deduced
the  convergence in distribution
$$
L_r^{(N)}  \longrightarrow_{d} \LL_r,
$$
in which $\LL_r$ is a distribution supported on $[0, \frac{1}{r}] $ whose cumulative distribution
function 
$$
 F_r(\alpha):= \mbox{Prob}[ 0 \le \LL_1 \le \alpha] , ~~~~\quad \quad 0 \le \alpha \le \frac{1}{r}, 
 $$
 is given by
 $$
 F_r(\alpha) =1+  \sum_{k=r}^{\lfloor \frac{1}{\alpha} \rfloor} \frac{(-1)^{k-r+1}}{(r-1)! (k-r)! k}
 \int_{{t_1+ \cdots + t_k \le \frac{1}{\alpha}}\atop{t_1, \cdots, t_k \ge 1}} \frac{dt_1}{t_1} \frac{dt_2}{t_2} \cdots \frac{dt_n}{t_n},
 $$
 and we set  $F_{r}(\alpha) =1$ for $\frac{1}{r} \le \alpha < \infty$.
This extends to general $r$  the case $r=1$ treated by Goncharov \cite{Gon44}. 
In 1976  Knuth and Trabb-Pardo  \cite[Sec. 4]{KTP76} defined $\rho_r(u)$  by
$\rho_r(u) :=F_r(\frac{1}{u})$, so that  $\rho_1(u) = \rho(u)$ by \eqn{gonch1}.
The functions $\rho_r(u)$ for $r \ge 1$ are uniquely determined by the conditions
\begin{enumerate}
\item 
(Initial condition) For $0 \le u \le 1$, it satisfies
\begin{equation*}
\rho_r(u) =1.
\end{equation*}
\item
 (Differential-difference equation) For $u > 1$, 
 \begin{equation*}
u \rho_r^{'}(u) = - \rho_r(u-1) + \rho_{r-1}(u-1),
\end{equation*}
with the convention $\rho_0 (u) \equiv 0.$
\end{enumerate}
Knuth and Trabb-Pardo  formulated these conditions in the integral equation form
$$
\rho_r(u) = 1 - \int_{1}^{u} ( \rho_r(t-1) - \rho_{r-1}(t-1))\frac{dt}{t}.
$$
One may  directly check that these equations imply $\rho_r(u) =1$ for $0 \le u \le r$.

Associated to the longest cycle  distribution is another interesting constant
$$
\lambda := G_{1,1} = \lim_{N \to \infty} \frac{E[M(\sigma)]}{N}.
$$
The formula \eqn{A913} gives
\begin{equation}\label{A917}
\lambda = \int_{0}^{\infty} e^{ -x - Ei(-x)} dx =  0.62432~99885 \dots
\end{equation}
This constant $\lambda$  is now
named the {\em Golomb-Dickman constant} in  Finch \cite[Sec. 5.4]{Fin03},
for the following reasons.
It was  first encountered in  1959  work of Golomb,  Welch and Goldstein  \cite{GWG59},
in a study of  the asymptotics of statistics associated to shift register sequences.
Their paper expressed this  constant  by an integral involving
a solution to a differential-difference equation,
described later in Golomb \cite[p. 91, equation (33)]{Gol67}.
In 1964 Golomb \cite{Gol64} asked for a closed form for this constant, and
the  Shepp and Lloyd's work  answered it with  \eqn{A917} above.
It was evaluated to $53$ places by Mitchell \cite{Mit68} in 1968.
In 1976 Knuth and Trabb-Pardo \cite[Sect. 10]{KTP76} 
showed that 
\begin{equation}\label{A918a}
\lambda = 1- \int_{1}^{\infty} \frac{\rho(u)}{u^2} du,
\end{equation}
where $\rho(x)$ is the Dickman function discussed in Section \ref{sec34}.
This associates the name of Dickman with this constant.
In fact the constant already appears in de Bruijn's 1951 work \cite[(5.2)]{deB51b}, where he showed
\begin{equation*}
\lambda = \int_{0}^{\infty} \frac{\rho(u)}{(u+1)^2} du,
\end{equation*}
as noted Wheeler \cite[p. 516]{Whe90}.
According to Golomb \cite[p. 192]{Gol67} another integral formula for this constant,
found by Nathan Fine,  is
\begin{equation}\label{918c}
\lambda = \int_{0}^1  e^{ Li(x)} dx,
\end{equation}
in which  $Li(x) := \int_{0}^x \frac{dt}{\log t}$ for $0 \le x <1.$  In 1990 Wheeler \cite[pp. 516--517]{Whe90}
established the formula
\begin{equation*}
\lambda = \int_{0}^{\infty} \frac{\rho(u)}{u+2} du
\end{equation*}
empirically noted earlier by Knuth and Trabb-Pardo, and also the formula
$$
\lambda = e^{\gamma} \int_{0}^{\infty} e^{-Ein(t) -2t}dt,
$$
in which $Ein(t)$ is given in \eqn{cexpi}.
The 
Golomb-Dickman constant $\lambda$ is not known
to be related to either Euler's constant $\gamma$ or the Euler-Gompertz constant $\delta$, and it
is not known whether it  is irrational or transcendental.

 The Golomb-Dickman
 constant $\lambda$ appears as a basic statistic in the  limit distribution $\LL_1$ for the scaled longest cycle,  
  whose expected value is given by 
  $$
 E [ \LL_1]  = \int_{0}^1 \alpha \,d F_1(\alpha) = - \int_{0}^1\frac{1}{\alpha}   \rho^{'} (\frac{1}{\alpha}) d\alpha.
 $$
  Letting $x= \frac{1}{\alpha}$ we obtain 
 \begin{equation*}
 E[\LL_1] =  - \int_{1}^{\infty} \frac{\rho'(x)}{x} dx = \int_{1}^{\infty} \rho(x-1) \frac{dx}{x^2} = \lambda,
 \end{equation*}
  see  \cite[(9.2)]{KTP76}.

In 1996 Gourdon \cite[Chap. VII, Th\'{e}or\`{e}me 2]{Gou96}
determined a complete asymptotic expansion of 
the expected values $E[M(\sigma)]$ as $N \to \infty$ in
powers of $\frac{1}{N}$, 
which both the Golomb-Dickman constant and Euler's constant appear, with the initial terms being
\begin{equation*}
E[ M(\sigma)]= \lambda N + \frac{1}{2} \lambda - \frac{e^{\gamma}}{24} \frac{1}{N} + 
O\left( \frac{1}{N^2}\right).
\end{equation*}
The omitted higher order terms for $k \ge 2$ in Gourdon's  asymptotic expansion  are of  form
$\frac{P_k(N)}{N^k}$ in which  each $P_k(N)$ is a oscillatory  function of $N$ which is periodic
with an integer period $p_k$
that grows with $k$ as $k \to \infty$.

A result in a different direction relates Euler's constant to the probability 
of all cycles having distinct lengths  as $N \to \infty$.
This result follows from work of  D. H. Lehmer \cite{Leh72}.


\begin{theorem}~\label{th380b} {\rm (Lehmer 1972)}
Draw a random permutation  $\sigma$ on $N$ elements with the uniform distribution.
Let $P_{d}(N)$ denotes the probability that  the cycles of $\sigma$
have distinct lengths, then 
$$
\lim_{N \to \infty} P_d(N) = e^{-\gamma}.
$$
\end{theorem}

\begin{proof}
 The  cycle lengths of a permutation $\sigma$ form a partition $\blambda= (1^{c_1},  2^{c_2}, \cdots, N^{c_N})$ of $N$
( written $\blambda \vdash N$), having $n(\lambda)= c_1+c_2 +\cdots+ c_N$ cycles  
with cycle length $j$ occurring $c_j$ times.
The number of permutations  $N(\blambda)$ giving rise to a given partition $\lambda$  is
 $$N(\blambda) = \frac{N!}{c_1! c_2!\cdots c_N! 1^{c_1} 2^{c_2} \cdots N^{c_N}},$$
 and by definition we have
 $$
 \frac{1}{N!} \sum_{\blambda \vdash N} N(\blambda)= 1.
 $$
In the case that all cycles are of distinct lengths, then all $c_i=0$ or $1$, and the formula for $N(\lambda)$ becomes 
  $$N(\blambda) = \frac{N!}{a_1 a_2 \cdots a_n},$$
 in which  $a_1 < a_2 < \cdots < a_n$ with $n=n(\lambda)$  are the lengths of the
 cycles.   (The division by  $a_i$ reflects  the fact that a cyclic shift of a cycle is the same cycle.)
   Lehmer \cite{Leh72} 
    assigned  weights   $(1^{c_1} 2^{c_2} \cdots N^{c_N})^{-1}$ to each partition,
  and counted different sets of such weighted partitions.
   Theorem 2 counted partitions with distinct summands,  in his notation  
 $$
 W_N^{\ast} := \sum_{\blambda \vdash N}{}^{'} \frac{1}{a_1 a_2 \cdots a_n},
 $$
with the prime indicating the sum runs over partitions having distinct parts.
 and showed that  $W_N^{\ast} \to e^{-\gamma}$ as $N \to \infty$. (Lehmer's weights
 agree with  the weights assigned to  partitions from the uniform distribution on  $S_N$ only for  those partitions
 having distinct parts.)
 \end{proof}

%
%
%
%
\subsection{Euler's constant and random permutations: shortest cycle}\label{sec38a}
\setcounter{equation}{0}

There are  striking parallels
 between   the distribution
 of cycles of  random permutations of $S_N$
and  the distributions of factorizations of random numbers 
  in an interval $[1,n]$, particularly concerning the number of
  these having factorizations of restricted types,
  allowing either factorizations with only small factors, treated in Section \ref{sec34},
  or having only large prime factors, treated in Section \ref{sec35}, where we take
  $N$ to be on the order of $\log n$. This relation was first observed by  
 Knuth and Trabb-Pardo \cite{KTP76}, whose 1976 work (already mentioned in Sections \ref{sec34} and  \ref{sec38})
was done
 to analyze the performance
of an algorithm for factoring integers by trial division. They 
observed in \cite[Sect. 10, p. 344]{KTP76}:  

\begin{quote} Therefore, if we are factoring the digits of a random $m$-digit 
number, the distribution of the number of digits in its prime factors is {\em approximately the
same as the distribution of the cycle lengths in a random permutation} on $m$ elements!
(Note that there are approximately $\ln m$ factors, and $\ln m$ cycles.)
\end{quote}
 They uncovered this  connection after noticing  that the Shepp-Lloyd formula for 
 Golomb's constant
 $\lambda$ in \eqn{A917},
 as evaluated by Mitchell \cite{Mit68},
 agreed  to $10$ places with 
 their calculation of Dickman's constant, as defined by \eqn{A918a}.
  A comparison of the $k$-th longest cycle formulas 
of Shepp and Lloyd in Theorem \ref{th92A} for cycles of length $\alpha N$ then revealed
matching  relatives  $\rho_{k}(x)$ of the Dickman function (defined  in Section \ref{sec34}).
These parallels extend to  the appearance of the Buchstab
function in the distribution of random permutations having no short cycle,
as we describe below.

 Concerning the shortest cycle, recall first that a  special case of Theorem  \ref{th92A}(2)
 shows  that the expected length of the shortest cycle is 
$$
E[m(\sigma)] =  e^{-\gamma} \log N (1 +o(1)) ~~\mbox{as}~~ N \to \infty,
$$

It is possible to get exact combinatorial formulas for the probability
$$
P(N, m) := \mbox{\rm Prob}[ \, m(\sigma) \ge m :\, \sigma \in S_N].
$$
 that a permutation has no cycle of length shorter than $m$.
 There are several limiting behaviors of this probability as $N \to \infty$,
 depending on the way $m=m(N)$
 grows as $N \to \infty$. There are three regimes, first, where  $m$ is
 constant, secondly where $m(N) = \alpha N$ grows proportionally   to $N$,
 with $0< \alpha \le  1$ (i.e. only long cycles occur) and thirdly,
 the intermediate regime, where $m(N) \to \infty$ but $\frac{m(N)}{N} \to 0$
 as $N \to \infty$.

First, suppose that  $m(N)=m$ is constant.
 This case has a long
history.  The  {\em derangement problem} ({\em Probl\`{e}me des m\'{e}nages}),
concerns the probability that a permutation
has no fixed point, which is the case $m=2$.
It was  raised by R\'{e}mond de Montmort \cite{Mon1708} in 1708
 and solved by Nicholas Bernoulli  in 1713, see  \cite[pp. 301--303]{Mon1713}.
 The derangement  problem was also solved by Euler in 1753 \cite[E201]{E201}, who
 was unaware of the earlier work.
 The well known answer is that
\begin{equation*}
P (N, 2) = \sum_{j=0}^N \frac{ (-1)^j}{j!}, 
\end{equation*}
and this exact formula yields the  result 
\begin{equation*}
\lim_{N \to \infty} P(N, 2) = \frac{1}{e}.
\end{equation*}
In 1952  Gruder \cite{Gru52} gave a generalization for permutations
having no cycle shorter than a fixed $m$, as follows.

\begin{theorem}~\label{th392a} {\rm (Gruder 1952)} 
For   fixed $m \ge 1$ there holds
\begin{equation}\label{391fm}
\lim_{N \to \infty} \mbox{\rm Prob}[\,m(\sigma) \,\ge m: \, \sigma \in S_N ] = e^{- H_{m-1}},
\end{equation}
in which $H_m$ denotes the $m$-th harmonic number.
\end{theorem}

\begin{proof}
Let $P_N(m, k)$ count the number of permutations of $N$ having exactly $k$ cycles, each of
length at least $m$, and set $P_N(m) = \sum_{k=1}^N P_N(m, k),$  so that $P(N, m) = \frac{P_N(m)}{N!}.$
These are given in the exponential generating function
\begin{eqnarray*}
\sum_{N=0}^{\infty} \frac{z^N}{N!} \Big(\sum_{k=1}^N P_N(m, k) u^k\Big)
 &=& \exp \Big[ u (\log \frac{1}{1-x} - \sum_{j=1}^{m-1} \frac{x^j}{j}) \Big]\\
& = & \frac{ \exp\Big[ -u\Big( x+ \frac{x^2}{2} + \cdots + \frac{x^{m-1}}{m-1}\Big) \Big]}{ (1-x)^u}.
\end{eqnarray*}
 Gruder \cite[Sect. 8, (85)]{Gru52} derives from the case $u=1$, 
$$
\sum_{N=0}^{\infty} P_N(m) \frac{x^N}{N!} = \frac{ \exp \left( -x- \frac{x^2}{2} - \cdots - \frac{x^{m-1}}{m-1}\right)}{1-x}.
$$
Using this fact Gruder \cite[Sect. 9, eqn. (94)]{Gru52}   derives  $\lim_{N \to \infty} \frac{N!}{P_N(m)} = e^{H_{m-1}}$,
giving the result.
\end{proof}

 Gruder also observes that one obtains Euler's constant from these values via the scaling limit
$$
 \lim_{m \to \infty} \lim_{N \to \infty} m P(N, m) = e^{-\gamma}.
$$
An asymptotic expansion concerning convergence to the limit \eqn{391fm}
for constant $m$ is now  available in  general circumstances.   
Panario and Richmond \cite[Theorem 3, eq. (9)]{PR01},
  give such an asymptotic expansion  for  the $r$-th largest cycle being $\ge m$ for a
 large class of probability models.

Secondly, for
 the  intermediate range where  $m=m(N) \to \infty$  with $\frac{m}{N} \to 0$ as $N \to \infty$, there is  the following
striking universal limit  
estimate,  in which $e^{-\gamma}$ appears.  \smallskip


\begin{theorem}~\label{th392} {\rm (Panario and Richmond 2001)}
Consider permutations on $N$ letters whose shortest cycle $m(\sigma) \ge m$.
Suppose  that  $m=m(N)$ depends on $N$ in such a way that   
$m(N) \to \infty$  and $\frac{m(N)}{N} \to 0$ 
both hold as $N \to \infty$.   
Under these conditions we have  
$$
\mbox{\rm Prob}[ \,m(\sigma) \ge m(N): \, \sigma \in S_N \, ] \sim  \frac{e^{-\gamma}}{m}, \quad\quad \mbox{as}~~~
N \to \infty.
$$
\end{theorem}

\begin{proof}
This result is deducible form  a general asymptotic result of 
 Panario and Richmond \cite[Theorem 3, eq. (11)]{PR01}, which
 applies to distribution of the $r$-th largest cycle of a logarithmic
 combinatorial structure.
 \end{proof}

Thirdly, we consider  the large range, concerning those  permutations  having shortest cycle of length at least a constant 
fraction  $\alpha N$ of the size of the permutation.  Here the resulting scaled density involves
the Buchstab function $\omega(u)$ treated in Section \ref{sec35}, with
 $\alpha= \frac{1}{u}$, with $u >1$. \smallskip


\begin{theorem}~\label{th393} {\rm (Panario and Richmond 2001)}
 Consider permutations on $N$ letters whose shortest cycle $m(\sigma) \ge m$.
Suppose  that  $m=m(N)$ depends on $N$ in such a way that  $\frac{m(N)}{N} \to \alpha$ 
holds as $N \to \infty$, with $0 < \alpha \le 1.$  
Under these conditions we have, for fixed $0 < \alpha \le 1$, that  
\begin{equation}\label{381aa}
\mbox{\rm Prob}[ \, m(\sigma) > \alpha N: \,\sigma \in S_N ] \sim 
\omega(1/\alpha)\frac{1}{\alpha N}, \,  \quad\quad \mbox{as} ~~~N \to \infty,
\end{equation}
where $\omega(u)$ denotes the Buchstab function.
\end{theorem}

\begin{proof}
 This result follows
 from a general result of   Panario and Richmond \cite[Theorem 3, eq. (10)]{PR01},
 which gives an asymptotic expansion  for  the $r$-th largest cycle. 
 The proof for this case  uses in part methods of Flajolet and Odlyzko \cite{FO90b}.
 \end{proof}

 The matching of the large range estimate \eqn{381aa} as $\alpha \to 0$ with the intermediate range
 estimate  follows from 
the limiting behavior of the Buchstab function $\omega(u) \to e^{-\gamma}$
as $u \to \infty.$ One may   compare  Theorem \ref{th392} with 
the sieving result in Theorem \ref{buchstab}.  They are quite parallel, and in both cases 
the factor  $e^{-\gamma}$ arises from the limiting 
 asymptotic behavior of the Buchstab function.

The  appearance of the Buchstab function in the sieving models in Section \ref{sec35} and
the permutation cycle structure above 
is precisely accounted for 
 in stochastic models 
developed by Arratia, Barbour, and Tavar\'{e} \cite{ABT97}, \cite{ABT99} in the late 1990's, which they termed  
Poisson-Dirichlet processes (and which we do not define here). 
They showed (\cite{ABT97},  \cite[Sec. 1.1, 1.2]{ABT03} that 
(logarithmically scaled) factorizations of random integers, and (scaled) cycle structures
of random permutations as $N \to \infty$
 lead to identical
limiting processes.
To describe the scalings, one  considers an ordered  prime  factorization of an 
 random integer $m = \prod_{i=1}^n p_i$, with $2 \le p_1 \le p_2 \le \cdots \le p_k$.
 Let $k := \Omega(m)$ count
 the number of prime factors of $m$ with multiplicity, one defines  the logarithmically scaled quantities
 $$
\bar{a}_i^{*}: = \frac{\log p_i}{\log m},  ~~~ 1 \le i \le k.
 $$
 The random quantities drawn are $(\bar{a}_1^{*}, \bar{a}_2^{*}, ..., \bar{a}_k^{\ast})$,
 where $k$ itself is also a random variable, and these quantities sum to $1$.
 To a random permutation $\sigma $ of $S_N$  let $a_1 \le a_2 \le \cdots \le a_{n(\sigma)}$
 represent the lengths of the cycles arranged in increasing order.
 Then associate to $\sigma$  the normalized cycle lengths
$$
\bar{a_i} := \frac{a_i}{N}, ~~~1 \le i \le n.
$$
Here $n=n(\sigma)$ is a random variable, and these quantities also  sum to $1$.
Arratia, Barbour and Tavar\'{e} \cite{ABT99}
  showed that in both cases, as $N \to \infty$ (resp. $n \to \infty$)  these random variables converge to
a limiting stochastic process, a Poisson-Dirichlet process of parameter $1$.
 They  furthermore
noted  (\cite[p. 34]{ABT03})
that both models required exactly the same normalizing constant: $e^{-\gamma}$.
The existence of a normalizing constant for Poisson-Dirichlet processes was noted
in 1977 by Vershik and Shmidt \cite{VS77}, who conjectured it to be $e^{-\gamma}$,
and this was proved in 1982 by Ignatov \cite{Ign82}.

The model for random cycles applies as well to random factorization of polynomials
of degree $N$ over the finite field $GF(q)$, as $q \to \infty$, see Arratia, Barbour and Tavar\'{e} \cite{ABT93}.
%
%
%
%
\subsection{Euler's constant and random finite functions }\label{sec39}
\setcounter{equation}{0}

Euler's constant also appears in  connection   with the distribution of  cycles in
a random finite function $F: [1, N] \to [1, N]$. 
 Iteration of a finite function on a given initial
seed leads to iterates of an element having a preperiodic part, followed by arrival at a cycle of the function.
Thus not every element of $N$ belongs to a cycle.

In 1968 Purdom and Williams \cite{PW68} established  results for random
finite functions that are analogous to results for random permutations given in Section \ref{sec38}.
They studied the length of the longest and shortest cycles for a random
function; note however that the expected length of the longest cycle grows 
proportionally to $\sqrt{N},$
rather than proportional to $N$ as in the case of a random permutation. \smallskip

\begin{theorem}~\label{th92B} {\rm (Purdom and Williams 1968)}
Pick a random function $F: [1, N] \to [1, N]$ with the uniform distribution
over all $N^N$ such functions.

(1) The length $M(F)$ of the longest cycle of $F$ under iteration
satisfies
\begin{equation*}
\lim_{N \to \infty}  \frac{E[M(F)]}{\sqrt{N}} = \lambda \sqrt{\frac{\pi}{2}},
\end{equation*}
in which $\lambda$ is the Golomb-Dickman  constant. 

(2) Let $m(F)$ denote the length of the shortest cycle under iteration of
the function $F$. Then the expected value $E[m(F)]$ satisfies
\begin{equation*}
\lim_{N \to \infty}  \frac{E[m(F)]}{\log N} = \frac{1}{2} e^{-\gamma}.
\end{equation*}
\end{theorem}

\begin{proof}
For fixed $N$, Purdom and Williams \cite{PW68} use the notation
$E[M(F)] = E_{F_N, 1}(\ell):=E[M(F)]$ and $E_{F_N, 1}(s):=E[m(F)]$.
The result  (1)  follows from   \cite[p. 550, second equation from bottom]{PW68}.
Purdom and Williams state a result related to (2)   without detailed proof. 
Namely, the  top equation on \cite[p. 551]{PW68} states
$$
E[m(F)] = S_{1,1} Q_{n}(1,1) + o(Q_n(1,1)),
$$
where
\begin{equation}\label{purdom}
Q_n(1,1) = \sum_{j=1}^n \frac{(n-1)! j (\log j)}{(n-j)! n^j} .
\end{equation}
Here the values $S_{r, k}$  
 are taken from  Shepp and Lloyd \cite{SL66}, where for $k=1$ they
 are given by the right hand side of  \eqn{A911},  and  for $k \ge 2$ 
 by the right hand side of  \eqn{A911b}, and in particular
$S_{1,1} = e^{-\gamma}.$ To estimate $Q_n(1,1)$ given by \eqn{purdom}
we use the identity ( \cite[p. 550]{PW68})
$$
Q_n(0) = \sum_{j=1}^n \frac{(n-1)! j}{(n-j)! n^j} =1.
$$
We observe that  $Q_n(1,1)$ is the expected value of the function $\log j$ with
respect to the  probability distribution given
by the individual  terms  in $Q_n(0)$. 
One may  deduce that $Q_n(1,1) = \frac{1}{2} \log n + o(\log n)$
by showing  that this probability distribution  is concentrated on values 
$j = \sqrt{n} + o(\sqrt{n})$. This gives (2) .
\end{proof}

The study of random finite functions is relevant for understanding the
behavior of the Pollard \cite{Pol75}  ``rho" method of factoring large integers.
The paper of Knuth and Trabb-Pardo \cite{KTP76} analyzed the performance 
of this algorithm, but their model used random permutations rather than
random functions. 
An   extensive study of limit theorems
for statistics for random finite  functions  was made by V. F. Kolchin \cite{Kol86}.
A very nice asymptotic analysis  of the
major statistics for random finite functions  is given in  Flajolet and Odlyzko \cite[Theorems 2 and 3]{FO90}.

%
%
%
%
\subsection{Euler's constant as a  Lyapunov exponent}\label{sec39b}
\setcounter{equation}{0}

Certain properties of Euler's constant  have  formal resemblance to properties of 
a  dynamical  entropy.
More precisely,  the quantity $C= e^{\gamma}$ might be interpretable as the growth rate of some
dynamical quantity, with $\gamma= \log C$ being the
associated  (topological or metric) entropy. 
Here we present results  about the growth rates of the
size of products of random matrices having normally distributed entires.
These growth rates naturally
involve $\gamma$, together with 
other known constants.

Before stating results precisely, we remark that products of random matrices often have
well-defined growth rates, which  can be stated either  in terms of
growth rates of matrix norms or in terms of Lyapunov exponents.
Let $\{ A(j): j \ge 1\}$  be a fixed sequence of $d \times d$ complex matrices.,
and let 
$M(n) = A(n) A(n-1) \cdots A(2) A(1)$ be  
 product of  $d \times d$ complex matrices. 
We define for an initial vector $\bv \in \CC^d$ and a vector norm $||\cdot||$, 
the {\em logarithmic growth exponent} of $\bv$ to be
\begin{equation}\label{700}
  \lambda(\bv) := \limsup_{n \to \infty} \frac{1}{n} \log {\frac{ ||M(n) \bv||}{||\bv||} }.
\end{equation}
This value depends on $\bv$ and on the particular sequence $\{ A(j): j \ge 1\}$ but is
independent of the choice of the vector norm.
For  a fixed infinite sequence
$\{ A(j): j \ge 1\}$  as $\bv$ varies there 
will at most $d$ distinct values  $\lambda_1^{*} \le \lambda_2^{*}\le \cdots \le \lambda_d^{*}$
for this limit.

Now let the sequence $\{ A(j): j \ge 1\}$
vary, with each $A(j)$ being
drawn independently from a given probability distribution on the set of $d \times d$ matrices.
In 1969 Furstenberg and Kesten \cite{FK60} showed that for almost all vectors $\bv$
(off a set of positive codimension) 
the limiting value  $\lambda_d^{*}$ takes with probability one
a constant value $\lambda_d$, provided the  distribution satisfied the condition
\begin{equation}\label{log-bound}
 \int_{X} \log^{+} ||A(x)|| d\mu(x) < \infty,
 \end{equation}
  where $(X, \Sigma, \mu)$ is the probability space, $||\cdot ||$ is a fixed matrix norm,
 required to be submultiplicative, i..e  $||M_1M_2|| \le ||M_1||||M_2||$, 
 and where $\log^{+} |x| := \max (0, \log |x|)$.
A corollary of the 
 multiplicative ergodic theorem proved by Oseledec in 1968 (see \cite{Ose68}, \cite{Rag79}) asserts that with probability one
 all  these logarithmic
  exponents take constant values, 
  provided the probability distribution on random matrices
  satisfies \eqn{log-bound}.
   These constant  values $\lambda_1 \le \lambda_2 \le \cdots \le \lambda_d$
 are called the  {\em Lyapunov exponents} of the random matrix product,
with  $\lambda_d$ being  the {\em maximal Lyapunov exponent}. 
 For almost all random products and almost all starting vectors $\bv_0$, 
 one will have  $\lambda(\bv) = \lambda_d$, compare  Pollicott \cite{Pol10}.
 The value $C_d := \exp(\lambda_d)$ may be thought of as approximating  
  the exponential growth rate of matrix norms $||M(n)^T M(n)||^{1/2}$,  as $n\to \infty$.

In general the Lyapunov exponents are hard to determine. In 1984
J. Cohen and C. M. Newman\cite[Theorem 2.5]{CN84} obtained the following
explicit result for normally distributed variables. Note that condition \eqn{log-bound}
holds for such variables.

%
\begin{theorem} \label{th71}{\em (Cohen and Newman 1984)}
Let $\{A(1)_{ij}: 1 \le i, j \le d\}$ be independent,  identically distributed (i.i.d.)  normal random variables 
with mean $0$ and variance $s$, i.e. $A(1)_{ij} \sim N(0, s^2)$. Let $\bv(0)$ be a nonzero vector in $\RR^d$.
Define for independent random draws
\begin{equation*}
\bv(n) = A(n) A(n-1) \cdots A(1) \bv(0),~~~n= 1, 2, 3, ...
\end{equation*}
and consider\footnote{Cohen and Newman write the left sides of \eqn{702a} 
and \eqn{703} as $\log \lambda$,  but here
we replace these by the rescaled variables
 $\lambda(\bv_0)$ and $\lambda$, in 
order  to match the definition of Lyapunov exponent given in \eqn{700}.}
\begin{equation}\label{702a}
 \lambda( \bv(0)) := \lim_{n \to \infty} \frac{1}{n} \log ||\bv(n)||
\end{equation}
using the Euclidean norm $||\cdot||$ on $\RR^d$. 
Then with probability one the limit exists in { \eqn{702a}}, is nonrandom, is  independent of $\bv(0)$,
and is given by
\begin{equation}\label{703}
 \lambda = \frac{1}{2} [ \log(s^2) + \log 2 + \psi(\frac{d}{2})], 
\end{equation}
in which $\psi(x)$ is the digamma function. Moreover as $n \to \infty$ the  random variables
$\frac{1}{\sqrt{n}} \log (e^{-\lambda n} ||\bv(n)||)$ converge in distribution to
$N(0, \sigma^2)$ with
\begin{equation*}
\sigma^2= \frac{1}{4} \psi'(\frac{d}{2}).
\end{equation*}
\end{theorem}

It is easy to deduce from  Cohen and Newman's result  the following consequence 
concerning $d \times d$ (non-commutative) random matrix products.


\begin{theorem} \label{th72}
Let $A(j)$ be  $d \times d$ matrices with entries drawn as  independent
identically distributed (i.i.d.)  normal random variables 
with variance $1$, i.e. $N(0, 1)$. 
Consider  the random matrix product 
$$
S(n) = A(n) A(n-1) \cdots A(1),
$$
and let  $||S||_{F} = \sqrt{\sum_{i,j}   |S_{i,j}|^2}$  be the Frobenius matrix norm.
Then, with probability one, the growth rate is given by
\begin{equation}\label{706}
\lim_{n \to \infty}  \left( ||S(n)||_{F}^2 \right)^{\frac{1}{n} }= \left\{
\begin{array}{ll}
2e^{-\gamma+ H_{d/2-1}} & ~~\mbox{if}~~ d ~~\mbox{is ~even},\\
~&~\\
\frac{1}{2} e^{-\gamma+ 2H_{d-1}- H_{(d-1)/2}} & ~~\mbox{if}~~ d ~~\mbox{is ~odd}.
\end{array}
\right. 
\end{equation}
Here
$H_{n}$ denotes the $n$-th harmonic number, using the  convention that $H_{0}=0.$
\end{theorem}

\begin{proof}
 We  use an exponentiated 
 version of the formula \eqn{703} in Theorem~\ref{th71}, 
 setting the variance $s=1$ and multiplying  both sides of \eqn{703} by 
 $2$ before exponentiating.
 We obtain, for a nonzero vector $\bv(0)$, that with probability one, 
 \begin{equation}\label{706b}
\lim_{n \to \infty}  \left( ||S(n)\bv(0)||^2 \right)^{\frac{1}{n} }= 2 e^{\psi(\frac{d}{2})}.
\end{equation}
where $||\cdot||$ is the Euclidean norm on vectors. To obtain the
Frobenius norm bound we use the identity
$||S||_F^2 = \sum_{j=1}^d || S \be_j||^2$
where $\be_i$ is the standard orthonormal basis of column vectors. 
On taking vectors $\bv_j(0)=\be_j$ we obtain
$$
||S(n)||_F^2 = \sum_{j=1}^d ||S(n) \bv_j(0)||^2,
$$
and using  \eqn{706b} and 
the fact that $\lim_{n \to \infty} d^{\frac{1}{n}} = 1$, we obtain with probability one that
\begin{equation*}
\lim_{n \to \infty}  \left( ||S(n)||_F^2 \right)^{\frac{1}{n} }= 2 e^{\psi(\frac{d}{2})}.
\end{equation*}
On applying the digamma formulas for $\psi(\frac{d}{2})$
given in Theorem~\ref{th30}, with the convention $H_{0}=0$, and noting for $d=2m+1$ that 
$$
2e^{\psi(d/2)} = \frac{1}{2} e^{-\gamma + 2 H_{2m-1}- H_{m-1}}= \frac{1}{2} e^{-\gamma + 2H_{2m} - \frac{2}{2m} - H_{m-1}}
= \frac{1}{2} e^{-\gamma + 2H_{2m}  - H_{m}},
$$
we obtain \eqn{706}.
\end{proof}

Rephrased in terms of Lyapunov exponents,
Theorem~\ref{th72} says that the corresponding $d \times d$ matrix system,
for odd dimension $d$ has top  Lyapunov exponent 
\begin{equation*}
\lambda_d = 
\frac{1}{2}\left( -\gamma - \log 2 + 2H_{d-1}- H_{\frac{d-1}{2}}\right),
\end{equation*}
and for even dimension $d$ has top Lyapunov exponent
\begin{equation*}
\lambda_d = \frac{1}{2} \left( -\gamma + \log 2 + H_{ \frac{d}{2} -1}\right).
\end{equation*}
The simplest expressions occur in dimension  $d=1$, giving
$\lambda_1= \frac{1}{2} (-\gamma-\log 2)$, and in 
$d=2$, giving $\lambda_2= \frac{1}{2}(-\gamma + \log 2)$.

Note that any real number can be obtained as a maximal Lyapunov exponent in Theorem~\ref{th71}
simply by adjusting the  variance parameter $s$ properly in the normal distribution $N(0,s)$. 
It follows that the 
content of the specific formulas in Theorem~\ref{th72} is in part {\em arithmetic}, where
one specifies the
variance $s$ of the normal distributions to be a  fixed  rational number.
Moreover, if we rescale $s$ to choose  normal distributions $N(0, s)$ with variance $s=2$ in dimension $1$ 
and with variance $s=\frac{1}{2}$ in  dimension $2$, then  we obtain the Lyapunov
exponents $\lambda_d= \frac{1}{2}\left(- \gamma \right)$  in these dimensions.

Returning to the theme of Euler's constant  being (possibly)
interpretable as  a dynamical entropy, 
one approach to the Riemann hypothesis suggests that it may be related to
an (as yet unknown) arithmetical dynamical system, cf. Deninger \cite{De92}, \cite{De93},
\cite{De98}, \cite{De00}.  There is  a statistical mechanics interpretation of the
Riemann zeta function as a partition function given in
Bost and Connes \cite{BC95},  with subsequent developments described in 
Connes and Marcolli \cite{CM08}.  For some other  views on 
arithmetical dynamical systems related to zeta functions 
see Lagarias \cite{La99}, \cite{La06} and  
Lapidus \cite[Chap. 5]{Lap08}.
Here we  note  the coincidence that in    
formulations of the Riemann hypothesis given earlier
in Theorems \ref{th54} and \ref{th55},
a ``growth rate" $e^{\pm \gamma}$  naturally appears.

%
%
%
%
\subsection{Euler's constant and periods}\label{sec310}
\setcounter{equation}{0}

In 2001 M. Kontsevich and D. Zagier \cite{KZ01} 
defined a  {\em period}  to be a complex number whose real and
imaginary parts are  values of absolutely convergent integrals of 
rational functions with rational coefficents, over domains in $\RR^n$
given by polynomial inequalities with rational coefficients.
More conceptually these are values of 
integrals of an algebraic differential form with algebraic coefficients,
integrated over an algebraic cycle, see Kontsevich \cite[Section 4.3]{Kon00}.
We will call such an integral a {\em period integral.}
They observed that  for the real part 
one can reduce such integrals to the case where all coefficients
are rational numbers, all differentials are top degree,  integrals are taken over divisors with
normal crossings  defined over $\QQ$. In particular they  introduced a ring $\sP$ of {\em effective periods},
which includes the field of all algebraic numbers $\overline{\QQ}$ as a subring, as
well as a ring  of {\em extended periods}  $\hat{\sP}$ obtained by adjoining  to $\sP$ the constant $\frac{1}{2 \pi i}$,
which conjecturally is not a period.

One has, for $n\ge 2$ the representation
\begin{equation*}
\zeta(n) = \int_{0< t_1 < t_2 <...< t_{n} < 1} \frac{dt_1}{1-t_1} \frac{dt_2}{t_2} \cdots \frac{dt_n}{t_n}
\end{equation*}
as a period integral. Thus convergent positive  integer zeta values are periods.
One more generally can  represent all (convergent) multiple zeta values 
\begin{equation*}
\zeta(n_1, n_2, ..., n_k) := \sum_{m_1> m_2 > ... > m_k >0}
 \frac{1}{m_1^{n_1} m_2^{n_2} \cdots m_k^{n_k}}.
\end{equation*}
at positive  integer arguments as period integrals.  Other period integrals are
associated with special values of $L$-functions, a direction formulated by
Deligne \cite{Del79}.   
An overview paper of Zagier \cite{Za94} gives a tour  of  
 many places that multiple zeta values appear.

Some other examples of periods given by Kontsevich and Zagier are
the values $\log \alpha$ for $\alpha$ a positive real
algebraic number.
In addition the values $(\Gamma(\frac{p}{q}))^q$ where
$\frac{p}{q}$ is a positive rational number are periods, see Kontsevich and Zagier \cite[p. 775]{KZ01}, 
 Andr\'{e} \cite[Chap. 24]{And04}. This can be deduced using Euler's Beta integral \eqn{250c}.
 Another well known example is (\cite[(5.11)]{Art64})
 $$
 \int_{0}^1 \frac{dt}{\sqrt{1-t^4}} = \frac{ (\Gamma(\frac{1}{4}))^2}{\sqrt{32 \pi}},
 $$
giving a  period of  an elliptic curve with complex multiplication.

Periods  arise in many places.
They arise in evaluating Feynman integrals in quantum field theory calculations.
 For more work on periods and zeta values in this context, see Belkale and Brosnan \cite{BB03}.
 Periods appear as multiple zeta values
 in Drinfeld's associator, cf. Drinfeld \cite{Dr89}, \cite{Dr90}, \cite{Dr92}. 
 Periods   appear in Vassiliev's knot invariants, due to the connection found by
 Kontsevich \cite{Kon93}, see Bar-Natan \cite{BN95}. They appear in
 expansions of Selberg integrals, cf. Terasoma \cite{Te02}.
 They appear in the  theory of motives as periods of mixed Tate motives, 
 cf. Terasoma \cite{Te02b}, \cite{Te06}, and Deligne and Goncharov \cite{DG05}.
 Other constants  that
 are periods are values of polylogarithms at integers.
 Algebraic multiples of beta values $B(\frac{r}{n}, \frac{s}{n})$ occur as periods of differentials on 
 abelian varieties, cf. Gross \cite[Rohrlich appendix]{Gross78}.
 In 2002  P. Cartier \cite{Ca02} gave a  Bourbaki seminar expos\'{e} surveying multiple zeta
 values, polylogarithms and related quantities. 
 
There is a good deal known about the irrationality or transcendence
of specific periods. The transcendence of $\zeta(2n)$ for $n \ge 1$ is immediate from
their expression as powers of $\pi$.
For odd zeta values, 
in 1978  Apery (\cite{Ap79}, \cite{Ap81})
 established that $\zeta(3)$ is irrational, see 
 van der Poorten \cite{vdP79} and Section \ref{sec311} for more details.
 In 1979
F. Beukers  \cite{Beu79} gave an elegant proof 
of the irrationality of $\zeta(3)$ suggested by the
form of Apery's proof,  showing  that certain integer linear combinations of  period integrals 
defined for integer $r, s \ge 0$ by 
\begin{equation*}
I_{r,s} :=\int_{0}^1 \int_{0}^1 \frac{ -\log xy}{1-xy}  x^ry^s \, dx dy= \int_{0}^1  \int_{0}^1  \int_{0}^{1} \frac{x^r y^s}{1-(1-xy)z} dz\, dx\, dy
\end{equation*}
were very small. He evaluated  
\begin{equation}\label{833a}
I_{0,0} =   \int_{0}^1 \int_{0}^1 \frac{ -\log xy}{1-xy}  \, dx\, dy  =  2 \, \zeta(3),
\end{equation}
 and for $r \ge 1$, 
\begin{equation*}
I_{r, r} = 2 \big( \zeta(3) - \frac{1}{1^3} - \frac{1}{2^3} - \cdots - \frac{1}{r^3}\big) = 2\big( \zeta(3) - H_{n,3}\big).
\end{equation*}
As a necessary part of
the analysis he  also showed for $r >s$ that each $I_{r, s}= I_{s,r}$ is a rational number with denominator  dividing $D_r^3$
where $D_r$ is the least common multiple $[1,2,..., r]$. 
It is now known that infinitely
many odd zeta values are irrational, without determining which ones (Rivoal \cite{Riv00}, Ball and Rivoal \cite{BR01}).
In particular Zudilin \cite{Zu01} showed
that at least one of
$\zeta(5), \zeta(7), \zeta(9), \zeta(11)$ is irrational.
See Fischler \cite{Fi04}
for a survey of recent  developments on irrationality of zeta values.
The state of the art for showing transcendence of periods is surveyed by
Waldschmidt \cite{Wa06}.

There is currently no effective way known to decide whether a
given number is a period. One can certify a number is a period by 
directly exhibiting
it as a $\bar{\QQ}$-linear combination of values of known period integrals.
 However no  properties of periods are currently known which
would be useful in distinguishing them from non-periods. Kontsevich and Zagier \cite[Problem 3]{KZ01}
raise the problem of exhibiting a single specific  
number that is provably not
a period.

As mentioned  in the  introduction it is conjectured: {\em Euler's constant is not
a period.} The conjectures that $e$ and $\gamma$ are not periods appears in
\cite[Sec. 1.1]{KZ01}. We have 
already seen a number of  different integral representations for Euler's
constant, which  however involve exponentials and/or logarithms, in addition
to rational functions. 
These integrals do  not resolve the question whether $\gamma$ is a period.

 Kontsevich and Zagier \cite[Sec. 4.3]{KZ01} also  suggest 
enlarging the class of periods  to allow  exponential functions in
the integrals. They define 
an {\em exponential period} to be
an absolutely convergent (multiple) integral of the product
of an algebraic function with the exponential of
an algebraic function, taken over a real semi-algebraic set, where all polynomials
entering the definition  of the algebraic functions
and the semi-algebraic set have algebraic coefficients\footnote{The semialgebraic set is
cut out by a system of  (real) polynomial inequalities $f_j(x_1, ..., x_n) \ge 0$,
and it is supposed that the coefficients of the $f_j$ are all real algebraic numbers.}.
The set of all such values forms a ring $\sEP$, which we
call  the ring of  {\em exponential periods}, and which contains
the field $\bar{\QQ}$ of all algebraic integers.
This ring is  countable, and is known to include  the constant $e$,
  all values  $\Gamma(\frac{p}{q})$
 at rational arguments, namely
 $\frac{p}{q} \in \QQ \smallsetminus \ZZ_{\le 0},$ 
and  the constant
 $$
 \sqrt{\pi} = \int_{-\infty}^{\infty} e^{- t^2} dt.
 $$
 It also includes  various values of period determinants  
 for confluent hypergeometric functions studied by Terasoma \cite[Theorem 2.3.3]{Te96}
 and by Bloch and Esnault \cite[Proposition 5.4]{BE00}.
 Here we observe that:
 \begin{enumerate}
 \item
 Euler's constant $\gamma \in \sEP$.
 \item
 The Euler-Gompertz  constant $\delta \in \sEP$.
 \end{enumerate}
 M. Kontsevich observes that 
an integral representation certifying that $\gamma \in \sEP$ 
is obtainable from  \eqn{225c} by substituting 
$-\log x = \int_{x}^1 \frac{dy}{y}$ for $0< x<1$ and $\log x= \int_{1}^x \frac{dy}{y}$ for $x>1$,
yielding
\begin{equation*}
\gamma = \int_{0}^1 \int_{x}^1 \frac{e^{-x}}{y} dy \,dx - \int_{1}^{\infty} \int_{1}^x \frac{e^{-x}}{y} dy \, dx.
\end{equation*}
An integral representation  certifying that   $\delta \in \sEP $ is given by  \eqn{560aa}, 
stating that
\[
\delta = \int_{0}^{\infty} \frac{e^{-t}}{1+t} dt.
\]
It is not known whether $\delta$ is a period.

%
%
%
%
\subsection{Diophantine approximations to Euler's constant}\label{sec311}
\setcounter{equation}{0}

Is Euler's constant rational or irrational? This is unknown. 

One can empirically test this possibility by obtaining good
rational approximations to Euler's constant.  The early calculations of
Euler's constant were based on  Euler-Maclaurin summation.
  In 1872 Glaisher \cite{Glaisher1872}  reviewed earlier
work on Euler's constant, which
included his own $1871$ calculation  accurate
to $100$ places (\cite{Glaisher1871}) and uncovered a mistake in an earlier calculation of Shanks. 
In $1878$ the astronomer J. C. Adams \cite{Ad1878}
determined $\gamma$ to $263$ places, again using Euler-Maclaurin summation,
an enormous labor done by hand.  In 1952 J. W. Wrench, Jr \cite{Wre52}, using a desk calculator,
 obtained $\gamma$ to $328$ places. 
In 1962 Knuth \cite{Kn62} automated this approach and obtained Euler's constant to  $1272$ places
with an electronic computer. 
However this record did not last for long. In 1963  D. W. Sweeney \cite{Sw63}  introduced  a new method to compute 
Euler's constant, based on the integral representation
$$
\gamma=  \sum_{k=1}^{\infty} \frac{(-1)^{k-1} n^k}{k! \,k} - \int_{n}^{\infty} \frac{e^{-u}}{u} du - \log n,
$$
By making  a suitable choice of $n$ in this formula, Sweeney obtained  $3566$ digits of $\gamma$. 
 The following year Beyer and Waterman \cite{BW74}  used a variant of this  method to compute
$\gamma$ to $7114$ places,  however only the initial $4879$ digits  of these turned out to be correct, as
determined by later computations. 
In 1977  Brent \cite{Bre77} used a modification of this method to compute  $\gamma$ and $e^{\gamma}$  to $20700$ places.
In 1980  Brent and McMillan \cite{BM80} formulated new algorithms based on the identity
$$
\gamma = \frac{U(n)}{V(n)} - \frac{K_0(2n)}{I_0(2n)},
$$
whose right hand side contains for $\nu=0$  the modified Bessel functions 
$$
I_{\nu}(z) := \sum_{k=0}^{\infty}\frac{(\frac{z}{2})^{\nu+2k}}{ k! \Gamma( \nu +k +1)} ~~\mbox{and} ~~
K_{0}(z) :=  -\frac{\partial}{\partial \nu} I_{\nu}(z) \vert_{\nu=0},
$$
and
\begin{eqnarray*}
U(n) & := &\sum_{k=0}^{\infty} \Big(\frac{n^k}{k!}\Big)^2 ( H_k - \log n),\\
V(n)& : = & I_0(2n) = \sum_{k=0}^{\infty} \Big( \frac{n^k}{k!}\Big)^2.
\end{eqnarray*}
The second term $0<K_0(2n)/I_0(2n)< \pi e^{-4n}$ can be made small by choosing $n$ large,
hence $\gamma$ is approximated by $\frac{U(n)}{V(n)}$, choosing $n$ suitably (their Algorithm B1).
With this algorithm they computed  both $\gamma$ and $e^{\gamma}$ to $30000$ places.
They also determined  the initial part of 
the  ordinary continued
fraction expansions  of both $\gamma$  and  $e^{\gamma}$,
and established the following result.\


\begin{theorem}~\label{th3141} {\rm (Brent and McMillan 1980)}

(1) If 
 $\gamma= \frac{m_0}{n_0}$ is rational, with $m_0, n_0 \in \ZZ_{>0}$ then
its denominator $n_0 > 10^{15000}.$

(2) If $e^{\gamma} = \frac{m_1}{n_1}$ is rational, with  $m_1, n_1 \in \ZZ_{>0}$ then
its denominator $n_1 > 10^{15000}.$
\end{theorem}

\begin{proof} 
We follow Brent \cite{Bre77}. 
For  any real number $\theta$   Theorem 17 of Khinchin \cite{Khinchin} states: if $\frac{p_n}{q_n}$
is an ordinary  continued fraction convergent for $\theta$ then $|q_n \theta - p_n| \le |q\theta -p|$ for all integers $p$ and $q$
with $0 < |q| \le q_n$. It therefore suffices to find a convergent $\frac{p_n}{q_n}$ with $q_n > 10^{15000}$ 
such that $\theta \ne \frac{p_n}{q_n}$, and the latter inequality holds if   $q_{n+1}$ exists.  
\end{proof}

The  continued fraction expansion of $\gamma$ begins
$$
\gamma = [0; 1, 1, 2, 1, 2, 1, 4, 3, 13, 5, 1, 1, 8, 1, 2, 4, 1, 1, 40, 1, 11, 3, 7, 1, 7, 1, 1, 5, 1, 49, 4,  ...],
$$
see \cite[Sequence A002852]{OEIS}. 
  Brent and McMillan also computed various statistics for the initial part of this  continued fraction, concerning
the number of partial quotients of different sizes of the initial continued fractions of $\gamma$ and of $e^{\gamma}$,
in order to test whether $\gamma$ behaves like a  ``random" real number.
They found good agreement with the predicted limiting distribution 
of partial quotients for a random real number $\theta= [a_0, a_1, a_2, \cdots]$.
It is a theorem of Kuz'min that  the distribution of the $j$-th partial quotient $\text{Prob}[a_j=k]$ 
 rapidly approaches the 
 Gauss-Kuz'min distribution 
$$
p(k) := \log_2 (1+ \frac{1}{k}) - \log_2( 1+ \frac{1}{k+1}),
$$
as $j \to \infty$, see  Khinchin \cite[III.15]{Khinchin}.  

Diophantine approximations to $\gamma$ can be extracted from various series expansions for $\gamma$
involving rational numbers. 
In 1910 Vacca \cite{Vac10} found the expansion
\begin{eqnarray*} 
\gamma &= &\sum_{k=1}^{\infty} (-1)^k \frac{ \lfloor \log_2 k\rfloor}{k}\\
& =& \Big(\frac{1}{2} - \frac{1}{3}\Big) + 2 \Big(\frac{1}{4} - \frac{1}{5} + \frac{1}{6} - \frac{1}{7}\Big) + 3 \Big( \frac{1}{8} - \frac{1}{9} + \cdots - \frac{1}{15}\Big) + \cdots.
\end{eqnarray*}
Truncating this expansion at $k=2m-1$ gives approximations $\frac{p_m}{q_m}$ to $\gamma$ satisfying
$$
    0 <  \gamma - \frac{p_m}{q_m} \le \frac{4(\log m + 1)}{m}.
$$
In 2010 Sondow \cite{Son10} gave refinements of this expansion yielding approximations
with a  convergence rate 
$O(\frac{\log m}{m^r})$ for an arbitrary but  fixed $r \ge 2$.
However approximations of these types fall far short  of establishing
 irrationality of $\gamma$,  since  the denominators $q_m$ grow
 exponentially with $m$.

Much recent work on Diophantine approximation 
of explicit constants $\theta$ has focused on finding 
 a series of rational approximations  $\frac{u_n}{v_n}$ with the properties:
 \begin{enumerate}
 \item
 Each sequence $u_n, v_n$ separately satisfies 
 the same  linear recurrence with coefficients that are polynomials in the ring $\ZZ[n]$,
 where $n$ is the recurrence parameter.
 \item
 The initial conditions for the recurrences $u_n, v_n$ are rational, so that both sequences
 consist of rational numbers. 
 \item
 One has $\frac{u_n}{v_n} \to \theta$ as $n \to \infty.$
 \end{enumerate}
 We give examples of such recurrences below, e.g. \eqn{3152a}. For such sequences 
the power series $U(z) = \sum_{n=1}^{\infty} u_n z^n$ resp. $V(z)= \sum_{n=1}^{\infty} v_n z^n$ will each (formally) satisfy  a 
homogeneous linear differential equation $D F(z)=0 $ in the $z$-variable,
whose coefficients are polynomials in  $\QQ[z]$, i.e. the operator $D$ has the form
$$
D=\sum_{j=0}^n R_j(z) \frac{d^j}{dz^j}.
$$
We will refer to such 
a sequence of approximations satisfying (1), (2) as being of  {\em $H$-type}, and will call any real number  obtainable
as a limit of such approximations an {\em elementary $H$-period}. More generally,
the  {\em ring  $\sHP$ of $H$-periods} will be the ring 
generated over $\overline{\QQ}$ by elementary $H$-periods\footnote{The name {\em $H$-period} is proposed
by analogy with the more studied class of $G$-periods given below.}.
The ring $\sHP$  is clearly countable, and its possible relations to  rings of periods $\sP$
or exponential periods $\sEP$  discussed in Section \ref{sec310} remain to be determined.

It turns out to be useful to  single out  a subclass of such numbers, which are the set
of such approximations   for which the linear differential operators $D$  have a
special property: \smallskip

{ \bf $G$-operator  Property.} {\em At some algebraic point $z_0 \in \overline{\QQ}$ the 
 linear differential operator with coefficients in $\CC[z]$ has 
 a full rank set of
a  $G$-function solutions. }\smallskip

The notion  of {\em  $G$-function}  is presented  in detail in Bombieri \cite{Bom81},
Andr\'{e} \cite{And89}, \cite{And03},  and
Dwork, Gerotto and Sullivan \cite{DGS94}. We give
a definition of a subclass of $G$-functions
obtained by specializing to the rational number case. (The general case allows
algebraic number coefficients all drawn from a fixed number field $K$.) 
It is a function given by a power
series $F(z) = \sum_{n=0}^{\infty} a_n z^n$, with rational $a_n$ such that
\begin{enumerate}
\item
There is a constant $C>0$ such that $|a_n| \le C^n$, so that the
power series $F(z)$ has a nonzero radius of convergence.
\item
There is  a constant $C' >0$ such that the least common multiple of the denominators of $a_1, .., a_n$
is at most $(C')^n$.
\item
The $a_n= \frac{p_n}{q_n}$ are solutions of a linear recurrence having coefficients that are
polynomials in the variable $n$ with integer coefficients.
\end{enumerate} 

The notion of $G$-operator
was formulated  in Andr\'{e} \cite[IV. 5]{And89}, \cite[Sect. 3]{And00a}.
Andr\'{e}'s definition of $G$-operator is given in 
another form, in terms of it
satisfying the Galochkin condition (from Galochkin \cite{Gal74}, also see Bombieri \cite{Bom81}), but this class of operators 
is known to coincide with those given by the definition above.
The  minimal differential operator satisfied by a $G$-function is known to be a
$G$-operator in Andr'{e}'s sense \cite[Theorem 3.2]{And00a},
by a result of D. and G.  Chudnovsky. Conversely, any $G$-operator in Andr\'{e}'s sense has
a full rank set of $G$-function solutions at any  algebraic point not a singular point
of the operator (\cite[Theorem 3.5]{And00a})..

Andr\'{e} \cite[p. 719]{And00a} has shown that $G$-operators, viewed in the complex domain,
are of a very restricted type. 
They must be  Fuchsian on the whole Riemann sphere  $\PP^1(\CC)$,
 i.e. they have only regular singular points, including the point  $\infty$,
 see also  \cite[Theorem 3.4.1]{And03}.
A conjecture formulated  in Andr\'{e} \cite{And89}, \cite[p. 718]{And00a}
is that all $G$-operators should come from  (arithmetical algebraic) ``geometry",
i.e. they each should be a product of  factors of Picard-Fuchs operators, controlling the
variation of cohomology in a parametrized family of algebraic varieties defined over $\bar{\QQ}$.

We will say that a  real number $\theta$  that
 is a limit of a convergent sequence of  rational approximations $\frac{u_n}{v_n}$ 
 with $u_n, v_n$ coefficients in 
 the power series expansions of two (rational) $G$-functions 
 is  an {\em elementary $G$-period}. 
 As above we  obtain a ring $\sGP$ of {\em $G$-periods}, defined as the ring
generated over $\overline{\QQ}$  by elementary $G$-periods. 
This ring   has recently been
shown in Fischler and Rivoal \cite[Theorem 3]{FR13},
to coincide with the field $\mbox{Frac}(\bG)$,
in which $\mbox{Frac}(\bG)$ is the fraction field 
of a certain ring $\bG$ constructed from $G$-function values. 
The ring  $\bG$ is  defined to be  the set of complex numbers $f(\alpha)$ where
$\alpha \in \bar{\QQ}$ and $f(z)$ is any branch of a (multi-valued)
analytic continuation of a $G$-function with coefficients defined over some  number
field 
which is nonsingular at $z=\alpha$.  
Fischler and Rivoal \cite[Theorem 1]{FR13} show that the ring $\bG$ coincides
with the set of all complex numbers  whose real and imaginary parts
can each  be written as $f(1)$, with  $f(z)$ being some 
$G$-function with rational coefficients (as in the definition above) and
whose power series expansion has radius of convergence exceeding $1$.
Their result \cite[Theorem 3]{FR13} also establishes
that the set of elementary $G$-periods is itself a ring containing all real
algebraic numbers, which coincides with
$\mbox{Frac}(\bG)\cap \RR$.
Again, possible relations of the field $\sGP$ with the Kontsevich-Zagier
rings of periods $\sP$ and extended periods $\sEP$
remain to be determined.\footnote{The paper \cite[Sect. 2.2]{FR13} discusses
the possibility   that the equality $\bG =\hat{\sP} =\sP [ \frac{1}{\pi}]$ might hold,
and sketches an argument due to a reviewer that supports the conjecture  that $\bG \subseteq \hat{\sP}$.}

As already mentioned in Section \ref{sec310}, a famous result of Ap\'{e}ry (\cite{Ap79}, \cite{Ap81})
in 1979 established that the period $\zeta(3)$ is irrational.
 Ap\'{e}ry's  result was
obtained by finding rational approximations $\frac{p_n}{q_n}$ to $\zeta(3)$ 
which were generated by two different solutions $\{ p_n, q_n: n \ge 1\}$, to the second order linear recurrence
with polynomial coefficients:
\begin{equation}\label{3152a}
n^3 \, u_n = (34n^3- 51n^2+27n -5) u_{n-1} - (n-1)^3 u_{n-2}
\end{equation}
with initial conditions $p_0=0, p_1=6$ and $q_0=1, q_1=5,$ respectively. Here $q_n$ are now
called {\em Ap\'{e}ry numbers}, and are given by 
\begin{equation}\label{3152b}
q_n = \sum_{k=0}^n {\binom{n+k}{k}}^2 {\binom{n}{k}}^2, 
\end{equation}
while $p_n$ are given by
\begin{equation}\label{3152c}
p_n = \sum_{k=0}^n {\binom{n+k}{k}}^2 {\binom{n}{k}}^2\Big( \sum_{m=1}^n \frac{1}{m^3} +
 \sum_{m=1}^k \frac{ (-1)^{m-1}}{ 2m^3 {\binom{n}{m}} {\binom{n+m}{m}}}\Big),
\end{equation}
see Fischler \cite[Sect. 1.2]{Fi04}). 
These approximations are
closely spaced enough and  good enough to imply that, for any $\epsilon>0$,  
and all $q \ge q(\epsilon)$ one has 
$$
|\zeta(3) - \frac{p}{q}| > q^{-\theta + \epsilon}  \mbox{for all}~~p \in \ZZ, 
$$
with exponent $\theta= 13.41782$,
which implies that $\zeta(3)$ is irrational.  
The method permits deriving  families of approximations given by related recurrences for $\zeta(m)$ 
for all $m \ge 2$ (cf. Cohen \cite{Coh81}) ,  but only for $m=2$ and $3$ are 
 they are good enough to prove irrationality. For each $m \ge 2$  the associated power series
$$
f(z) = \sum_{n \ge 0} p_n z^n,  ~~\qquad ~g(z) = \sum_{n=0}^{\infty} q_n z^n
$$
will be solutions to a linear differential equation with polynomial coefficients,
and are of $H$-type. For   $m=2, 3$ these approximations are known to be of $G$-type.
A conceptually new proof of irrationality of $\zeta(3)$
using the same approximations, but connecting them with modular forms on the
congruence subgroup $\Gamma_1(6)$ of $SL(2, \ZZ)$ was presented in 1987 by Beukers \cite{Beu87}.

Ap\'{e}ry's discovery  was based on series acceleration transformations
of the series $\zeta(3) = \sum_{n=1}^{\infty} \frac{1}{n^3}.$
An initial  example of Ap\'{e}ry of such accelerations is  the identity
\begin{equation}\label{3152d}
 \zeta(3) =\frac{5}{2} \sum_{n=1}^{\infty} \frac{(-1)^{n-1}}{n^3 {\binom{2n}{n}}},
\end{equation}
see  \cite[Sect. 3]{vdP79}.  The partial sums of this series  converge at
a rate $O( 4^{-n})$ to $\frac{2}{5} \zeta(3)$, but 
are by themselves  insufficient to prove irrationality.
 Ap\'{e}ry's proof  used further  series acceleration transformations, described in \cite[Sect. 4]{vdP79}.  
It is interesting to note that the identity \eqn{3152d} is a special case of identities
in an 1890  memoir of A. A. Markoff \cite{Markov1890} on series acceleration, as is explained in
Kondratieva and Sadov \cite{KS05}.   
In 1976 Gosper (\cite{Gos76},  \cite{Gos90}) noted  related identities,
also found by series acceleration techniques, 
including
\begin{equation*}
 \zeta(3) = \sum_{n=1}^{\infty} \frac{30n-11}{16(2n-1)n^3 {\binom{2n-1}{n}}^2}.
\end{equation*}

A new approach to  Ap\'{e}ry's  irrationality
proof  of $\zeta (3)$ was found by Nesterenko \cite{Nes96}
in 1996, which used a relation of the Ap\'{e}ry approximations $p_n, q_n$  to 
Pad\'{e} approximations to polylogarithms.
Nesterenko also   found
 a  new continued fraction for $2\zeta(3)$, 
\begin{equation}\label{3152g}
2 \zeta(3) =2+ \cfrac{1}{2+ \cfrac{2}{4+ \cfrac{1}{3+\cfrac{4}{2+ \cfrac{2}{4+\cfrac{6}{6+ \cfrac{4}{5+\cfrac{9}{4+
 \cdots}}}}}}}}.
\end{equation}
Writing this as $2 \zeta(3) = b_0 + \cfrac{a_1}{b_1  + \cdots}$, 
 the numerators $a_n$  grow quadratically in $n$ while the denominators $b_n$ 
grow linearly with $n$,  and are given for $k \ge 0$ (excluding $a_1$) by
$$
\begin{array}{llll}
 a_{4k+1}= k(k+1), & b_{4k+1}= 2k+2, &   a_{4k+2} = (k+1)(k+2),  & b_{4k+2}= 2k+4,   \\
  a_{4k+3}= (k+1)^2, & b_{4k+3}= 2k+3, &a_{4k+4} = (k+2)^2 ,& b_{4k+4}= 2k+2. 
\end{array}
$$
This result was obtained starting from   Pad\'{e} type simultaneous polynomial approximations to the logarithm,
dilogarithm and trilogarithm found 
by Gutnik \cite{Gut83},  see also the survey of  Beukers \cite{Beu81}. When specialized at
the point $x=1$ these Pad\'{e} approximations yield the  sequences $q_n, -2p_n$. 
The  continued fraction \eqn{3152g} 
was obtained from  a  Mellin-Barnes type integral for a special case of 
the  Meijer G-function \cite{Mei36}, 
a generalized hypergeometric function
(see  \cite[Chap. V]{EMOT53}).  A survey of methods using 
Pad\'{e} approximations to polylogarithms   is given in Nesterenko \cite{Nes03}.
This approach has led to further results, including a new proof  of irrationality of $\zeta(3)$
of Nesterenko \cite{Nes09}.

A natural question concerns whether  $\gamma$ has rational approximations of $H$-type.
A  family of such rational approximations to $\gamma$ were found in 2007  
 through  the combined efforts of A. I. Aptekarev, A. I. Bogolyubskii, D. V. Khristoforov, V. G. Lysov
and D. N. Tulyakov in the seven papers in the volume \cite{Apt07}.
Their result is  stated in Aptekarev and Tulyakov \cite{AT09} 
in the following  slightly improved form.\smallskip


\begin{theorem}\label{th3152} {\em (Aptekarev, Bogolyubskii, Khristoforov, Lysov, Tulyakov 2007)}
Euler's constant is approximated by
a ratio of two rational solutions of the   third-order recurrence relation with polynomial coefficients
\begin{eqnarray*} 
(16n-15) q_{n+1} &= & (128n^3 + 40n^2 -82n-45) q_n 
-  n^2(256n^3 -240n^2 +64n-7) q_{n-1}\nonumber \\
&&~~~ + (16n+1)n^2(n-1)^2 q_{n-2}.
\end{eqnarray*}
The  two solutions $\{p_n: n \ge 0\}$ and $\{ q_n: n \ge 0\}$ are determined by the  initial conditions 
$$
p_0 :=0, ~~~p_1 :=2, ~~~p_2 := 31,
$$
and
$$
q_0 := 1, ~~~q_1 :=3, ~~~q_2 := 50.
$$
They have the following properties:\smallskip

 (1)  (Integrality)  
 For  $n \ge 0,$   
 $$  p_n \in \ZZ,~~~q_n \in \ZZ. $$

 (2)  (Denominator growth rate) For $n \ge 1$, 
 \begin{equation*}
 q_n= (2n)! \frac{ e^{\sqrt{2n}}}{\sqrt[4]{n}}
 \left( \frac{1}{\sqrt{\pi}(4e)^{3/8}} +
  O ( \frac{1}{\sqrt{n}}) \right).
 \end{equation*}
 
 (3)  (Euler's constant approximation) For $n \ge 1$, 
\begin{equation*}
 \gamma- \frac{p_n}{q_n} = -2 \pi e^{- 2 \sqrt{2n}}\left( 1+ O (\frac{1}{\sqrt{n}})\right).
\end{equation*}
\end{theorem}

\noindent 
The original result showed in place of (1) the weaker result that $q_n \in \ZZ$
and $D_n p_n \in \ZZ$ with $D_n= \mbox{l.c.m.} [1, 2, ..., n]$.
The integrality result for $p_n$ follows from a later result of Tulyakov \cite{Tu09},
 which finds a more complicated system of recurrences that these sequences
 satisfy, in terms of which the integrality of $p_n$ is manifest. 
These approximants clearly are of $H$-type.

The results in Theorem~\ref{th3152}  are based on  formulas derived from a
 family of multiple orthogonal polynomials 
in the sense of Aptekarev, Branquinho and van Assche \cite{ABA03}. 
These lead to exact integral formulas
\begin{equation*}
q_n = \int_{0}^{\infty} Q_n(x) e^{-x}dx,
\end{equation*}
and
\begin{equation}\label{862}
p_n - \gamma q_n = \int_{0}^{\infty} Q_n(x) e^{-x}\log x \, dx,
\end{equation}
in which $Q_n(x)$ are the family of  polynomials (of degree $2n$) given by
\begin{equation*}
Q_n(x) = \frac{1}{(n!)^2} \frac{e^x}{x-1} \left( \frac{d}{dx}\right)^n x^n \left( \frac{d}{dx}\right)^n (x-1)^{2n+1} x^n e^{-x}.
\end{equation*}
Note for  $Q_0(x) = 1$ that \eqn{862} becomes the integral formula  \eqn{225c} for $\Gamma'(1)$.
Also in 2009  Rivoal \cite{Riv09} found  an alternative construction of these approximants 
  that makes use of Pad\'{e} acceleration methods
for series similar to Euler's divergent series \eqn{554a}. His method 
applies more generally to numbers $\gamma + \log x$, where $x>0$ is rational. 
In 2010 Kh. and T. Hessami Pilehrood \cite[Corollary 4]{HP09}
found closed forms for 
these approximants,
\begin{equation*}
q_n = \sum_{k=0}^n {\binom{n}{k}}^2 (n+k)!, \quad p_n = \sum_{k=0}^n {\binom{n}{k}}^2(n+k)! ( H_{n+k} + 2 H_{n-k} - 2 H_k),
\end{equation*}
where the $H_k$ are harmonic numbers. These  rational approximations $\frac{p_n}{q_n}$ 
thus are of  a  kind  similar in appearance to Ap\'{e}ry's
approximation sequence \eqn{3152b}, \eqn{3152c}  to $\zeta(3)$, and also satisfy a third order recurrence.

Recently Kh. and T. Hessami Pilehrood \cite[Theorem 1, Corollary 1]{HP13} constructed
an elegant sequence of rational approximations converging to Euler's constant,
also analogous to the Ap\'{e}ry   approximations to $\zeta(3)$, 
 with slightly better convergence properties
than those given in Theorem \ref{th3152}, 
and having a surprising connection with certain approximations to the Euler-Gompertz constant.
They set 
\begin{equation*}
q_n   =  \sum_{k=0}^n {\binom{n}{k}}^2 k! , \quad
p_n   =  \sum_{k=0}^n {\binom{n}{k}}^2 k! \,(2 H_{n-k} - H_k).
\end{equation*}
The sequence $q_n$ satisfies the homogeneous second order recurrence 
\begin{equation}\label{892b}
q_{n+2} = 2(n+2) q_{n+1} - (n+1)^2 q_n.
\end{equation}
and initial conditions $q_0=1, q_1=2$.
The sequence $p_n$ satisfies the inhomogeneous second order recurrence 
\begin{equation*}
p_{n+2} = 2(n+2) p_{n+1} - (n+1)^2 p_n -\frac{n}{n+2},
\end{equation*}
with initial conditions $p_0=0, p_1=1$, and it satisfies a homogeneous third order recurrence.


\begin{theorem}\label{th3153}  {\em (Kh. and T. Hessami Pilehrood 2013)} 
Let $(q_n)_{n \ge 0}, (p_n)_{n \ge 0}$ be defined as above.
Then $q_n \in \ZZ$ and  $D_n p_n \in \ZZ$, where $D_n= l.c.m. [1, 2, ..., n]$,
and for all $n \ge 1,$
\begin{equation*}
\gamma- \frac{p_n}{q_n} = -e^{-4 \sqrt{n}}\Big( 2\pi + O(\frac{1}{\sqrt{n}}) \Big).
\end{equation*}
Here  the growth rate of the sequence  $q_n$ is, for all $n \ge 1$, 
\begin{equation*}
q_n = n! \frac{e^{2 \sqrt{n}}}{\sqrt[4]{n}}\Big(\frac{1}{2 \sqrt{\pi e}} + O(\frac{1}{\sqrt{n}}) \Big)
\end{equation*}
\end{theorem}

These approximations converge to $\gamma$ from one side. The sequence $D_n$
needed to clear the denominator of $p_n$ is well known to have
growth rate 
$$
D_n = e^{n(1+o(1))},
$$ 
 an asymptotic  result that is equivalent 
in strength to the prime number theorem.

The Euler-Gompertz constant $\delta$ also has a convergent
series of rational approximations with  {\em  the same denominator sequence $q_n$}
as the Euler constant approximations above.  
The new numerators $s_n$ satisfy the same recurrence  \eqn{892b} as 
the denominators $q_n$, 
i.e. 
$$s_{n+2} = 2(n+2) s_{n+1} - (n+1)^2 s_n,$$
 but with  initial conditions $s_0=0, s_1=1$.
The $s_n$ are integers, and are also given by the expression
\begin{equation*}
s_n = \sum_{k=1}^n a(k-1) (k+1){\binom{n}{k}} \frac{(n-1)!}{(k-1)!}, 
\end{equation*}
in which the function values  $a(m)$ are
\begin{equation*}
a(m) = \sum_{k=0}^m (-1)^k \, k!,
\end{equation*}
see \cite[Sequence A002793]{OEIS}.
Here the $a(m)$  are the  partial sums of Euler's divergent series discussed in Section \ref{sec25}.
These approximations $\frac{s_n}{q_n}$ are exactly 
the partial quotients $\frac{s_n}{q_n}$ of the Laguerre continued fraction \eqn{249e}
for the Euler-Gompertz constant.  

 One  has the following rate of convergence estimate  for 
 this sequence of rational approximations (\cite[Theorem 2]{HP13}).


\begin{theorem}\label{th3154} {\em (Kh. and T. Hessami Pilehrood 2013)} 
Let $(q_n)_{n \ge 0},  (s_n)_{n \ge 0}$ be defined as above. 
Then $q_n , s_n \in \ZZ$, and for all $n \ge 1,$ the approximations $s_n/q_n$ to
the Euler-Gompertz constant satisfy
\begin{equation*}
\delta- \frac{s_n}{q_n} = e^{-4 \sqrt{n}}\Big( 2\pi e + O(\frac{1}{\sqrt{n}}) \Big).
\end{equation*}
\end{theorem}

The proof of Theorem~\ref{th3154} relates these approximations to values
of a Mellin-Barnes type integral
\begin{equation*}
I_n = (n!)^2 \frac{1}{2 \pi i} \int_{c-i\infty}^{c+i \infty} \frac{\Gamma(s-n)^2}{\Gamma(s+1)} ds.
\end{equation*}
where the vertical line of integration has $c>n$. The integral is expressed in
terms of  Whittaker function $W_{\kappa, \nu}(z)$ (a confluent hypergeometric function).
The proof displays some  identities
involving the Euler-Gompertz constant and  Whittaker function values, namely
\begin{equation*}
\delta = \frac{ W_{-1/2, 0} (1)}{W_{1/2, 0}(1)}
\end{equation*}
using 
\begin{equation*}
W_{-1/2, 0}(1) = \frac{\delta}{\sqrt{e}}, \quad\quad  W_{1/2, 0}(1) = \frac{1}{\sqrt{e}}.
\end{equation*}

Table \ref{tab3151} presents data on the approximations $(q_n, p_n, s_n)$ for small $n$   
providing good approximations to  $\gamma$
and $\delta$. Recall  that $p_n$ are rational numbers, while
$q_n, D_n p_n, s_n$ are integers, with $D_n= {\rm l.c.m.} [1,2, ..., n]$.

 \begin{table}\centering
\renewcommand{\arraystretch}{.85}
\begin{tabular}{|r|r|r|r|r|}
\hline
\multicolumn{1}{|c|}{$n$} &
\multicolumn{1}{c|}{$q_n$}& 
\multicolumn{1}{c|}{$D_np_n$}&
\multicolumn{1}{c|}{$D_n$}&

\multicolumn{1}{c|}{$s_n$}\\ \hline
0 &      1   &0 & 1 &0  \\ 
 1 &     2 &   $1$  & 1& 1 \\
 2 &     7  &  $ 8$   &  2&4    \\
 3 &     34&  $ 118$ &  6 &20 \\ 
 4 &   209  & $1450$ &12 & 124\\ 
 5 &    1546 & $53584$ & 60 & 920 \\ 
 6 &    13327  & $461718$ &60 & 7940 \\ 
 7 &     130922 & $31744896$ &  420 &78040\\ 
 8 &     1441729 & $699097494$ & 840 &859580 \\ 
 9 &      17572114& $25561222652$ & 2520 &10477880\\ 
10 &       234662231& $341343759982$ & 2520 &139931620 \\ \hline
\end{tabular}
\bigskip

\noindent \caption{Approximants $p_n/q_n$ and $s_n/q_n$ to $\gamma$ and $\delta$, with $D_n= {\rm l.c.m.} [1,2, ..., n]$.}
\label{tab3151}
\end{table}

The approximations given  in Theorems \ref{th3153}  and \ref{th3154},
certify that both $\gamma$ and $\delta$ are elementary $H$-periods, so
that both $\gamma, \delta \in \sHP$.
These approximations 
are of insufficient quality to imply the irrationality of either Euler's
constant or of the Euler-Gompertz constant. 
It remains a challenge to find a sequence of 
Diophantine approximations to $\gamma$ (resp. $\delta$) 
that approach them  sufficiently fast   to certify  irrationality.

One may  ask whether $\gamma$ has Diophantine approximations of $G$-type. This is unknown. The
current expectation is that this is not possible, i.e. that  $\gamma \not\in \sGP$. 
Based on the observation of Rivoal and Fischler that $\sGP= \mbox{Frac}(\bG)$, and the
 the belief (\cite[Sect. 2.2]{FR13}) that the ring
$\bG$ may coincide with the Kontsevich-Zagier ring $\hat{\sP}$ of
extended periods, this expectation amounts to   a stronger form of  Conjecture \ref{conj2}.

We conclude this section with recent results concerning $G$-type and $H$-type approximations
to other constants considered by Euler.
In $2010$ Rivoal \cite{Riv10b} showed 
that  the periods  $\Gamma(\frac{k}{n})^n$, where $\frac{k}{n} >0$ ,
 can be approximated  by sequences $\frac{p_n}{q_n}$ 
of rational numbers of the $G$-type, so are elementary $G$-periods.
In another paper Rivoal \cite{Riv10a} finds sequences
of rational approximations showing  that the individual numbers $\Gamma(\frac{k}{n})$,
which are not known to be periods  for $0< k<n$,  are elementary $H$-periods.  

Finally we note that various authors  have suggested other approaches to establishing
the irrationality of Euler's constant, e.g. Sondow \cite{Son03}, \cite{Son09}.
and Sondow and Zudilin \cite{SZu06}.

%
%
%
%
\subsection{Transcendence results related to Euler's constant}\label{sec312}
\setcounter{equation}{0}

The first transcendence results for numbers involving Euler's constant came from
the breakthrough of Shidlovskii \cite{Shi59}, \cite{Shi62} 
in the 1950's and 1960's on transcendental values of $E$-functions,
which is detailed in his book \cite{Sh89}.  The class of {\em $E$-functions} was introduced in 1929 by
Siegel \cite[p. 223]{Sie29},  consisting of those analytic functions  $F(z) = \sum_{n=0}^{\infty} c_n \frac{z^n}{n!}$
whose power series expansions have the following properties:
\begin{enumerate}
\item
The coefficients $c_n$ belong to a fixed algebraic number field $K$ 
(a finite extension of $\QQ$).
For each $\epsilon>0$,   the maximum of the absolute values of the algebraic
conjugates of $c_n$ is bounded by $O(n^{n\epsilon})$ as $n \to \infty$.
\item
For  each $\epsilon >0$
 such that  there is a sequence of integers $q_0, q_1, q_2 ...$ such that
$q_n c_n$ is an algebraic integer, with 
$q_n = O( n^{n \epsilon})$ 
as $n \to \infty$.
\item
The function $F(z)$ satisfies a (nontrivial) linear differential equation
$$
\sum_{j=0}^n R_j(z) \frac{d^j}{dz^j} y(z) = 0, 
$$
whose coefficients are polynomials $R_j(z) \in K_1[z]$, where $K_1$
is some algebraic number field.\footnote{This condition on $K_1$ could be weakened to require only
that the coefficients are in $\CC[z]$.}
\end{enumerate}
Each $E$-function is an entire function  of the complex variable $z$. The set of all $E$-functions
forms a ring $\bE$ (under $(+, \times)$)  which is also closed under differentiation, and this ring  is an algebra
over $\bar{\QQ}.$

We also define a  subclass of   $E$-functions  called {\em $E^{*}$-functions}
  in  parallel with the class of $G$-functions,  following Shidlovskii \cite[Chap. 13]{Sh89}.
An analytic function $F(z)= \sum_{n=1}^{\infty} c_n \frac{z^n}{n!}$ is an {\em $E^{*}$-function} if it has  the following properties:
\begin{enumerate}
\item
$F(z)$ is an $E$-function.
\item
There is some constant $C>0$ such that the maximum of the absolute values of the algebraic
conjugates of $c_n$
 is bounded by $O(C^n)$ as $n \to \infty$.
\item
There is some constant $C' >0$
 such that  there is a sequence of integers $q_0, q_1, q_2 ...$ such that each
$q_n c_n$ is an algebraic integer, and with 
$q_n = O((C')^{n})$ 
as $n \to \infty$.
\end{enumerate}
 This definition  narrows the growth rate conditions on
the coefficients compared to Siegel's conditions.
 The set  of  all  $E^{*}$-functions  forms a ring ${\bE}^{*}$ closed under differentiation
which is also an algebra over $\bar{\QQ}$.
 Certainly  $\bE^{*} \subseteq \bE$, and it is conjectured that the two rings are equal
 (\cite[p. 407]{Sh89}). Results of   Andr\'{e} (\cite{And00a}, \cite{And00b})   apply to the class $\bE^{\ast}$.

Siegel \cite[Teil I]{Sie29} originally applied his method to prove transcendence of algebraic values 
of the Bessel function $J_0(\alpha)$
for algebraic $\alpha$ where $J_0(\alpha) \ne 0$.
In formalizing  his method in 1949,  
Siegel \cite{Sie49} 
first converted an $n$-th order differential operator to a linear system of
first order   linear differential equations
$$
\frac{dy_i}{dx} = \sum_{j=1}^n R_{ij}(x) y_j, ~~~1 \le i \le n,
$$
where the  functions $R_{i,j} (x)$ are rational functions. In order for his proofs to
apply, he required  this system to   satisfy an extra condition which he called {\em normal.}
 Siegel 
was able to verify normality  in  a  few cases,
including the exponential function and  certain Bessel functions.
A great  breakthrough of Shidlovskii \cite{Shi59} in 1959 permitted the method to
prove transcendence of solution values for all cases where the obvious
necessary conditions are satisfied.  
The Siegel-Shidlovskii theorem states  that if the $E_i(x)$ are algebraically
independent functions over the rational function field $K(x)$, then at any non-zero algebraic
number not a pole of any function $R_{ij}(x)$, the values $E_j(\alpha)$ are algebraically
independent.  For treatments of these results see  Shidlovskii \cite{Sh89} and 
Feldman and Nesterenko \cite[Chap. 5]{FN98}.

More recent progress on $E$-functions  includes  a study in 1988 of the Siegel normality condition by
Beukers, Brownawell and Heckman \cite{BBH88} which  obtains effective measures
of algebraic independence in some cases.
In 2000  Andr\'{e} \cite[Theorem 4.3]{And00a}
(see also \cite{And03}) 
 established that for each $E^{*}$-function  the minimal order linear differential equation
with $\CC(z)$ coefficients that it satisfies has singular points only at $z= 0$ and $z=\infty$; that is, 
it has a basis of $n$ independent  holomorphic solutions at all other points. Furthermore $z=0$ is 
necessarily a regular
singular point and $\infty$ is an irregular singular point of a restricted type. 
This result goes a long way towards explaining  why all the known
 $E^{*}$-functions are
of hypergeometric type. 
Andr\'{e} \cite{And00b}
 applied this characterization to give a new proof of the Siegel-Shidlovskii 
theorem. Finally we remark that although transcendence results are usually associated to $E$-functions
and irrationality results to $G$-functions, in 1996 Andr\'{e} \cite{And96} obtained transcendence
results for certain $G$-function values. 

 Returning to Euler's constant, in 1968 Mahler \cite{Mah68} applied Shidlovskii's methods
to certain Bessel functions. For example
$$J_0(z) = \sum_{n=0}^{\infty} \frac{(-1)^n}{ (n!)^2 } (\frac{z}{2})^{2n}$$
  is a Bessel function of the first kind, and
\begin{equation}
\label{bessel3}
Y_0(z) = \frac{2}{\pi} \Big(\log (\frac{z}{2} )+\gamma\Big) J_0(z) 
+ \frac{2}{\pi}\Big( \sum_{n=1}^{\infty} (-1)^{n-1} \frac{H_n}{(n!)^2} (\frac{z^2}{4})^n\Big)
\end{equation}
is a Bessel function of the second kind, 
see \cite[(9.1.13)]{AS}. 
 Here $J_0(z)$ is an $E^{*}$-function, but $Y_0(z)$, which contains Euler's
 constant in its expansion \eqn{bessel3},   is not, because it has a logarithmic
 singularity at $x=0$.  However the modified function
  $$
  \tilde{Y}_0(z) := \frac{\pi}{2} Y_0(z) -  \Big(\log (\frac{z}{2}) + \gamma\Big) J_0(z)= 
  \sum_{n=1}^{\infty} (-1)^{n-1} \frac{H_n}{(n!)^2} (\frac{z^2}{4})^n
  $$
is an $E^{*}$-function, and Mahler uses this fact. 
He noted the following special case of his general results.


\begin{theorem}\label{th100a} {\em (Mahler 1968)}
The number
\begin{equation*}
\frac{\pi}{2} \frac{Y_0(2)}{J_0(2)} - \gamma
\end{equation*}
is transcendental.
\end{theorem}

We have
$J_0(2) = \sum_{n=0}^{\infty} \frac{(-1)^{n}}{ (n!)^2}$ and 
$\frac{\pi}{2} Y_0(2) =\gamma J_0(2)+ \sum_{n=1}^{\infty} (-1)^{n-1}
\frac{H_n}{(n!)^2},$
from which it follows that  Mahler's transcendental number is
\begin{equation*}
\frac{\pi}{2} \frac{Y_0(2)}{J_0(2)} - \gamma =
\frac{ \sum_{n=1}^{\infty} \frac{(-1)^{n-1} H_n}{(n!)^2} }{\sum_{n=0}^{\infty}\frac{(-1)^n}{ (n!)^2}}.
\end{equation*}

More recent transcendence results apply to Euler's constant along
with other constants.
In 2012 Rivoal \cite{Riv10c} proved a result  which, as a special case,  implies 
the transcendence of at least one of the Euler constant $\gamma$
and the Euler-Gompertz constant $\EG$. It uses the Shidlovskii method
and improves on an observation of Aptekarev \cite{Apt09}
 that at least one of $\gamma$ and $\EG$ must be irrational. 
 This transcendence of at least one of $\gamma$ and $\delta$ was also established about
 the same time by Kh. and T. Hessami Pilehrood \cite[Corollary 3]{HP13}.


\begin{theorem}\label{th102} {\em (Rivoal 2012)}
Euler's constant $\gamma$ and the Euler-Gompertz  constant
$$
\EG := \int_{0}^{\infty} \frac{e^{-w}}{1+w} \, dw =  \int_{0}^1 \frac{dv}{1- \log v}
$$
together have the following properties:\smallskip

(1) (Simultaneous Diophantine Approximation)
One cannot approximate the pair $(\gamma, \EG)$ very well 
with rationals $(\frac{p}{q}, \frac{r}{q})$ having the same denominator. 
For each $\epsilon>0$ there is a constant $C(\epsilon)>0$ such that
for all integers $p, q, r$ with $q \ne0$, there holds
\begin{equation*}
|\gamma - \frac{p}{q}| + |\EG- \frac{r}{q}| \ge \frac{ C(\epsilon)} {H^{3+ \epsilon}},
\end{equation*}
where $H= \max(|p|, |q|, |r|)$.

(2) (Transcendence) The transcendence degree of the field $\QQ(e, \gamma, \delta)$ 
generated by $e , \gamma$ and $\delta$  over the rational numbers $\QQ$  is at least two.
\end{theorem}

\begin{proof}
This follows from a much more general 
result of Rivoal's  \cite[Theorem 1]{Riv10c},
which  concerns values of the  function 
$$
\sG_{\alpha}(z) :=  z^{-\alpha} \int_{0}^{\infty} (t+z)^{\alpha -1} e^{-t} dt.
$$
Here we specialize to the case  $\alpha=0$,  and on taking $z=1$ we have   
\begin{equation*}
\sG_0(1) = \int_{0}^{\infty}  \frac{e^{-t}}{t+1} dt = \EG,
\end{equation*} 
using Hardy's integral \eqn{560aa}. The function $\sG_0(z)$
satisfies a linear differential equation, but is not an $E$-function in the sense of transcendence theory.
 Rivoal makes use of the identity
\begin{equation}\label{922}
\gamma + \log z =  -e^{-z} \sG_0(z)-\sE(-z),
\end{equation}
valid for $z \in \CC \smallsetminus \RR_{\le 0}$, 
where the entire function 
$$
\sE(z) := \sum_{n=1}^{\infty} \frac{z^n}{n \cdot n! },
$$
 is an $E^{*}$-function.

Result (i) follows from  the assertion of  Theorem 1 (ii) of Rivoal \cite{Riv10c},
specialized to  
parameter values $\alpha=0$ and $z=1$. This result uses explicit Hermite-Pad\'{e}
approximants to  the  $E^{*}$-functions  $1, e^z, \mathcal{E}_{\alpha}(z)$,
where
$\mathcal{E}_{\alpha}(z) = \sum_{m=0}^{\infty} \frac{z^m}{m!(m+\alpha +1)}$,
with $\alpha \ne -1, -2, -3, ...$. Here $\mathcal{E}(z)$ corresponds to $\alpha=-1$
with the divergent constant term dropped from the power series expansion.

Result  (ii) follows from the assertion of Theorem 2 (ii) of Rivoal \cite{Riv10c},
which asserts for (complex) algebraic numbers $z \not\in (-\infty, 0]$ that
the field generated over $\QQ$ by the three numbers $e^z, \,  \gamma+ \log z, \,  \sG_0(z) $
has transcendence degree at least two. It makes use of \eqn{922},
which can be rewritten as the integral identity
\begin{equation*}
-\gamma= \log z  + z \int_{0}^1 e^{-tz} \log t \,  dt
+ e^{-z} \int_{0}^{\infty} \frac{e^{-t}}{t+z} \, dt
\end{equation*}
It   uses the specialization $z=1$, which essentially gives  Hardy's identity \eqn{560b}
divided by $e$, which we may rewrite as
\begin{equation*}
-\sE(-1)= \gamma + \frac{\delta}{e}.
\end{equation*}
The result then follows using  Shidlovskii's Second Fundamental
Theorem \cite[p. 123]{Sh89}. \end{proof}
 
 Rivoal \cite{Riv10c} remarks that the transcendence result Theorem \ref{th102} (2) is implicit
 in the approach that Mahler \cite{Mah68}  formulated in 1968.
At the  end of his paper Mahler \cite[p. 173]{Mah68} states without details that 
 his method  extends to show 
that for rational $\nu_0>0$  and nonzero algebraic numbers $\alpha$,  the integrals 
$$
I_k(\nu_0, \alpha) := \int_{0}^{1} x^{\nu_0 -1} (\log x)^k e^{-\alpha x} dx, ~~~k=0, 1,2, ...
$$ 
are algebraically independent transcendental numbers.
Choosing $\nu_0=1, \alpha=1$
and $k=0,1$, we have
 $$I_0(1, 1) :=  \int_{0}^{1} e^{-x} dx= 1- \frac{1}{e},$$
  and, using \eqn{559f}, 
$$ I_1(1,1)= \int_{0}^1 (\log x)\, e^{-x} dx = \int_{0}^{\infty}(\log x)\, e^{-x} dx- \frac{1}{e} \int_0^{\infty} \log (x+1) e^{-x} dx= 
-\gamma -\frac{\delta}{e}.
$$
Mahler's statement asserts that these  are algebraically independent transcendental numbers. 
This assertion implies that $\QQ (\gamma, \delta, e)$
has transcendence degree at least two over $\QQ$.

We next present  recent results which concern the transcendence of 
collections of generalized
Euler constants discussed in Section \ref{sec37}. These results  
are applications of Baker's results on linear forms in logarithms, and allow 
the possibility of one exceptional algebraic value.
Recall 
that the {\em Euler-Lehmer constants}, for $0 \le  h < k$ studied by Lehmer \cite{Leh75},  are defined by
\[
\gamma(h, k)  := \lim_{x \to \infty} \Big(\sum_{\substack{0 \le n < x\\ n \equiv h~ (\bmod k)}} \frac{1}{n} - \frac{\log x}{k} \Big).
\]
In 2010 M. R. Murty and N. Saradha \cite[Theorem 1]{MS10} obtained the following result. 


\begin{theorem}\label{th3162} {\em (Murty and Saradha 2010)}
In the infinite list of Euler-Lehmer constants 
\[ 
\{ \gamma(h, k) :   1 \le h < k,~~  \mbox{for all}~~  k \ge 2\}, 
 \] 
at most one value is an algebraic number. If 
Euler's constant $\gamma$ is an algebraic number, then only
the number $\gamma(2,4) = \frac{1}{4} \gamma$ in the above
list is algebraic.
\end{theorem}

Their proof uses the fact that a large set of  linear combinations of $\gamma(h,k)$
are equal to logarithms of integers. 
 Then Baker's bounds for  linear forms in logarithms of
algebraic numbers are applied  to derive a contradiction from 
the assumption that the set contains  at least  two algebraic numbers.
   An immediate consequence of Theorem \ref{th3162} is the fact that over all rational numbers
$0< x\le 1$ at least one of $\Gamma(x), \Gamma'(x)$ is transcendental, with
at most one possible exceptional $x$. (\cite[Corollary]{MS10}), see also
Murty and Saradha \cite{MS07}. 

In a related direction Murty and Zaytseva \cite[Theorem 4]{MZ13}  obtained a 
corresponding transcendence  result for the generalized
Euler constants $\gamma(\Omega)$ associated to a finite set $\Omega$ of primes
that were introduced by Diamond and Ford \cite{DF08}.
Recall from Section \ref{sec37} that these constants are obtained by setting
$\zeta_{\Omega}(s) = \prod_{p  \in \Omega} \Big(1- p^{-s} \Big) \zeta(s)$ and 
defining $\gamma(\Omega)$ as the constant term in 
its Laurent expansion around $s=1$, writing
\[
\zeta_{\Omega} (s) = \frac{D(\Omega)}{s-1} + \gamma(\Omega) + \sum_{n=1}^{\infty} \gamma_n(\Omega) (s-1)^n.
\]
In particular $\gamma(\emptyset) = \gamma$. 


\begin{theorem}\label{th3163} {\em (Murty and Zaytseva 2013)}
In the infinite list of Diamond-Ford 
generalized Euler constants $ \gamma(\Omega)$ where  $ \Omega$ runs over all
finite subsets of primes (including the empty set),   all numbers are transcendental with at
most one exception.
\end{theorem}

The constants $\gamma(\Omega)$ can be expressed as finite  linear
combinations of  Euler-Lehmer numbers, cf. \eqn{604aa},
so this result is closely related to the previous theorem.
This result is also proved by contradiction, using Baker's results on linear forms in logarithms. 
If $\gamma$ were algebraic then $\gamma(\emptyset)$
would be the unique exceptional value.
However one   expects there to be no algebraic value in the list.

\renewcommand{\theequation}{\arabic{section}.\arabic{equation}} 
%
%
%
%
\section{Concluding remarks}\label{sec4}
\setcounter{equation}{0}

Euler's constant appears {\em sui generis}. Despite
three centuries of effort, it seems unrelated to other
constants in any simple way. 

A first consequence arising from this  is:  {\em  the  appearance
of Euler's constant in ostensibly unrelated fields of mathematics may
signal some hidden relationship between these fields.}
A striking example   is the analogy
between results concerning  the structure of factorizations of
random integers and cycle structure of random permutations.
These results are given  for  restricted
factorizations 
in Sections \ref{sec34} and \ref{sec35} and for
restricted cycle structure of permutations in Sections  \ref{sec38} and \ref{sec38a},
respectively.
Results in these two fields were initially proved separately, and were later 
 unified  in limiting  probability models
developed by Arratia, Barbour and Tavar\'{e} \cite{ABT97}.
These models account not only for the main parts of
the probability distribution but also for  
 tail distributions,  where there are large deviations results in
which Euler's constant appears.
 On a technical level an explanation 
for the parallel results in the two subjects  was offered by 
Panario and Richmond \cite{PR01}
at the level of  asymptotic analysis of generating functions of a particular form.
This  research development  is an example of  one of Euler's 
research methods detailed in Section \ref{sec26}, that of  looking for a general scheme
uniting two special problems.

In a second direction, the study of Euler's constant and
multiple zeta values can now  be   viewed in a larger
context of rings of periods connected with algebraic integrals. 
The sui generis nature of $\gamma$
is  sharply formulated in the conjecture that it is  not  a Kontsevich-Zagier period.
On the other hand, it is  known to belong to  the larger ring
of  exponential periods $\sEP$,  as described in Section \ref{sec310}. 
This research development illustrates another  of Euler's research methods,
to keep searching for a fuller understanding of any problem already solved.

In a third direction, one may ask 
for an explanation of  the appearance of
Euler's constant in various sieve method
formulas. Here
Friedlander and Iwaniec \cite[p. 239]{FI10} comment 
 that these appearances  
stem  from two distinct  sources. The first is through
Mertens's product  theorem \ref{th51}. The second is via differential-difference
equations of the types treated in Sections \ref{sec34} and  \ref{sec35}.

In a fourth direction, there remain mysteries about the appearance of Euler's constant in
number theory in the gamma and zeta function. \smallskip

{\bf Observation.} {\em Euler's constant $\gamma$ appears  in the 
Laurent expansion of the Riemann zeta function
 at $s=1$; here the zeta function
 is a function encoding the properties of the finite primes. Euler's constant also appears
in the Taylor expansion  of the gamma function at $s=1$; here  the gamma function
encodes   properties 
of the (so-called) ``infinite prime" or "real prime" (Archimedean place).}\smallskip

The analogy between the   "real prime" and the finite primes traces back to Ostrowski's
work in the period 1913-1917 which
classified all absolute values on $\QQ$ (Ostrowski's Theorem).
Such analogies have been pursued for a
century;  a comprehensive treatment of such analogies 
is given in  Haran \cite{Har01}.

A mild mystery is to give a conceptual explanation for
this occurrence of Euler's constant
in two apparently different contexts, one associated to the finite primes, and the other associated to the 
real  prime (Archimedean place). 
There is a known connection between these two contexts, given
by the product formula for rational numbers 
over all places, finite and infinite. It says
\[
\prod_{p \le \infty} |r|_{p} = 1, \qquad   r \in \QQ.
\]
in which $|\cdot|_p$ denotes the $p$-adic valuation, normalized with $|p|_p = \frac{1}{p}$
and $|x| _{\infty}$ being the real absolute value. Does this relation explain the
common appearance?  
On  closer examination  there appears to be  a mismatch   between these
two occurrences, visible in 
 the functional equation of the Riemann zeta function.  The functional equation
 can be written 
 \begin{equation*}
\hat{\zeta}(s) = \hat{\zeta}(1-s),
\end{equation*}
in which appears  the {\em  completed zeta function}
\begin{equation*}
\hat{\zeta}(s) := \pi^{-\frac{s}{2}} \Gamma(\frac{s}{2}) \zeta(s)= \pi^{-\frac{s}{2}} \Gamma(\frac{s}{2})
\prod_{p}\left( 1- \frac{1}{p^s} \right)^{-1}.
\end{equation*}
The completed zeta function is given 
as an Euler product over all ``primes", both the finite primes and  the 
Archimedean Euler factor $\pi^{-\frac{s}{2}} \Gamma(\frac{s}{2})$. 
In this formula evaluation at the  
point $s=1$ for the Riemann zeta function corresponds to evaluation of  the
gamma function  at the shifted point  $s' = \frac{s}{2}= \frac{1}{2}$. 
The mismatch is that for the gamma function, Euler's constant most 
naturally appears in connection with its derivative at 
the point $s'=1$,  rather than at the point $s'= \frac{1}{2}, $ cf \eqn{201}. 
 Is there a conceptual explanation of this mismatch?  Theorem ~\ref{th30}
 associates Euler's constant with the digamma function at $s=\frac{1}{2}$, but
in that case an extra factor of $\log 2$ appears.
Perhaps  this mismatch can be resolved by 
considering an asymmetric  functional equation for the  function  $\Gamma(s) \zeta(s)$. 
Here one has, valid for $\text{Re}(s)>1$, the integral formula
\[
 \Gamma(s) \zeta(s)= \int_{0}^{\infty} \frac{1}{e^x -1} x^{s-1} dx.
\] 
One of Riemann's 
proofs of the  functional equation for $\zeta(s)$
is based on this integral.

Another  mystery concerns  whether there may exist an interesting
arithmetical  dynamical interpretation of Euler's constant,
different from that given in Section \ref{sec39}. 
The   logarithmic derivative of the archimedean factor $\pi^{-\frac{s}{2}} \Gamma(\frac{s}{2})$ at
the value $s=2$ (that is, $s' =1$) produces the value  $-g/2$ where 
\begin{equation*}
g:= \gamma + \log \pi.
\end{equation*}
The constant $g$  appears elsewhere in the
``explicit formulas" of prime number theory, as a contribution at the (special) point
$x=1$ in the Fourier test  function space, cf.  \cite[p. 280]{BL99}.  
Perhaps one should search for a dynamical
system whose entropy is the constant $-g:= -\gamma - \log \pi \approx -1.721945$.

In a fifth  direction, Euler's constant can be viewed as a value
emerging from regularizing the zeta function at $s=1$, 
identified with the constant term in its Laurent series expansion at $s=1$.
There are a number of different proposals
for ``renormalizing"  multiple zeta values, including the value of  ``$\zeta(1)$",
with the intent to preserve other algebraic structures.
This is an active area
of current research.  Various alternatives are discussed in 
Ihara, Kaneko and Zagier \cite{IKZ06},
Guo and Zhang \cite{GZ08} , \cite{GZ08b},  Guo, Paycha, Xie and Zhang \cite{GPXZ09}
and Manchon and Paycha \cite{MP10}. 
Some of these authors  (for example \cite{IKZ06}) 
prefer to  assign  the ``renormalized"  value $0$
to $``\zeta(1)"$, to be obtained  by further subtracting off Euler's constant.

This direction generalizes to algebraic number fields.
For background we refer to  Neukirch \cite[Chap. VII. Section 5]{Neu99}.
The Dedekind zeta function $\zeta_K(s)$ of a number field $K$ is given by 
$$
\zeta_K(s) = \sum_{A} \frac{1}{N_{K/\QQ}(A)^{s}}
$$
where $A$ runs over all the ideals in the ring of integers $O_K$ of 
the number field, and $N_{K/\QQ}$ is the norm function.
 This function analytically continues to a meromorphic function on $\CC$
which  has a simple pole at $s=1$, with Laurent expansion
\begin{equation}
\label{dedz}
\zeta_K(s) = \frac{ \alpha_{-1}(K)}{s-1} + \alpha_0(K) + \sum_{j=1}^{\infty} \alpha_j(K)(s-1)^j.
\end{equation}
The  term  $\alpha_{-1}(K)$ encodes arithmetic information about the field $K$,
given in the {\em analytic class number formula}
$$
 Res_{s=1}(\zeta_K(s)) := \alpha_{-1}(K) = \frac{h_KR_K}{e^{g_K}},
$$
in which $h_K$ is the (narrow) class number of $K$, $R_K$ is the regulator of $K$
(the log-covolume of the unit group of $K$) and $g_K$ is the genus of $K$, given by
$$
g_K := \log \frac{w_K |d_K|^{1/2}}{2^{r_1} (2\pi)^{r_2}},
$$
in which $w_K$ is the number of roots of unity in $K$, $d_K$ is the absolute discriminant of $K$,
$r_1$ and $r_2$ denote the number of real (resp. complex) Galois conjugate fields to $K$, inside $\bar{\QQ}$.
The constant term $\alpha_0(K)$ 
 is then an  analogue of Euler's constant for the number field $K$, since $\alpha_0(\QQ) = \gamma$.
 A refined notion studies
the constant term 
of a zeta  function attached to an individual ideal class $\sC$ of  a number field $K$. Such a partial zeta function
 has the form
$$
\zeta(\sC,s) := \sum _{A \in \sC}  \frac{1}{N_{K/\QQ}( A)^{s}}
$$ 
where the sum runs over all integral ideals $A$   
of  the number field belonging to the  ideal class $\sC$.
The Dedekind zeta function is the sum of  ideal class zeta functions
over all $h_K$ ideal classes.  These partial zeta functions $\zeta(\sC, s)$ extend to meromorphic functions of $s$ on $\CC$
with only singularity a simple pole at $s=1$ having a residue term 
 that is  independent of $\sC$, i.e. $\alpha_{-1}(\sC, K) =\frac{1}{h_K} \alpha_{-1}(K)$
(\cite[Theorem VII (5.9)]{Neu99}).

 For imaginary
quadratic fields the constant term $a_0(\sC, K)$  in the Laurent expansion
at $s=1$  of an ideal class zeta function 
can be expressed in terms of a modular form using the first Kronecker limit formula;
cf. Kronecker \cite{Kron}.
 This limit formula concerns 
the real-analytic Eisenstein series, for $\tau=x+iy$ with $y >0$ and $s \in \CC$
$$
E(\tau, s) = \sum_{(m,n) \ne (0,0)} \frac{y^s}{|m \tau + n|^{2s}},
$$
which for fixed $\tau$  analytically 
continues to a meromorphic function of $s \in \CC$ and has a simple pole at $s=1$ with
residue $\pi$ and Laurent expansion
$$
E(\tau, s) = \frac{\pi}{s-1} + c_0(\tau) + \sum_{j=1}^{\infty} c_n(\tau) (s-1)^n.
$$
The first Kronecker limit formula expresses the constant term $c_0(\tau)$ as
\begin{equation*}
c_0(\tau)= 2 \pi (\gamma - \log 2) - \log (\sqrt{y} |\eta(\tau)|^2),
\end{equation*}
in which $\eta(\tau)$ is the Dedekind eta function, a modular form given by
$$
\eta(\tau) = q^{1/24} \prod_{n=1}^{\infty} (1- q^n),~~~~~~\mbox{with} ~~~q= e^{2 \pi i \tau},
$$
see \cite[Sect. 2]{Zag75}.
In   1923 Herglotz \cite{Her23ab}
found an analogous   formula for the constant
term in the Laurent series expansion at $s=1$ of $\zeta(\sC, s)$ for 
an ideal class $\sC$ of a real quadratic field, which involved 
Dedekind sums and an auxiliary transcendental function.  In 1975 by Zagier \cite{Zag75} obtained
another formula for this case which is expressed using   continued fractions and an auxiliary function.

In 2006 Ihara \cite{Ih06}
studied another  quantity extracted from  the Laurent expansion \eqn{dedz} of the Dedekind zeta function
of a number field at $s=1$, given by  
\begin{equation*}
\gamma(K) := \frac{\alpha_0(K)}{\alpha_{-1}(K)},
\end{equation*}
which he termed the  {\em Euler-Kronecker constant} of the
number field $K$.
Here $\gamma(K)$ represents a scaled form of the constant term, and it generalizes  Euler's constant
since $\gamma(\QQ) = \gamma.$
Motivated by the function field case, Ihara hopes that $\gamma(K)$ encodes interesting
information of a different type on the field $K$, especially as $K$ varies. 
Under the assumption of
the Generalized Riemann Hypothesis, Ihara obtained  upper and lower bounds
$$
- c_1 \log D_K  \le \gamma(K) \le c_2 \log\log D_K,
$$
where $D_K$ is the (absolute value of the)  discriminant of $K$ and $c_1, c_2>0$ are absolute constants.
Further unconditional upper bounds on $\gamma(K)$ were obtained by Murty \cite{Mur11}
and lower bounds were given by Badzyan \cite{Bad10}.

Finally,  ongoing attempts to place Euler's constant in a larger
context have continued in  many directions. These include studies of new special functions
involving Euler's constant, such as  the {\em generalized Euler-constant function} $\gamma(z)$ proposed
by Sondow and Hadjicostas \cite{SH07}. 
Two other striking formulas, found in 2005 by  Sondow \cite{Son05}, are the double integrals
$$
\int_{0}^1 \int_{0}^1 \frac{1-x}{(1-xy)(-\log xy)} dx dy = \gamma
$$
and
$$
\int_{0}^1 \int_{0}^1 \frac{1-x}{(1+xy)(-\log xy)} dx dy = \log \frac{4}{\pi},
$$
 Guillera and Sondow \cite{GS08} then found a common integral generalization 
 with a complex parameter that
 unifies these formulas with 
  double integrals of Beukers \cite{Beu79} for $\zeta(2)$ and $\zeta(3)$,
  see  \eqn{833a}.

Many other generalizations and 
variants of Euler's constant  have
been proposed.
  These include a $p$-adic analogue of Euler's constant introduced by 
J. Diamond \cite{Di77} in 1977 (see also Koblitz \cite{Kob78} and Cohen \cite[Sect. 11.5]{Coh07b}) and
 $q$-analogues of Euler's constant formulated by Kurokawa
and Wakayama \cite{KW04} in 2004.
A topic of perennial interest is the search for  new expansions and representations 
for Euler's constant. For example  1917 work of 
Ramanujan \cite{Ram17} was extended in 1994 by  Brent \cite{Bre94}.
Some new expansions  involve  accelerated convergence of known expansions, e.g. 
Elsner \cite{Els95}, K. and T. Hessami Pilehrood \cite{HP07},  and
Pr\'{e}vost \cite{Pr08}.  Recent developments of the Dickman function led
to the study of new constants involving Euler's constant and polylogarithms,
cf. Broadhurst \cite{Bro10} and Soundararajan \cite{Sou10}.

All these studies, investigations  of irrationality,
connections with the Riemann hypothesis, possible interpretation as a (non)-period,
and  formulating generalizations of Euler's constant,  may together
lead to  a better understanding of the place of Euler's constant $\gamma$ in mathematics.\smallskip

\paragraph{\bf Acknowledgments.} This paper started with  a talk
  on Euler's constant given  in 2009 in the ``What is...?" seminar run by
Li-zhen Ji at the University of Michigan.
 I  am grateful to  Jordan Bell for providing 
 Latin translations of Euler's work in Sections 2.1--2.3,  and for 
 pointing out relevant Euler papers E55, E125, E629,
 and thank 
 David Pengelley for  corrections and critical
comments on Euler's work in Section 2. 
  I thank Sylvie Paycha for comments on dimensional regularization in Section 3.2;
  Doug Hensley, Pieter Moree, K. Soundararajan and 
Gerald Tenenbaum for  remarks on  Dickman's function  in  Section 3.5;
 Andrew Granville 
for extensive comments on the anatomy of integers (\cite{Gra11}), affecting Sections 3.5, 3.6 and 3.9-3.11;
Tony Bloch for pointing out work of 
Cohen and Newman described in Section 3.13; 
Maxim Kontsevich for comments on rings of periods  in Section 3.14;  and
Stephane Fischler, Tanguy Rivoal and Wadim Zudilin for comments on Sections 3.15 and 3.16. 
I thank Daniel Fiorilli, Yusheng Luo  and Rob Rhoades for important corrections.
I  am especially grateful to  Juan Arias de Reyna  for  detailed checking, detecting
many misprints and errors,  and supplying  additional references. 
 Finally I thank the reviewers  for many  historical and mathematical suggestions
 and corrections.

%
%
%
%



\end{document}